\documentclass{amsart}
\usepackage{amssymb}
\usepackage{amsmath}
\usepackage{verbatim}
\usepackage{url}
\usepackage{graphicx}
\usepackage{color}
\usepackage{enumerate}

\newcommand{\eps}{\varepsilon}

\newcommand{\opnm}{\operatorname}

\newcommand{\BigO}{\mathcal{O}}

\usepackage{mathrsfs}

\newtheorem{theorem}{Theorem}%[section]
\newtheorem{proposition}[theorem]{Proposition}
\newtheorem{corollary}[theorem]{Corollary}
\newtheorem{lemma}[theorem]{Lemma}
\theoremstyle{definition}

\theoremstyle{remark}
\newtheorem{remark}[theorem]{Remark}
\newtheorem*{remark*}{Remark}
\newtheorem{example}[theorem]{Example}
\newtheorem*{acknowledgements}{Acknowledgements}

\numberwithin{equation}{section}
\numberwithin{theorem}{section}
\numberwithin{figure}{section}

\newcommand{\Hm}[1]{\leavevmode{\marginpar{\tiny%
$\hbox to 0mm{\hspace*{-0.5mm}$\leftarrow$\hss}%
\vcenter{\vrule depth 0.1mm height 0.1mm width \the\marginparwidth}%
\hbox to
0mm{\hss$\rightarrow$\hspace*{-0.5mm}}$\\\relax\raggedright #1}}}

\title[Weak solution operators for evolution equations]{On weak and strong solution operators for evolution equations coming from quadratic operators}

\author{Alexandru Aleman}

\email{aleman@maths.lth.se}

\address{Lund University, Mathematics, Faculty of Science, P.O. Box 118, S-221 00 Lund, Sweden}

\author{Joe Viola}

\email{Joseph.Viola@univ-nantes.fr}

\address{Laboratoire de Math\'ematiques J. Leray, UMR 6629 du CNRS, Universit\'e de Nantes, 2, rue de la
Houssini\`ere, 44322 Nantes Cedex 03, France}

\begin{document}

\begin{abstract}
	We identify, through a change of variables, solution operators for evolution equations with generators given by certain simple first-order differential operators acting on Fock spaces.  This analysis applies, through unitary equivalence, to a broad class of supersymmetric quadratic multiplication-differentiation operators acting on $L^2(\Bbb{R}^n)$ which includes the elliptic and weakly elliptic quadratic operators.  We demonstrate a variety of sharp results on boundedness, decay, and return to equilibrium for these solution operators, connecting the short-time behavior with the range of the symbol and the long-time behavior with the eigenvalues of their generators. This is particularly striking when it allows for the definition of solution operators which are compact and regularizing for large times for certain operators whose spectrum is the entire complex plane.
\end{abstract}

\maketitle

\setcounter{tocdepth}{1}
\tableofcontents

\section{Introduction}

\subsection{Background and summary of results}\label{subsec.intro}

Evolution equations of the form
\begin{equation}\label{eq.evolution.general}
	\left\{\begin{array}{l} \partial_t u + Pu = 0,\\ u(0,x) = u_0\end{array}\right.
\end{equation}
appear throughout mathematical physics.  A fundamental example comes from the harmonic oscillator
\begin{equation}\label{eq.def.h.o}
	Q_0u = \frac{1}{2}(-\Delta + |x|^2 - n) u,
\end{equation}
chosen here to satisfy $\opnm{Spec}Q_0 = \Bbb{N}$. Solving the evolution problem for $Q_0$, as well as the Schr\"odinger evolution problem for $iQ_0$, through the spectral decomposition of $Q_0$ as a self-adjoint operator on $L^2(\Bbb{R}^n)$ is one of the most important model systems in quantum mechanics.  The analysis of the harmonic oscillator through its decomposition into creation-annihilation operators is also one of the primary motivations behind the study of Fock spaces; see for instance \cite[Ch.~1]{FoBook} or \cite{Bargmann_1961}.

When studying non-selfadjoint operators, approximations which are quadratic in $(x,-i\partial_x)$ retain significant power as microlocal models for more general operators.  The spectral theory of these operators under an ellipticity assumption was resolved in \cite{Sj1974}, \cite{BdM1974}.  The semigroups generated by quadratic operators under a definite or semidefinite assumption have been extensively studied in many works including \cite{Ho1995}, \cite{Boulton2002}, \cite{PS2008a}, \cite{HiPS2009}, \cite{OtPaPS2012}.  Because of applications including stochastic partial differential equations, there has been recent interest in situations where positivity only appears after averaging, as discussed in \cite{HeSjSt2005}, \cite{HeNiBook}, \cite{VillaniBook} among many others.

It has been known for some time that, in the non-selfadjoint case, relaxing the semidefiniteness assumption is catastrophic for the definition of the semigroup from the point of view of the numerical range. From works such as \cite{Da1999}, \cite{PS2008b}, and \cite{DeSjZw2004}, we can find broad classes of operators $P$ acting on $L^2(\Bbb{R}^n)$ for which
\begin{equation}\label{eq.pseudomodes.general}
	Pu_k = z_ku_k + \BigO(e^{-|z_k|/C})
\end{equation}
for sequences $\{z_k\}_{k \in \Bbb{N}}$ of complex numbers with $\Re z_k \to -\infty$ and pseudomodes $u_k \in C_0^\infty(\Bbb{R}^n)$ which are normalized in $L^2(\Bbb{R}^n)$.  These pseudomodes show that the resolvent norm at $z_k$ explodes and that the numerical range of $P$ extends indefinitely into the left half-plane, so the standard methods of constructing a semigroup such as the Hille-Yosida theorem fail. This situation can easily arise even when, from the spectral point of view, $P$ is well-behaved, having a compact resolvent and spectrum contained in a sector 
\[
	\opnm{Spec} P \subset \{|\Im \lambda| \leq C \Re \lambda \}
\]
for some $C > 0$.

In this work we study evolution equations with quadratic generators which may be written as
\begin{equation}\label{eq.supersymmetric}
	Q = B(D_x - A_- x) \cdot (D_x - A_+x), \quad D_{x_j} = -i\partial_{x_j}
\end{equation}
for matrices $B, A_+,$ and $A_-$ with $A_{\pm}$ symmetric, $A_\pm^\top = A_\pm$, and having positive and negative definite imaginary parts, $\pm \Im A_\pm > 0$. For example, the harmonic oscillator $Q_0$ in \eqref{eq.def.h.o} may be written with $A_+ = \overline{A_-} = i$ and $B = 1/2$.

This is a supersymmetric structure in the sense of \cite[Def.~1.1]{HeHiSj2014}, in that
\[
	Q = Bd_{\varphi_-}^* d_{\varphi_+}
\]
with $d_{\varphi_\pm} = e^{\varphi_\pm}D_x e^{-\varphi_\pm}$ and
\[
	\varphi_+(x) = \frac{i}{2}A_+ x\cdot x, \quad \varphi_-(x) = \frac{i}{2}\overline{A_-}x\cdot x.
\]
This resembles \cite[Eq.\ (11), (12)]{Wi1982} but allows the operator to be non-selfadjoint in two ways: the matrix $B$ may not be self-adjoint, and the functions $\varphi_+$ and $\varphi_-$ may be different. For any operator
\[
	q(x,D_x) = \sum_{|\alpha + \beta| = 2} q_{\alpha \beta} x^\alpha D_x^\beta, \quad q_{\alpha\beta} \in \Bbb{C},
\]
we have in Proposition \ref{prop.which.q} below necessary and sufficient conditions for existence of a decomposition \eqref{eq.supersymmetric}, up to an additive constant. Such a decomposition is known to exist when the symbol $q(x,\xi)$ is elliptic
\[
	\Re q(x,\xi) \geq \frac{1}{C}|(x,\xi)|^2
\]
or when $\Re q(x,\xi) \geq 0$ and, in addition, the zero set of the real part excepting the origin, $(\Re q)^{-1}(\{0\}) \backslash \{0\}$, contains no integral curve of the Hamilton vector field $H_{\Im q} = (\partial_\xi \Im q, -\partial_x \Im q)$. Following \cite{HiPS2009}, this latter condition is equivalent to insisting that
\begin{equation}\label{eq.intro.ell.part}
	\sum_{j=0}^{k_0} \Re q(H_{\Im q}^j(x,\xi)) \geq \frac{1}{C}|(x,\xi)|^2
\end{equation}
for some $0 \leq k_0 \leq 2n-1$, which we will assume is chosen minimal. (The expression \eqref{eq.supersymmetric} can be deduced from \cite{Sj1974} in the elliptic case, and under the weaker hypothesis \eqref{eq.intro.ell.part} the same proof suffices following, for example, \cite[Prop.~2.1]{Vi2013}.)

For $Q$ as in \eqref{eq.supersymmetric}, we recall in Theorem \ref{thm.eigen.core} and Proposition \ref{prop.which.q} the proof \cite[Thm.~3.5]{Sj1974} that there are complex numbers
\[
	\lambda_1, \dots, \lambda_n \in q(\Bbb{R}^{2n})
\]
and polynomials $p_\alpha(x)$ of degree $|\alpha|$ for all $\alpha \in \Bbb{N}^n$ such that 
\[
	u_\alpha(x) = p_\alpha(x) e^{\frac{i}{2}A_+ x \cdot x}
\]
is a generalized eigenfunction of $Q$ with eigenvalue
\[
	\lambda_\alpha = \sum_{j=1}^n \alpha_j \lambda_j.
\]

There are four central goals of the present work.  First, we show that there is a simple computable criterion for boundedness and compactness of the closed densely defined operator $\exp(-tQ)$, for $t \in\Bbb{C}$, on $L^2(\Bbb{R}^n)$, which may be realized as a graph closure beginning with the span of the eigenfunctions $\{u_\alpha\}_{\alpha \in \Bbb{N}}$. Second, we improve the characterizations of compactness, regularization, and decay for these solution operators by comparing with a solution operator for the harmonic oscillator $Q_0$.  Third, we show that the boundedness and compactness for small $|t|$ depends essentially on the range $q(\Bbb{R}^{2n})$ instead of on the eigenvalues $\{\lambda_j\}$.  Finally, we show that for $t > 0$ large the boundedness and compactness of $\exp(-tQ)$ depends essentially only on the real parts of the eigenvalues $\{\lambda_j\}$, which is also reflected in return to equilibrium.

While the results in the body of the paper generally have more precise information, we sum up these four results as follows.  Throughout the remainder of this section, $Q$ is assumed to be written in the form \eqref{eq.supersymmetric} with $A_\pm$ symmetric and $\pm \Im A_\pm > 0$. The eigenvalues $\{\lambda_j\}_{j=1}^n$ are as above.

\begin{theorem}\label{thm.intro.computable}
The solution operator $\exp(-tQ)$, for all $t \in \Bbb{C}$, exists as a closed densely defined operator on $L^2(\Bbb{R}^n)$ with a core given by the span of the generalized eigenfunctions $\{u_\alpha\}$. There exist $\Phi:\Bbb{C}^n \to \Bbb{R}$ real-quadratic and strictly convex and a matrix $M$ with $\opnm{Spec}M = \{\lambda_1,\dots,\lambda_n\}$ such that $\exp(-tQ)$ is bounded if and only if the function
\begin{equation}\label{eq.weight.diff}
	\Phi(e^{tM}z) - \Phi(z)
\end{equation}
is convex and is compact if and only if the function is strictly convex.
\end{theorem}

When $\exp(-tQ)$ is compact, we have very strong decay, regularization, and compactness properties which follow from comparison with semigroup coming from the harmonic oscillator \eqref{eq.def.h.o}.  What is more, in Theorem \ref{thm.mixed.h.o}, we use these techniques to obtain sharp results on how solution operators coming from different harmonic oscillators --- meaning different positive definite self-adjoint operators in the form \eqref{eq.supersymmetric} --- relate to one another under composition.

\begin{theorem}\label{thm.intro.ho.compare}
Let $Q_0$ be as in \eqref{eq.def.h.o}. Whenever $\exp(-tQ)$ is compact, there exists some $\delta = \delta(t) > 0$ such that
\begin{equation}\label{eq.intro.ho.compare}
	\exp(\delta Q_0)\exp(-tQ) \in \mathcal{L}(L^2(\Bbb{R}^n)),
\end{equation}
meaning that the operator is bounded on $L^2(\Bbb{R}^n)$.
\end{theorem}

Writing 
\[
	\exp(-tQ) = \exp(-\delta Q_0)\left(\exp(\delta Q_0)\exp(-tQ)\right)
\]
therefore gives regularity and decay for $\exp(-tQ)u$ when $u \in L^2(\Bbb{R}^n)$, and also implies that the singular values of $\exp(-tQ)$ decay exponentially rapidly like those of $\exp(-\delta Q_0)$,
\[
	s_j(\exp(-tQ)) \leq C\exp\left(-\frac{j^{1/n}}{C}\right).
\]

We have that, as $t \to 0^+$, the boundedness and compactness properties of $\exp(-tQ)$ can be read off from the ellipticity properties of the symbol $q(x,\xi)$.

\begin{theorem}\label{thm.intro.small.t}
The solution operator $\exp(-tQ)$ is bounded for all $t \in [0, \infty)$ if and only if $\Re q(x,\xi) \geq 0$. Furthermore, $\exp(-tQ)$ is compact for all $t \in (0, \infty)$ if and only if \eqref{eq.intro.ell.part} holds, and in this case for $k_0$ minimal in \eqref{eq.intro.ell.part} and 
\begin{equation}\label{eq.def.delta.star}
	\delta^*(t) = \sup\{\delta\in\Bbb{R} \::\: \exp(\delta Q_0)\exp(-tQ) \in \mathcal{L}(L^2(\Bbb{R}^n))\},
\end{equation}
we have
\[
	\delta^*(t) \asymp t^{2k_0+1}, \quad t \to 0^+,
\]
in the sense that the ratio is bounded above and below by positive constants.
\end{theorem}

We recall following \cite[Thm.~1.2]{Vi2013} that the eigenfunctions $\{u_\alpha\}_{\alpha \in \Bbb{N}}$ give a natural decomposition of $L^2(\Bbb{R}^n)$ in energy levels $\opnm{Span}\{u_\alpha \::\: |\alpha| = m\}$, though these may not be orthogonal. We therefore introduce the associated projections
\[
	\Pi_m : L^2(\Bbb{R}^n) \to \opnm{Span}\{u_\alpha \::\: |\alpha| \leq m\},
\]
which commute with $Q$ and one another, which may be deduced from \eqref{eq.def.Pi} below.  The question of return to equilibrium generally concerns $\exp(-tQ)(1-\Pi_0)$, since the range of $\Pi_0$ is $\opnm{Span}\{u_0\}$ and $u_0$ is $\exp(-tQ)$ invariant. We obtain a sharp estimate valid for any $\Pi_m$.

Note that $\opnm{Spec} M \subset \{\Re \lambda > 0\}$ implies that $\|e^{-tM}\| \to 0$  exponentially rapidly as $t \to \infty$ for $t \in \Bbb{R}$. Note also that if $\Re \lambda_j < 0$ for some $j$ then $\exp(-tQ)$ is never bounded for $t > 0$ since $k\lambda_j$ is an eigenvalue of $Q$ for all $k \in \Bbb{N}$.

\begin{theorem}\label{thm.intro.return}
Suppose that $\Re \lambda_j > 0$ for all $j = 1,\dots,n$. Then there exists $T > 0$ sufficiently large such that $\exp(-tQ)$ is compact for all $t \geq T$. Furthermore, with $\rho = \min\{\Re \lambda_j\}$ and $J \in \Bbb{N}$ the size of the largest Jordan block in $M$ for an eigenvalue where $\Re \lambda_j = \rho$,
\[
	\|\exp(-tQ)(1-\Pi_m)\|_{\mathcal{L}(L^2(\Bbb{R}^n))} \asymp \|e^{-tM}\|^{m+1} \asymp \left(t^{J-1} e^{-\rho t}\right)^{m+1}, \quad t > T,
\]
in the sense that the ratios are bounded from above and below by positive constants.
\end{theorem}

\begin{proof}
By Proposition \ref{prop.which.q}, any operator of the form \eqref{eq.supersymmetric} is equivalent to 
\[
	P = Mz\cdot \partial_z
\]
acting on a weighted space of holomorphic functions $H_\Phi$; see Section \ref{subsec.Definitions} for definitions. The corresponding solution operator is given by a change of variables (Proposition \ref{prop.solve.exptP}). Theorem \ref{thm.intro.computable} then follows from Theorems \ref{thm.boundedness.0} and \ref{thm.eigen.core}. That Theorem \ref{thm.intro.ho.compare} holds for \emph{some} harmonic oscillator is the content of Theorem \ref{thm.boundedness.delta} and Proposition \ref{prop.Hermite}; we obtain the result for $Q_0$ because of the Lipschitz relation between harmonic oscillator semigroups near $t = 0$ given by Theorem \ref{thm.mixed.h.o} and Remark \ref{rem.Lipschitz}. Theorem \ref{thm.intro.small.t} is the same as Theorem \ref{thm.subell.decay} in view of Proposition \ref{prop.subell.relation}. Finally, the compactness claim in Theorem \ref{thm.intro.return} is essentially obvious since \eqref{eq.weight.diff} holds automatically when $e^{-tM}\to 0$, but it may be viewed as a special case of Theorem \ref{thm.bounded.tau}, which considers all $t \in \Bbb{C}$ simultaneously. The rest of Theorem \ref{thm.intro.return} is Theorem \ref{thm.return.by.PiN} in the case $\delta = 0$.
\end{proof}

Under the symmetry assumption $A_+ = \overline{A_-}$ in \eqref{eq.supersymmetric}, discussed in Section \ref{subsec.orthogonal}, one can obtain even stronger results: in particular, after a reduction to $A_+ = -A_- = i$, Theorems \ref{thm.intro.ho.compare}, \ref{thm.intro.small.t}, and \ref{thm.intro.return} are linked by
\begin{equation}\label{eq.intro.short.results}
	\|e^{-tM}\| = e^{-\delta^*(t)} = \|\exp(-tQ)(\Pi_{m+1}-\Pi_m)\|_{\mathcal{L}(L^2(\Bbb{R}))}^{\frac{1}{m+1}}.
\end{equation}
Many of the results under this assumption may be realized with simpler proofs relying only on a standard Bargmann transform, and for this reason, we present these results and the natural singular value decomposition independently in \cite{AlVi2014a}.

The plan of the paper is follows.  For the remainder of the introduction, we illustrate the results to follow with two families of concrete examples and then briefly discuss interesting alternate approaches not used here.  Section \ref{sec.Fock} is devoted to the definition and analysis of our operators on Fock spaces.  Section \ref{sec.real} describes the equivalence between quadratic operators in the form \eqref{eq.supersymmetric} on $L^2(\Bbb{R}^n)$ and the operators considered on Fock spaces, as well as related results.  Finally, Section \ref{sec.return} applies this analysis to the problem of return to equilibrium.

\begin{acknowledgements}
The authors would like to thank Johannes Sj\"ostrand for helpful suggestions, as well as Michael Hitrik and Karel Pravda-Starov for an interesting and useful discussion. The authors would also like to thank the anonymous referee for a careful reading and useful suggestions and corrections. The second author is grateful for the support of the Agence Nationale de la Recherche (ANR) project NOSEVOL, ANR 2011 BS01019 01.
\end{acknowledgements}

\subsection{Examples}\label{subsec.examples}

In order to make our results explicit, we discuss their application to well-studied and simple examples.

\subsubsection{The rotated harmonic oscillator}\label{subsubsec.RHO}

We consider the rotated harmonic oscillator
\begin{equation}\label{eq.RHO}
	\begin{aligned}
	Q_\theta &= \frac{1}{2}(D_x + ie^{i\theta}x)(D_x - ie^{i\theta}x)
		\\ &= \frac{1}{2}\left(-\frac{d^2}{dx^2} + e^{2i\theta}x^2 - e^{i\theta}\right),
	\end{aligned}
\end{equation}
where $\theta \in (-\pi/2,\pi/2)$, as an operator on $L^2(\Bbb{R})$. This operator (or variants thereof) appears in \cite{Ex1983}, \cite{Da1999a}, \cite{Boulton2002}, and many other works. We know that $Q_\theta$ has a compact resolvent and that the spectrum of $Q_\theta$ lies in the right half-plane,
\[
	\opnm{Spec} Q_\theta = e^{i\theta}\Bbb{N}.
\]
The eigenfunctions of $Q_\theta$ come from the analytic continuations of the Hermite functions $h_k$ recalled later in \eqref{eq.def.Hermite}; specifically, a complete set of eigenfunctions is given by the formula 
\[
	g_k(x) = e^{i\theta/4}h_k(e^{i\theta/2}x),
\]
which verify
\begin{equation}\label{eq.RHO.eigensystem}
	Q_\theta g_k = ke^{i\theta}g_k, \quad k \in \Bbb{N}.
\end{equation}
The functions $\{g_k\}_{k \in \Bbb{N}}$ form a complete set in that the closure of their span is $L^2(\Bbb{R})$. They do not, however, form a basis, meaning that not every function in $L^2(\Bbb{R})$ can be uniquely expressed as a norm-convergent expansion in basis vectors with fixed coefficients, because their spectral projections 
\begin{equation}\label{eq.RHO.projections}
	\pi_k u(x) = \langle u, \overline{g_k}\rangle g_k(x)
\end{equation}
have exponentially-growing norms, \cite{DaKu2004}.  For a detailed discussion of this phenomenon, see \cite[Sec.~3.3]{Da2007}.

From \cite{Da1999a} and \cite{DeSjZw2004} we have that pseudomodes for $Q_\theta$ of the type \eqref{eq.pseudomodes.general} exist with, for instance, $z_k = ke^{i\tilde{\theta}}$ when $\tilde{\theta}\in (0, 2\theta)$.  We also have from \cite[Prop.~1]{Boulton2002} that the numerical range of $Q_\theta$ is
\[
	\opnm{Num}(Q_\theta) = \{t_1 + e^{2i\theta}t_2 \in \Bbb{C} \::\: t_1, t_2 \geq 0, t_1t_2 \geq 1/4\}.
\]
Therefore both the pseudospectrum and the numerical range of $Q_\theta$ more or less fill out the sector of complex numbers with argument between $0$ and $2\theta$.

We now apply the results contained in the present work to the solution operators generated by these rotated harmonic oscillators.

Following \cite[Ex.~2.6]{Vi2013} with a change of variables, we see that Theorem \ref{thm.intro.computable} applies to $Q_\theta$ with
\[
	M = e^{i\theta}
\]
and
\begin{equation}\label{eq.RHO.weight}
	\Phi(z) = \frac{1}{2}(|z|^2 - (\sin\theta) \Re z^2).
\end{equation}
The conditions for boundedness and compactness in Theorem \ref{thm.intro.computable} can be easily checked by computer, since we see that $\exp(-tQ_\theta)$ is bounded if and only if
\[
	\Phi(e^{tM}z) - \Phi(z) \geq 0, \quad \forall z \in \Bbb{C},
\]
and is compact if and only if the inequality holds strictly.  Since the left-hand side is a quadratic form in $z \in \Bbb{C} \sim \Bbb{R}^2$, this inequality may be verified by checking the eigenvalues of the corresponding Hessian matrix.

Since $\Phi$ is a strictly convex real-quadratic function on $\Bbb{C}$, the condition for boundedness in Theorem \ref{thm.intro.computable} corresponds to the dynamical condition
\begin{equation}\label{eq.dynamics}
	\{\Phi(e^{tM}z) = 1\} = e^{-tM}\{\Phi(z) = 1\} \subset \{\Phi(z) \leq 1\}.
\end{equation}
The weight $\Phi$ is decreasing along all trajectories $z \mapsto e^{-tM}z$ if and only if $|\theta| \leq \pi/4$, corresponding to the ellipticity condition
\[
	\Re (\xi^2 + e^{2i\theta}x^2) \geq 0, \quad \forall (x,\xi) \in \Bbb{R}^{2n}.
\]
This is reflected in boundedness of $\exp(-tQ_\theta)$ as $t \to 0^+$ by Theorem \ref{thm.intro.small.t}.

Let us consider $\theta = 5\pi/12$, for which the property $\Re Q_\theta \geq 0$ no longer holds.  In Figure \ref{fig.RHO.dyn}, we illustrate the condition \eqref{eq.dynamics} by drawing the fixed ellipse $\{\Phi(z) = 1\}$ as a heavy black curve and drawing the ellipses $e^{-tM}\{\Phi(z)=1\}$ as $t \geq 0$ increases. Since $\Re M = \cos\theta$, the long-time dynamics is an exponential contraction; this reflects the long-time boundedness and compactness in Theorem \ref{thm.intro.return}.  We see that for small times $\exp(-tQ_\theta)$ is unbounded, but becomes bounded again at $t_1 \approx 3.011$, when the major axes of the ellipses are sufficiently close. The operator becomes unbounded again at $t_2 \approx 3.549$ and continues to be unbounded up to $t_3 \approx 5.862$. Beyond $t_3$, the exponential contraction is enough to guarantee that $\exp(-tQ_\theta)$ is bounded and compact for all $t \in (t_3, \infty)$.

\begin{figure}
	\centering
	\includegraphics[width=0.95\textwidth]{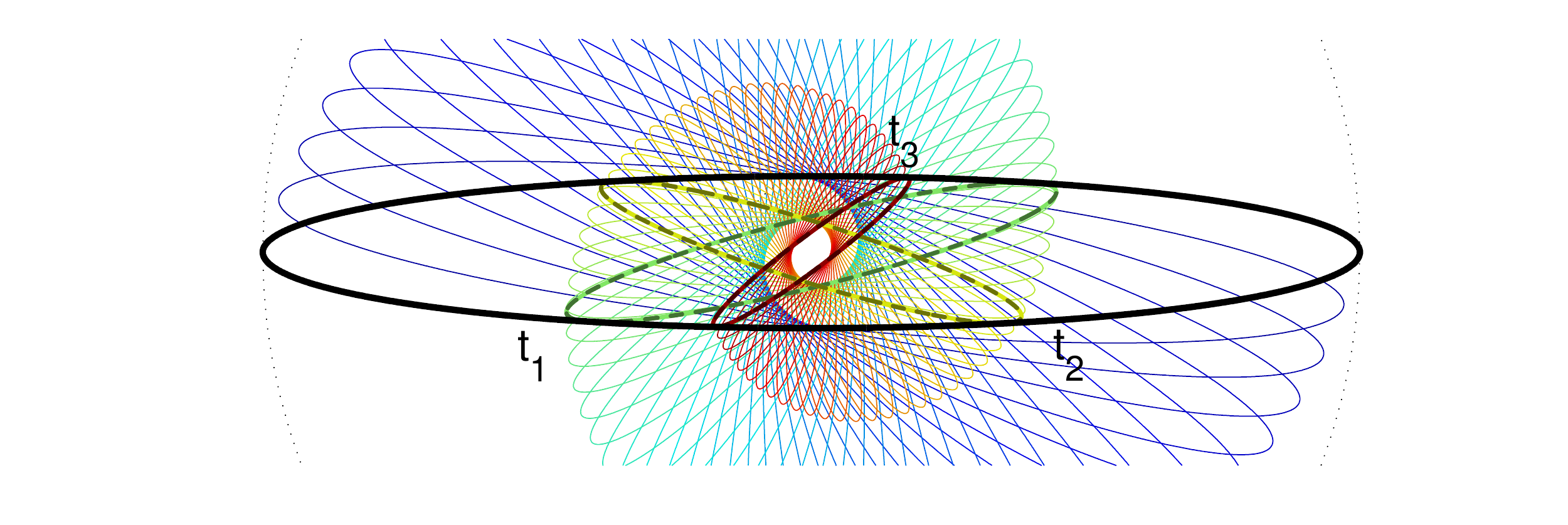}
	\caption{Illustration of \eqref{eq.dynamics} for \eqref{eq.RHO} with $\theta = 5\pi/12$.}
	\label{fig.RHO.dyn}
\end{figure}

Geometrically, it is clear that if we let $\theta \to \pi/2$ from below, the number of times that the operator $\exp(-tQ_\theta)$ for $t > 0$ goes from being unbounded to bounded, and vice versa, goes to infinity, since the rate of contraction tends to zero as the first eccentricity of the ellipses tends to one. Nonetheless, from Theorem \ref{thm.intro.return} we have that, for any $\theta \in (-\pi/2, \pi/2)$, there exists some $T > 0$ where $\exp(-tQ_\theta)$ is compact for all $t \geq T$. Furthermore, for all $u \in L^2(\Bbb{R})$ and $t > T$, the solution operator $\exp(-tQ_\theta)$ is given, up to any fixed order, by the spectral decomposition using \eqref{eq.RHO.projections}:
\[
	\left\|\exp(-tQ_\theta)u - \sum_{k=0}^N e^{-tke^{i\theta}} \pi_k u\right\|_{L^2(\Bbb{R})} = \BigO_N(e^{-t(N+1)\cos\theta}\|u\|_{L^2(\Bbb{R})}).
\]

In fact, Theorem \ref{thm.intro.computable} allows us to easily determine for which $\tau \in \Bbb{C}$ the operator $\exp(-\tau Q_\theta)$ is bounded; for $\theta = 5\pi/12$, we present this set in Figure \ref{fig.ex.Davies} alongside the range of the symbol
\[
	q_\theta(x,\xi) = \xi^2 + e^{2i\theta}x^2
\]
and the eigenvalues of $Q_\theta$, which are $e^{i\theta}\Bbb{N}$. We see that for $|\tau|$ small, the set of $\bar{\tau}$ for which $\exp(\tau Q_\theta)$ is bounded is the sector in opposition to the range of the symbol, which may be defined by
\[
	\{\tau \::\: \Re(\tau q(x,\xi)) \leq 0, ~ \forall (x,\xi) \in \Bbb{R}^2\}.
\]
Formally, this is a consequence of Theorem \ref{thm.intro.small.t}.  For large times, the same role is played by the half-plane in opposition to the spectrum of $Q_\theta$:
\begin{multline*}
	\{\tau \::\: \Re (\tau e^{i\theta}) \leq C_\theta,~\forall \lambda \in \opnm{Spec} Q_\theta\} \subset \{\tau\::\: \exp(\tau Q_\theta) \in \mathcal{L}(L^2(\Bbb{R}))\} 
	\\ \subset \{\tau \::\: \Re (\tau e^{i\theta}) \leq 0,~\forall \lambda \in \opnm{Spec} Q_\theta\}
\end{multline*}
for some $C_\theta > 0$, which is a consequence of Theorem \ref{thm.bounded.tau}.

\begin{figure}
\centering
	\includegraphics[width=0.45\textwidth]{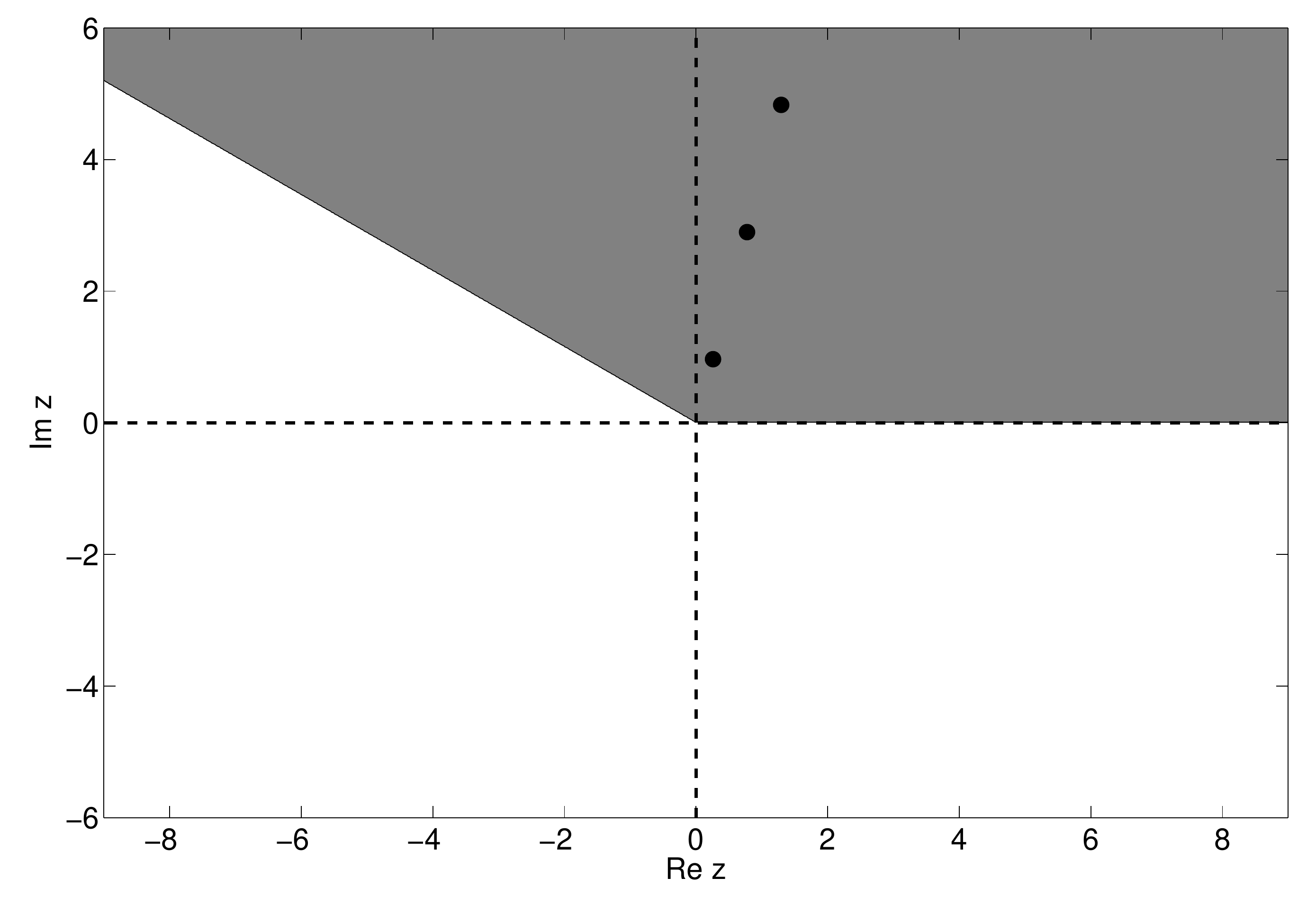}
	\includegraphics[width=0.45\textwidth]{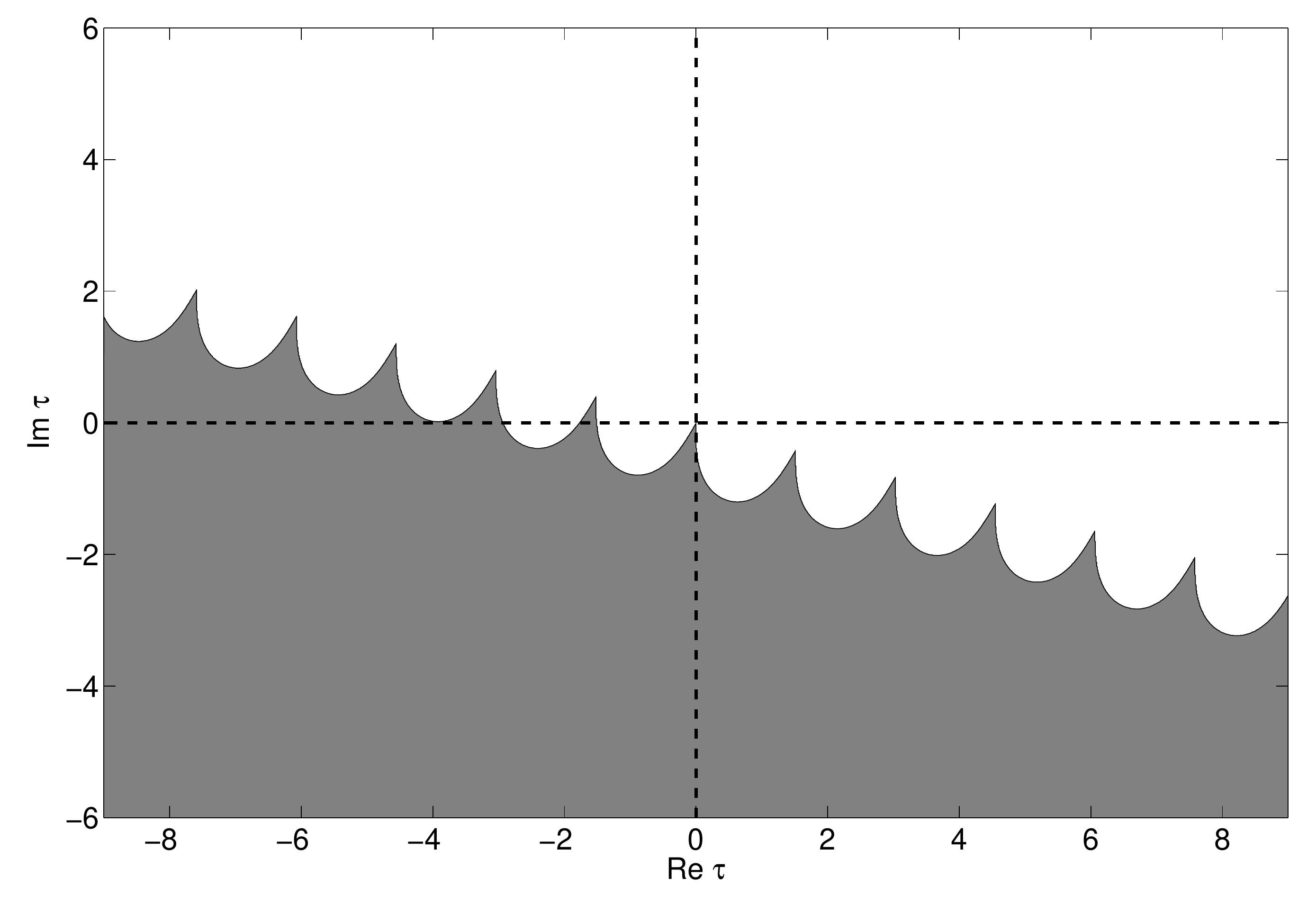}
	\caption{On the left: the range of the symbol of $Q_{\theta_0}$ from \eqref{eq.RHO} for $\theta_0 = 5\pi/12$ and the eigenvalues of $Q_{\theta_0}$; on the right: those $\bar{\tau}$ for which $\exp(\tau Q_{\theta_0})$ is a compact operator.}
	\label{fig.ex.Davies}
\end{figure}

\subsubsection{The Fokker-Planck quadratic model and non-elliptic perturbations}\label{subsubsec.FP}

We also consider the operator
\begin{equation}\label{eq.def.KFP.ex}
	Q_{a,b} = \frac{b}{2}(x_1^2 - \partial_{x_1}^2-1) + \frac{1}{2}(x_2^2 - \partial_{x_2}^2-1) + a(x_1\partial_{x_2} - x_2\partial_{x_1}), \quad a,b\in \Bbb{R}.
\end{equation}
This operator is non-normal whenever $a \neq 0$ and $b \neq 1$ (which we assume henceforth) and when $b = 0$ it coincides with the Fokker-Planck quadratic model \cite[Sec.~5.5]{HeNiBook}.  When $b > 0$, the operator is elliptic in the classical sense. The definition of the semigroup $\exp(-tQ_{a,b})$ for $b \geq 0$ and $t \geq 0$ is well-known and has been the subject of extensive study (see for instance \cite[Sec.~5.5.1]{HeNiBook} and references therein), though we arrive at new results both in this previously-studed situation and in the novel case $b < 0$.  

For $A_\pm = \pm i$ and
\begin{equation}\label{eq.KFP.M}
	M_{a,b} = \left(\begin{array}{cc} b & -a \\ a & 1\end{array}\right),
\end{equation}
we have the following decomposition as in \eqref{eq.supersymmetric}:
\[
	Q_{a,b} = \frac{1}{2}M(D_x + ix)\cdot (D_x - ix).
\]
Note that
\[
	\opnm{Spec} M_{a,b} = \{\lambda_+, \lambda_-\}, \quad \lambda_\pm = \frac{1}{2}(1+b \pm \sqrt{(1-b)^2-4a^2}),
\]
repeated if $(1-b)^2 = 4a^2$. When $b \geq 0$, it is known \cite[Thm.~3.5]{Sj1974}, \cite[Sec.~5.5]{HeNiBook} that
\begin{equation}\label{eq.KFP.spec}
	\opnm{Spec} Q_{a,b} = \{\alpha_1 \lambda_+ + \alpha_2 \lambda_- \::\: \alpha_1, \alpha_2 \in \Bbb{N}\}.
\end{equation}

Since $Q_{a,b}$ leaves invariant the spaces of Hermite functions \eqref{eq.def.Hermite} of fixed degree, meaning
\[
	E_m = \opnm{span}\{h_\alpha(x)\::\: |\alpha| = m\},
\]
it is elementary that $Q_{a,b}$ possesses a complete family of generalized eigenfunctions which may be obtained from the matrix representation of $Q_{a,b}$ on each $E_m$; in fact, the corresponding eigenvalues continue to be given by \eqref{eq.KFP.spec}.  The orthogonal decomposition of $L^2(\Bbb{R}^2)$ into the spaces $E_m$ also lends itself to the family of projections
\begin{equation}\label{eq.KFP.ex.projections}
	\Pi_N u = \sum_{m \leq N} \pi_{E_m} u = \sum_{|\alpha| \leq N} \langle u, h_\alpha\rangle h_\alpha.
\end{equation}

Theorem \ref{thm.intro.computable} applies with the matrix $M$ and the weight $\Phi(z) = \frac{1}{2}|z|^2$ for $z \in \Bbb{C}^2$ (see, e.g.,\ \cite[Ex.~2.7]{Vi2013}), and because $A_+ = \overline{A_-}$, we are in a situation where \eqref{eq.intro.short.results} holds. We have that $\exp(-tQ_{a,b})$ is bounded whenever $\|e^{-tM_{a,b}}\| \leq 1$ and is compact whenever $\|e^{-tM_{a,b}}\| < 1$, and the norm of this matrix exponential gives sharp estimates on decay, regularization, and return to equilibrium.

For $t > 0$, it is clear that $\exp(-tQ_{a,b})$ can only be bounded when $\Re \lambda_\pm \geq 0$. For $b \neq 0$ we have that
\[
	\|e^{-tM_{a,b}}\| = 1-t\min\{b,1\} + \BigO(t^2),
\]
and so $\exp(-tQ_{a,b})$ is bounded for small $t > 0$ if $b > 0$ and unbounded for small $t > 0$ if $b < 0$, which corresponds to ellipticity of $Q_{a,b}$.  That is, the symbol
\[
	q_{a,b}(x,\xi) = \frac{b}{2}(x_1^2 + \xi_1^2) + \frac{1}{2}(x_2^2 + \xi_2^2) - ia(x_1\xi_2 - x_2\xi_1)
\]
has a positive definite real part for $b > 0$, a non-definite real part for $b < 0$, and a positive semidefinite real part when $b = 0$.

When $b = 0$ and $a \neq 0$, we show in Proposition \ref{prop.FP.return} that
\[
	\|e^{-tM_{a,0}}\| = 1 - \frac{a^2}{12}t^3 + \BigO(t^4).
\]
This corresponds to the fact that $k_0 = 1$ in \eqref{eq.intro.ell.part}, which corresponds to small-time regularization by Theorem \ref{thm.intro.small.t} and to small-time decay by \eqref{eq.intro.short.results}.

If $b < 0$ and $a \neq 0$, then $\opnm{Spec} Q_{a,b} = \Bbb{C}$ by Theorem \ref{thm.spectrum.C}. Nonetheless, so long as $\Re \lambda_\pm > 0$, for $t > 0$ sufficiently large one has a strongly regularizing solution operator and exponentially rapid return to equilibrium by Theorem \ref{thm.intro.return}.

These different behaviors can be interpreted in terms of the dynamics of $\dot{z}(t) = M_{a,b}z(t)$, as shown in Figure \ref{fig.KFP.flow}. When $b > 0$, the integral curves which begin on the unit circle depart towards infinity immediately, corresponding to rapid regularization and return to equilibrium.  When $b = 0$, there are integral curves which are tangent to the unit circle, but all tend outwards; this corresponds to regularization and return to equilibrium which begins slowly. When $b < 0$, some level curves penetrate the unit circle, reflecting that the solution operator is wildly unbounded in certain directions of phase space.  On the other hand, the qualitative large-time behavior, where curves tend to infinity reflecting regularization and return to equilibrium, is stable.

\begin{figure}
\centering
	\includegraphics[width=0.32\textwidth]{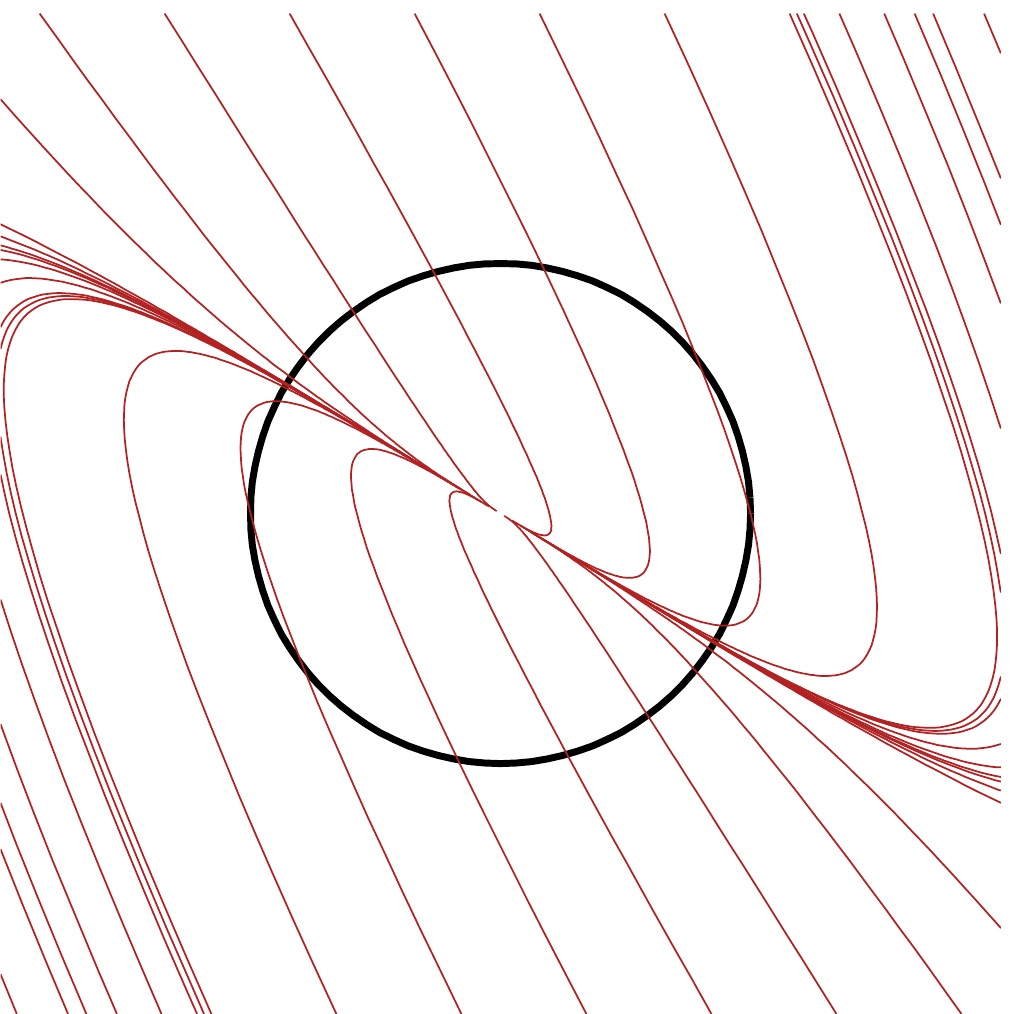}
	\includegraphics[width=0.32\textwidth]{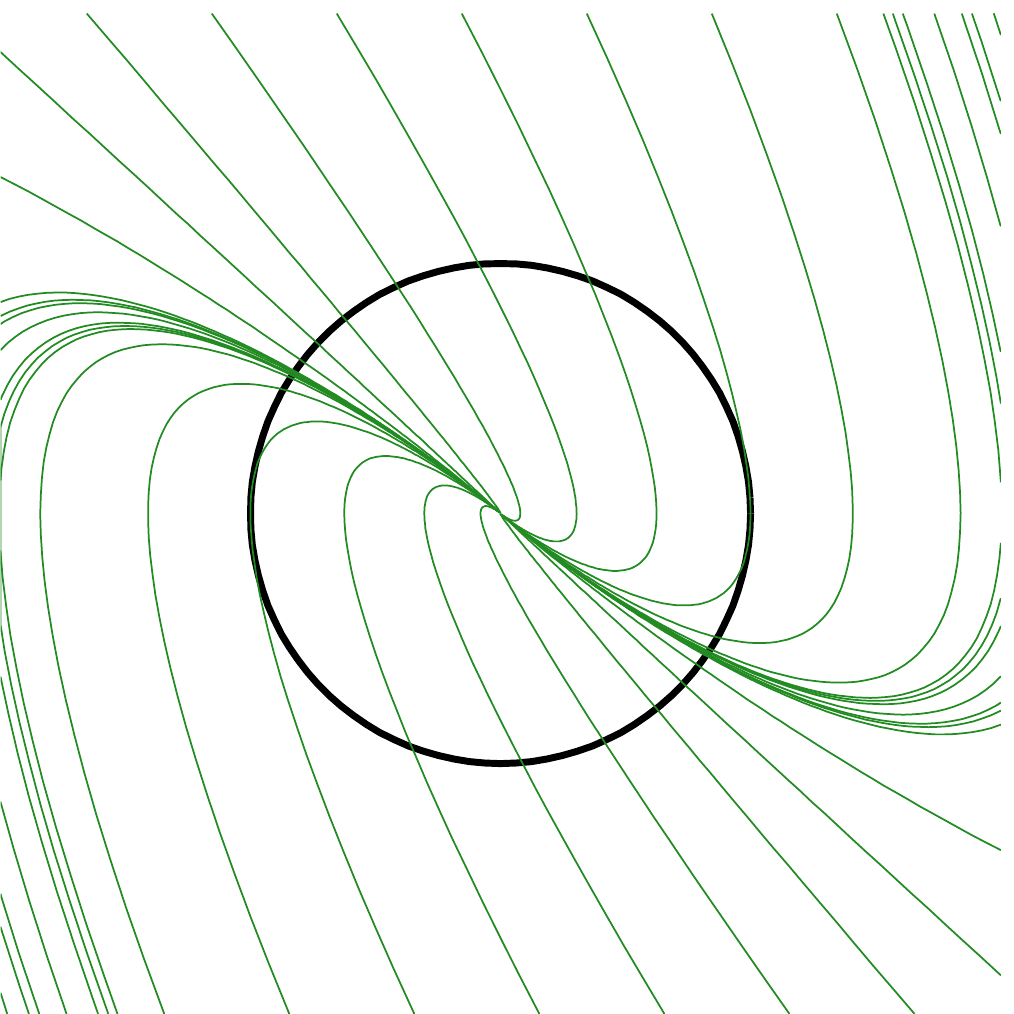}
	\includegraphics[width=0.32\textwidth]{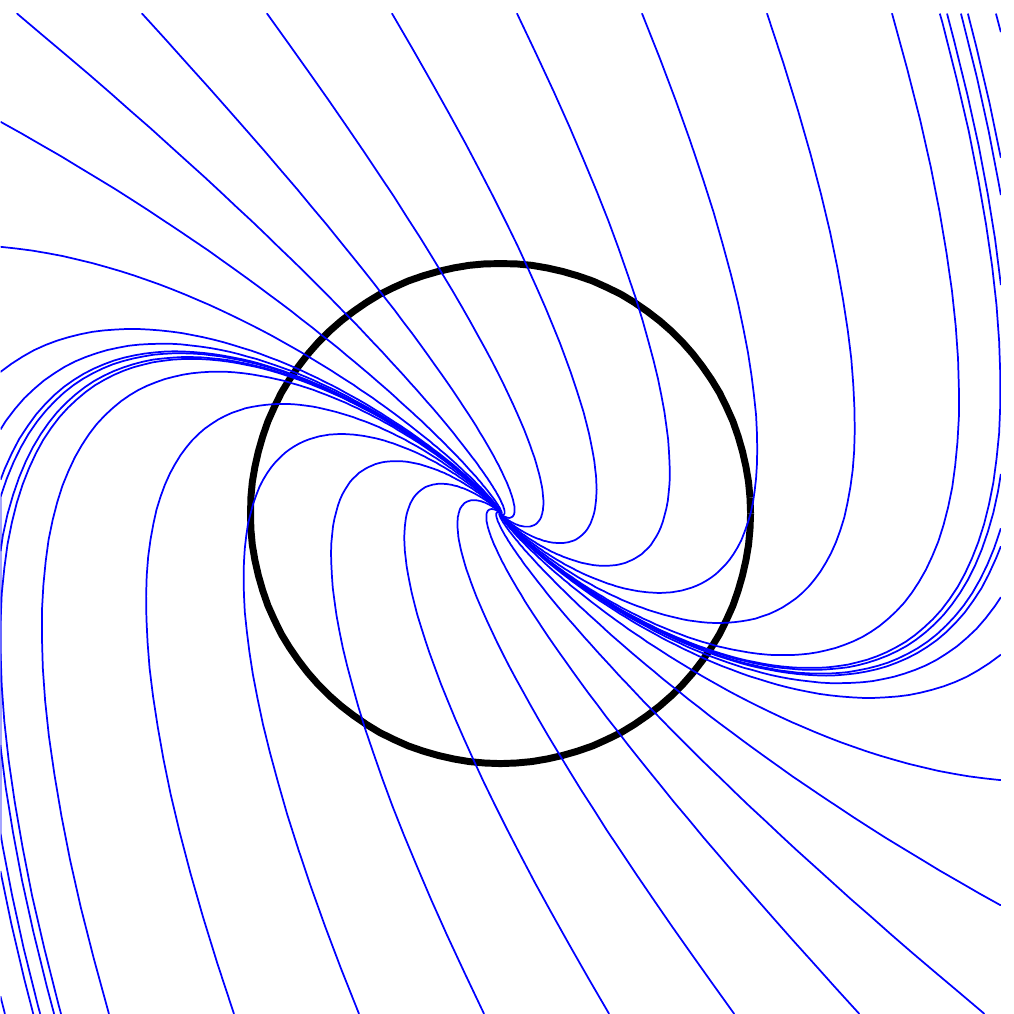}
	\caption{The flow $z \mapsto \exp(tM_{a,b})z$ with $z \in \Bbb{R}^2$ compared with the unit circle for $M_{a,b}$ from \eqref{eq.KFP.M} with $a = 1/2$ and $b = -1/8, 0, 1/8$ left to right.}
	\label{fig.KFP.flow}
\end{figure}

We also can identify the region of $\tau \in \Bbb{C}$ for which $e^{\tau Q_{a,b}}$ is a bounded operator as well as its norm. In Figure \ref{fig.FP.examples}, we study the curves 
\[
	\log\|e^{\tau M_{a,b}}\| = 0, -0.5, -1.0, \dots, -10.0,
\]
appearing from right to left. We only display $\Re \tau \leq 0, \Im \tau \geq 0$ because the norm is invariant under complex conjugation of $\tau$ since $M_{a,b}$ has real entries and because $e^{\tau Q_{a,b}}$ is never bounded when $\Re \tau > 0$. In the left and middle figures, the dotted curves $\{\arg \tau = \arg i\lambda_+\}$ and $\{\Re \tau = -2\log\Im \tau\}$ represent the characterization of the transition from boundedness and unboundedness for large $|\tau|$ from Theorem \ref{thm.bounded.tau}; the corresponding curve for the figure on the right would be the imaginary axis.

\begin{figure}
\centering
	\includegraphics[width=0.2\textwidth]{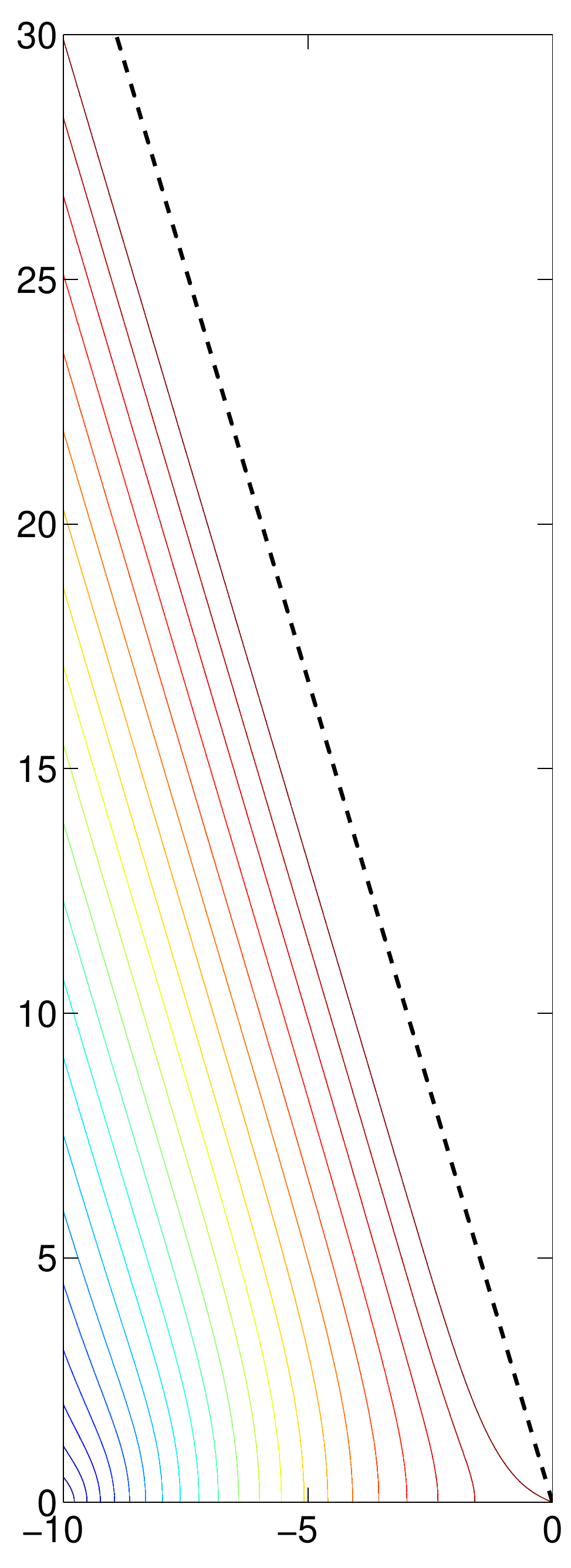}
	\includegraphics[width=0.2\textwidth]{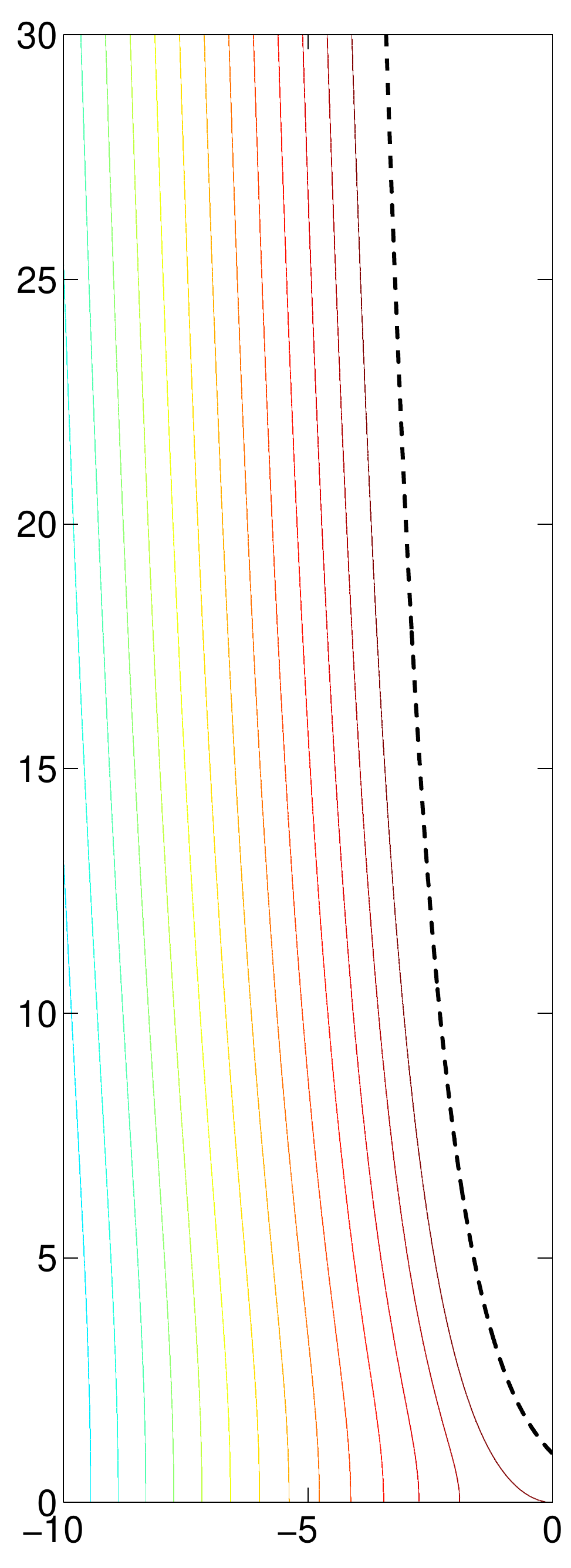}
	\includegraphics[width=0.2\textwidth]{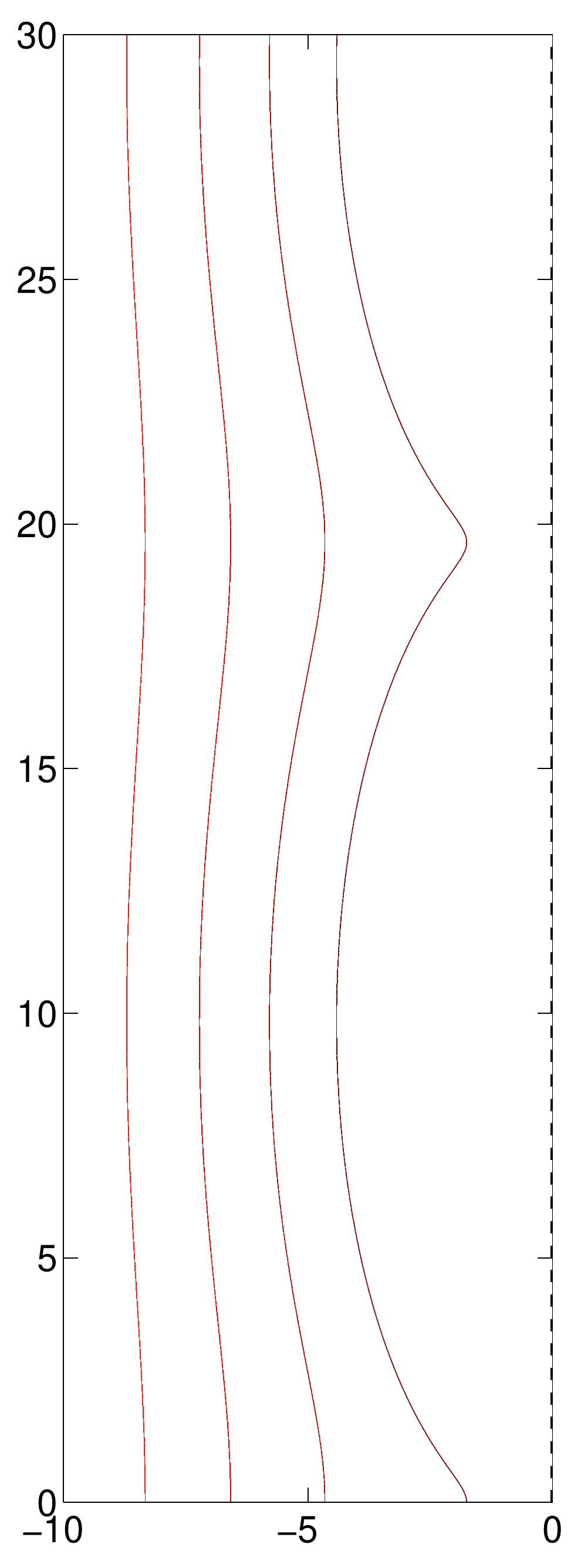}
	\caption{Plots representing boundedness and return to equilibrium for $\exp(\tau Q_{a,b})$ for $a = 1/2$ and $b = 0.05, 0, -0.05$ from left to right.}
	\label{fig.FP.examples}
\end{figure}

\begin{comment}
figure('Position', [0 0 350 960])
set(gca, 'Units', 'pixels')
set(gca, 'Position', [40 40 300 900])
set(gca, 'FontSize', 18)
tR = -10:0.02:0; tI = 0:0.02:30; nGrid = meshgrid(tR, tI);

% KFPell :
B = 0.05; A = 1/2;
M = 2*[B, -A; A, 1];
tic;
for j = 1:length(tR)
for k = 1:length(tI)
nGrid(k,j) = norm(expm((tR(j)+i*tI(k))*M));
end
end
toc;
contour(tR, tI, log(nGrid), -10:0.5:0)
caxis([-10, 0])
hold on;
L = eig(M); L = L(1); L = L/abs(L); LList = 0:40; LList = LList*i*L;
plot(real(LList), imag(LList), 'k--', 'LineWidth', 2)
hold off;

% KFPpart :
B = 0; A = 1/2;
M = 2*[B, -A; A, 1];
tic;
for j = 1:length(tR)
for k = 1:length(tI)
nGrid(k,j) = norm(expm((tR(j)+i*tI(k))*M));
end
end
toc;
contour(tR, tI, log(nGrid), -10:0.5:0)
caxis([-10,0])
hold on;
plot(-1*log(1:0.01:40), 1:0.01:40, 'k--', 'LineWidth', 2)
hold off;

% KFPnon :
B = -0.05; A = 1/2;
M = [B, -A; A, 1];
tic;
for j = 1:length(tR)
for k = 1:length(tI)
nGrid(k,j) = norm(expm((tR(j)+i*tI(k))*M));
end
end
toc;
contour(tR, tI, log(nGrid), -10:0.5:0)
caxis([-10,0])
hold on;
plot(zeros(1,31), 0:30, 'k--', 'LineWidth', 2)
hold off;
\end{comment}

\subsection{Paths not taken}\label{subsec.paths.not.taken}

To finish the introduction, we take a moment to mention alternate approaches which support the results found throughout the present work.  We find that the Fock-space approach used here allows us to provide more precise results more easily, but there certainly may be useful information which can be discovered by following another road.

We recall that under an ellipticity hypothesis, H\"ormander \cite{Ho1995} extended the classical Mehler formula for the harmonic oscillator to the Weyl quantization --- see \eqref{eq.Weyl} --- of quadratic forms $q:\Bbb{R}^n_x \times \Bbb{R}^n_\xi \to \Bbb{C}$ for which $\Re q \geq 0$. Under this assumption, the solution operator $\exp(-tq^w(x,D_x))$ to the evolution equation
\[
	\left\{\begin{array}{l} \partial_t u + q^w(x,D_x)u = 0, \\ u(0,x) = u_0(x) \in L^2(\Bbb{R}^n) \end{array}\right.
\]
was identified as the Weyl quantization of the symbol
\[
	p_t(x,\xi) = (\det \cos tF)^{-1/2} \exp(-\sigma((x,\xi), \tan (tF)(x,\xi))
\]
with the symplectic inner product $\sigma$ in \eqref{eq.def.sigma} and the fundamental matrix $F$ in \eqref{eq.def.F}.

It is possible to define $p_t(x,\xi)$ even without the hypothesis $\Re q \geq 0$.  What is more, one can guess that $\exp(-tq^w(x,D_x))$ should be bounded if and only if
\[
	(x,\xi) \mapsto \sigma((x,\xi), \tan(tF)(x,\xi))
\]
is a positive semidefinite quadratic form on $\Bbb{R}^n_x\times \Bbb{R}^n_\xi$.  Numerically, this apparently agrees with examples in Section \ref{subsec.examples}.  However, it seems more difficult to justify the weak definition when this quadratic form is not positive semidefinite or to describe conditions for positivity of this quadratic form, which involves a matrix tangent and the symplectic inner product, in an intuitive way.  On the other hand, the hypotheses for this Mehler formula do not rely on the symplectic assumptions of Proposition \ref{prop.which.q}, so a deeper study of this approach certainly could be fruitful.

Our approach of recasting a solution operator as a change of weight on a Fock space also appears in \cite{HiPS2009} and \cite{Sj2010}, among other works. In general, the evolved weight $\Phi_t(z)$ solves a Hamilton-Jacobi equation
\begin{equation}\label{eq.HJ.weight}
	\partial_t \Phi_t(z) +\Re p(z, -2i\partial_z\Phi_t(z)) = 0
\end{equation}
for the symbol $p$ of a pseudodifferential operator acting on a Fock space.  The normal form in which we put our operators results in this $t$-dependent weight arising in a very natural and elementary way, and it also allows us to describe the properties of this weight easily, even for long times. In treating more general operators or multiple operators at the same time, which cannot generally be put simultaneously into normal forms, this more general approach has proven very useful.

One could also consider the decomposition in eigenfunctions associated to our operators.  Following the classical theory in \cite[Sec.~3]{Sj1974}, recapitulated in Theorem \ref{thm.eigen.core}, our operators admit a family of eigenfunctions and corresponding eigenvalues parameterized by multi-indices.  If, for the relevant matrix $M$ in \eqref{eq.def.P}, we have $\opnm{Spec} M \subset \{\Re \lambda > 0\}$, then the eigenvalues $\lambda_\alpha$ obey $\Re \lambda_\alpha \geq |\alpha|/C$ for some $C > 0$.  There are natural projections $\Pi_\alpha$ associated with the eigenfunctions, and one has that $\|\Pi_\alpha\| \leq Ce^{C|\alpha|}$ for some $C > 0$, \cite[Cor.~1.6]{Vi2013}.  (This exponential rate of growth is frequently attained.) This supports our finding that, when $\opnm{Spec} M \subset \{\Re \lambda > 0\}$, the operator $\exp(-tP)$ is defined and bounded for sufficiently large real $t$, simply because
\[
	u \mapsto \sum_{\alpha \in \Bbb{N}^n} e^{-t\lambda_\alpha} \Pi_\alpha u
\]
is a norm-convergent series for $t > 0$ large (cf.\ \cite[Cor.~14.5.2]{Da2007}).  On the other hand, this decomposition is very difficult to manipulate, particularly for small $t$. Indeed, this reasoning does not show that, for $Q_\theta$ in \eqref{eq.RHO}, the operator $\exp(-tQ_\theta)$ is bounded for $|\theta| < \pi/4$ and $t > 0$, even though this is well-known \cite{Boulton2002}.

Finally, many of the major features of the right-hand side of Figure \ref{fig.ex.Davies} can be deduced from established results and periodicity.  Specifically, for $Q_\theta$ as in \eqref{eq.RHO} with $0 < \theta < \pi/2$, we have that $e^{i\psi}Q_\theta$ is elliptic if $-\pi/2 < \psi < \pi/2-2\theta$.  Therefore, for $\tau \in \Bbb{C}\backslash \{0\}$, we have boundedness for the solution operator in a sector in the complex plane:
\[
	\arg \tau \in \left(\frac{\pi}{2}, \frac{3\pi}{2}-2\theta\right) \implies \exp(\tau Q_\theta) \in \mathcal{L}(L^2(\Bbb{R})),
\]
and the operator is also Hilbert-Schmidt and regularizing by \cite{Boulton2002} or \cite{PS2008a}. As a consequence, the behavior of $\exp(\tau Q_\theta)$ is determined by the behavior on the complete set of eigenfunctions \eqref{eq.RHO.eigensystem}.  It is clear that for any $k \in \Bbb{N}$ and $j \in \Bbb{Z}$,
\[
	\exp((\tau + i\pi e^{-i\theta}j)Q_\theta)g_k = (-1)^j\exp(\tau Q_\theta)g_k,
\]
revealing that the set where $\exp(\tau Q_\theta)$ is bounded is periodic as seen in Figure \ref{fig.ex.Davies}. Naturally, this approach relies on a periodicity in the eigenvalues which is quite rare in dimension greater than one; furthermore, we improve the description of both the set where $\exp(\tau Q_\theta)$ is bounded or compact as well as the description of its compactness and regularization properties.  It is nonetheless interesting to have this alternate confirmation, and even in higher dimensions there seem to be certain operators exhibiting possible quasi-periodicity phenomena, for example the operator for the rightmost plot of Figure \ref{fig.FP.examples}.

\section{Solution operators for certain quadratic operators on Fock spaces}\label{sec.Fock}

In this section, we begin by defining our Fock spaces and our operators acting on them, leading immediately to a natural weak solution of the corresponding evolution equation.  We then establish a variety of results on the structure of these Fock spaces, in order to better understand the solution operators.  This puts us in a position to establish several sharp results on the boundedness and compactness properties of these operators, and we finish by proving a variety of consequences.

\subsection{Definition of the operator and solution of the evolution equation}\label{subsec.Definitions}

We begin by defining some quadratically weighted Fock spaces and our operators which act on them.  We focus on real-valued weight functions satisfying
\begin{equation}\label{eq.Phi.assumptions}
	\Phi:\Bbb{C}^n \to \Bbb{R} \textnormal{ is real-quadratic and strictly convex}.
\end{equation}
Using $dL(z)$ for Lebesgue measure on $\Bbb{C}^n \sim \Bbb{R}^n_{\Re z} \times \Bbb{R}^n_{\Im z}$, we define the associated Fock space
\[
	H_\Phi = \opnm{Hol}(\Bbb{C}^n) \cap L^2(\Bbb{C}^n, e^{-2\Phi(z)}\, dL(z)).
\]
The norm and inner product on $H_\Phi$ are given by the weighted $L^2$ space, meaning that
\begin{equation}\label{eq.def.norm.HPhi}
	\|u\|_\Phi^2 = \int_{\Bbb{C}^n} |u(z)|^2 e^{-2\Phi(z)}\,dL(z).
\end{equation}
Throughout, we use the subscript $\Phi$ to identify the weight, which changes frequently.  We also use the notation $\Phi(F\cdot)$ in the subscript to mean the weight $\Phi(Fz)$.

For $M= (m_{j,k})_{j,k=1}^n$ any matrix, define
\begin{equation}\label{eq.def.P}
	P = (Mz)\cdot \partial_z = \sum_{j,k=1}^n m_{j,k}z_k\partial_{z_j}.
\end{equation}
Any derivatives of functions on $\Bbb{C}^n$ are assumed to be holomorphic, as in $\partial_z = \frac{1}{2}(\partial_{\Re z} - i \partial_{\Im z})$. If $M$ is not in Jordan normal form, we may put it in Jordan normal form through a change of variables like \eqref{eq.def.V}.

Our object of study is the evolution equation
\begin{equation}\label{eq.evolution}
	\left\{ \begin{array}{r}\partial_t u(t,z) + Pu(t,z) = 0,
		\\ u(0,z) = u_0(z) \in H_\Phi.\end{array}\right.
\end{equation}
We may solve this equation for all real and complex times through a change of variables.

\begin{proposition}\label{prop.solve.exptP}
Let $P$ be as in \eqref{eq.def.P} acting on $H_\Phi$ for $\Phi$ verifying \eqref{eq.Phi.assumptions}.  Then the evolution problem \eqref{eq.evolution} admits the solution
\[
	u(t,z) = u_0(e^{-tM}z),
\]
which is unique in the space of holomorphic functions on $\Bbb{C}_t \times \Bbb{C}^n_z$.  We therefore write henceforth
\[
	\exp(\tau P)u_0(z) = u_0(e^{\tau M}z),
\]
which is a closed densely defined operator on $H_\Phi$ when equipped with its maximal domain 
\[
	\{u_0 \in H_\Phi \::\: u_0(e^{\tau M}z) \in H_\Phi\}.
\]
The norm of $\exp(\tau P)u_0$ may be calculated via the formula
\begin{equation}\label{eq.norm.change.vars}
	\|\exp(\tau P)u_0\|_{\Phi} = \exp(-\Re(\tau \opnm{Tr} M))\|u_0\|_{\Phi_{-\tau}}, \quad \Phi_{-\tau}(z) = \Phi(e^{-\tau M}z).
\end{equation}
\end{proposition}

\begin{remark*}
It is then clear from the definition of the norm \eqref{eq.def.norm.HPhi} that $\exp(\tau P)$ is bounded whenever $\Phi_{-\tau} \geq \Phi$; we see in Theorem \ref{thm.boundedness.0} below that this condition is necessary as well. We also show in Theorem \ref{thm.eigen.core} (see also Proposition \ref{prop.polys.core}) that the polynomials form a core for $\exp(\tau P)$; this is a natural minimal domain for $\exp(\tau P)$ because it is a dense subset of $H_\Phi$ which can be realized as the span of the generalized eigenfunctions of $P$.
\end{remark*}

\begin{proof} That $u(t,z)$ is holomorphic and solves \eqref{eq.evolution} is immediate from the fact that $\partial_z u(Fz) = F^\top u'(Fz)$ for any matrix $F \in \Bbb{M}_{n\times n}(\Bbb{C})$ and any holomorphic function $u:\Bbb{C}^n \to \Bbb{C}$.  Unicity follows from noting that any solution $u(t,z)$ must obey $\partial_t(u(t, e^{tM}z)) = 0$ and therefore $u(t,e^{tM}z) = u_0(z)$.

Since $e^{\tau M}$ is invertible, $\exp(\tau P)$ is a linear isomorphism on the space of polynomials which is dense in $H_\Phi$ (see e.g.\ \cite[Rem.~2.5]{Vi2013}), and therefore $\exp(\tau P)$ is densely defined.  Convergence in $H_\Phi$ implies convergence in $L^2_{\textnormal{loc}}(\Bbb{C}^n)$ which, for holomorphic functions, implies pointwise convergence.  (That pointwise evaluation in $H_\Phi$ is continuous means that $H_\Phi$ is a reproducing kernel Hilbert space.)  Therefore if $(u_k, \exp(\tau P)u_k) \to (u,v)$ in $H_\Phi \times H_\Phi$, then $u_k \to u$ pointwise, so $\exp(\tau P)u_k \to u(e^{\tau M}\cdot)$ pointwise.  This identifies that $v = \exp(\tau P)u$, so the graph of $\exp(\tau P)$ is closed.

We have a general fact regarding changes of variables on Fock spaces: if $F \in GL_n(\Bbb{C})$ is an invertible matrix, then
\begin{equation}\label{eq.def.V}
	\mathcal{V}_F: H_{\Phi} \ni u(z) \mapsto |\det F| u(Fz) \in H_{\Phi(F\cdot)}
\end{equation}
is unitary with inverse $\mathcal{V}^*_F = \mathcal{V}_{F^{-1}}$, which follows immediately from a change of variables applied to \eqref{eq.def.norm.HPhi}.  We note also that
\begin{equation}\label{eq.trans.V}
	\mathcal{V}_F P \mathcal{V}_F^* = F^{-1}MFz\cdot \partial_z.
\end{equation}
Then the norm computation \eqref{eq.norm.change.vars} follows from the observations that 
\[
	\mathcal{V}_{e^{-\tau M}}\exp(\tau P) u(z) = |\det e^{-\tau M}|u(z)
\]
and that $|\det e^{-\tau M}| = e^{-\Re \tau \opnm{Tr} M}$.
\end{proof}

\subsection{Results on the structure of Fock spaces}

Next, we collect a series of statements about the structure of Fock spaces $H_\Phi$ for $\Phi$ obeying \eqref{eq.Phi.assumptions}.  To begin, we recall several useful decompositions of the weight function $\Phi$ and, more generally, real quadratic forms on $\Bbb{C}^n$.

\begin{lemma}\label{lem.decomposition}
Let $\Phi:\Bbb{C}^n \to \Bbb{R}$ be a real-valued real-quadratic form on $\Bbb{C}^n$.  
\begin{enumerate}[(i)]
\item Then $\Phi$ may be decomposed into Hermitian and pluriharmonic parts,
\begin{equation}\label{eq.decompose.herm.plh}
	\begin{aligned}
	\Phi(z) & = \Phi_{\textnormal{herm}}(z) + \Phi_{\textnormal{plh}}(z)
	\\ &= \frac{1}{2}(\Phi(z) + \Phi(iz)) + \frac{1}{2}(\Phi(z) - \Phi(iz))
	\\ &= \langle z, \Phi''_{\bar{z}z} z\rangle + \Re (z\cdot \Phi''_{zz} z).
	\end{aligned}
\end{equation}
Because $\Phi$ is real-valued, $\Phi''_{\bar{z}z} = \overline{\Phi''_{z\bar{z}}}$ is a Hermitian matrix.
\item Furthermore, $\Phi$ is convex if and only if
\[
	\Phi_{\textnormal{herm}}(z) \geq |\Phi_{\textnormal{plh}}(z)|, \quad \forall z \in \Bbb{C}^n
\]
and is strictly convex if and only if the inequality is strict on $\{|z| = 1\}$.  Therefore $\Phi''_{\bar{z}z}$ is positive semidefinite whenever $\Phi$ is convex and positive definite whenever $\Phi$ is strictly convex.
\item Whenever $\Phi''_{\bar{z}z}$ is positive semidefinite, we may write
\begin{equation}\label{eq.decompose.Phi}
	\Phi(z) = \frac{1}{2}|Gz|^2 - \Re h(z)
\end{equation}
where $G \in \Bbb{M}_{n\times n}(\Bbb{C})$ may be taken positive semidefinite Hermitian and $h(z) = \frac{1}{2}z\cdot \Phi''_{zz}z$ is holomorphic.
\item\label{it.decomposition.Sigma} Whenever $\Phi''_{\bar{z}z}$ is positive definite we may take $G$ in \eqref{eq.decompose.Phi} to be positive definite Hermitian and there exists a unitary matrix $U$ such that
\[
	\Phi(UG^{-1}z) = \frac{1}{2}\left(|z|^2 - \Re z\cdot \Sigma z\right)
\]
where $\Sigma = (G^{-1})^\top U^\top \Phi''_{zz}UG^{-1}$ is diagonal with entries in $[0,1)$.
\end{enumerate}
\end{lemma}

For proofs, which are more or less elementary, we refer the reader to \cite[Sec.~4.1]{Vi2013} and references therein, but similar statements exist throughout the literature.

We turn to the reproducing kernel of $H_\Phi$. Recall that the reproducing kernel at $w \in \Bbb{C}^n$ for $H_\Phi$ is the function $k_w \in H_\Phi$ such that
\begin{equation}\label{eq.def.reproducing}
	\langle f, k_w\rangle_{\Phi} = f(w), \quad \forall f \in H_\Phi.
\end{equation}
We begin by identifying this reproducing kernel through a reduction to a reference weight
\begin{equation}\label{eq.def.Psi}
	\Psi(z) = \frac{1}{2}|z|^2.
\end{equation}

\begin{lemma}\label{lem.reproducing}
	Let $\Phi$ satisfy \eqref{eq.Phi.assumptions} and recall the decomposition \eqref{eq.decompose.Phi}.
	
	Then the map
	\begin{equation}\label{eq.def.U}
		\mathcal{U}u(z) = |\det G| u(Gz) e^{-h(z)}:H_\Psi \to H_\Phi
	\end{equation}
	is unitary.  Consequently,
	\begin{enumerate}[(i)]
	\item\label{it.reproducing.1} the reproducing kernel at $w \in \Bbb{C}$ for $H_\Phi(\Bbb{C}^n)$ is given by
	\begin{equation}\label{eq.def.kw}
		k_w(z) = \pi^{-n}|\det G|^2\exp\left((Gz)\cdot(\overline{Gw}) - h(z) - \overline{h(w)}\right),
	\end{equation}
	and
	\item\label{it.reproducing.2} the set
	\begin{equation}\label{eq.def.e.alpha}
		\left\{e_\alpha = \frac{|\det G|}{\sqrt{\pi^n \alpha!}}(Gz)^\alpha e^{-h(z)}\::\: \alpha \in \Bbb{N}^n\right\}
	\end{equation}
	forms an orthonormal basis of $H_\Phi$.
	\end{enumerate}
\end{lemma}

\begin{proof}
	In addition to \eqref{eq.def.V}, we record one more transformation between Fock spaces, depending on a holomorphic function $g:\Bbb{C}^n\to \Bbb{C}$:
	\begin{equation}\label{eq.def.W}
		\mathcal{W}_g: H_\Phi \ni u(z) \mapsto u(z)e^{g(z)} \in H_{\Phi + \Re g}.
	\end{equation}
	From the definition \eqref{eq.def.norm.HPhi} of the norm in $H_\Phi$, it is clear that $\mathcal{W}_g$ is unitary with inverse $\mathcal{W}_g^* = \mathcal{W}_{-g}$.  
	For later use, we also note that
	\begin{equation}\label{eq.trans.W}
		\mathcal{W}_g P \mathcal{W}_g^* = Mz\cdot(\partial_z - g'(z)).
	\end{equation}
	Then the fact that $\mathcal{U}:H_\Psi \to H_\Phi$ is unitary with inverse
	\[
		\mathcal{U}^*u(z) = \frac{1}{|\det G|}u(G^{-1}z)e^{h(G^{-1}z)}
	\]
	follows directly from writing $\mathcal{U} = \mathcal{W}_{-h}\mathcal{V}_{G}$.  Since the reproducing kernel at $w$ for $H_\Psi$ is
	\[
		\tilde{k}_w = \pi^{-n}\exp\left(z\cdot \bar{w}\right),
	\]
	we have
	\[
		\langle u, \mathcal{U} \tilde{k}_{Gw}\rangle_{\Phi} = \mathcal{U}^* u(Gw) = \frac{1}{|\det G|}u(w)e^{h(w)}.
	\]
	Therefore the reproducing kernel at $w\in \Bbb{C}^n$ for $H_\Phi$ is given by the formula
	\[
		k_w = |\det G| e^{-h(w)} \mathcal{U} \tilde{k}_{Gw},
	\]
	and a direct computation gives claim (\ref{it.reproducing.1}).

	Claim (\ref{it.reproducing.2}) follows from writing $e_\alpha = \mathcal{U}u_\alpha$, where
	\begin{equation}\label{eq.def.f.alpha}
		f_\alpha(z) = \frac{1}{\sqrt{\pi^n \alpha!}}z^\alpha,
	\end{equation}
	since $\{f_\alpha\}_{\alpha \in \Bbb{N}}$ forms an orthonormal basis in $H_\Psi$.
\end{proof}

We remark again that the injection from $H_{\Phi_1}$ to $H_{\Phi_2}$ is clearly bounded whenever $\Phi_2 \geq \Phi_1 - C$ for some $C \in \Bbb{R}$.  We show now that this is a necessary condition in the setting of weights satisfying \eqref{eq.Phi.assumptions}.

\begin{proposition}\label{prop.embedding}
Let $\Phi_j, j = 1,2$ be quadratic forms on $\Bbb{C}^n$ obeying \eqref{eq.Phi.assumptions}, decomposed according to \eqref{eq.decompose.Phi} with $G_j$ and $h_j$, $j=1,2$.  Then the injection
\begin{equation}\label{eq.def.iota}
	\iota : H_{\Phi_1} \to H_{\Phi_2}
\end{equation}
is bounded if and only if
\begin{equation}\label{eq.embedding.inequality}
	\Phi_2(z) \geq \Phi_1(z),
\end{equation}
If, in addition, the injection $\iota$ is compact, then this inequality must hold strictly on $\{|z| = 1\}$.
\end{proposition}

\begin{proof}
Let $k_w^{(j)}(z)$ be the reproducing kernel for $\Phi_j$ with $j=1,2$ according to Lemma \ref{lem.reproducing}.  Then for all $u \in H_{\Phi_1} \cap H_{\Phi_2}$, a set which includes the polynomials which are dense in both $H_{\Phi_1}$ and $H_{\Phi_2}$, we have
\[
	\langle \iota u, k_w^{(2)}\rangle_{\Phi_2} = u(w) = \langle u, k_w^{(1)}\rangle_{\Phi_1}.
\]
Therefore $\iota^* k_w^{(2)} = k_w^{(1)}$, so if $\iota$ is a bounded operator then
\[
	\|k_w^{(1)}\|_{\Phi_1} \leq \|\iota\| \|k_w^{(2)}\|_{\Phi_2}, \quad \forall w \in \Bbb{C}^n.
\]
We see from \eqref{eq.def.kw} that, for $j=1,2$,
\begin{equation}\label{eq.kww}
	\|k_w^{(j)}\|_{\Phi_j}^2 = k_w^{(j)}(w) = \pi^{-n}|\det G_j|^2 e^{2\Phi_j(w)},
\end{equation}
so if $\iota$ is bounded, then
\[
	\pi^{-n}|\det G_1|^2 e^{2\Phi_1(w)} \leq \|\iota\|^2 \pi^{-n}|\det G_2|^2 e^{2\Phi_2(w)}, \quad \forall w \in \Bbb{C}^n.
\]
The lower bound
\begin{equation}\label{eq.change.Phi.bound}
	\|\iota\| \geq |\det G_2^{-1}G_1| \sup_{w \in \Bbb{C}^n} \exp(\Phi_1(w)-\Phi_2(w))
\end{equation}
for the norm of $\iota$ follows immediately, and this implies \eqref{eq.embedding.inequality} because $\Phi_1(w) - \Phi_2(w)$ is quadratic and therefore must be bounded above by zero if it is bounded above at all.  Sufficiency of \eqref{eq.embedding.inequality} is clear from the definition of the norm on $H_{\Phi_j}$.

For the claim about compactness, we first show that the normalized reproducing kernels $k_w^{(2)}/\|k_w^{(2)}\|_{\Phi_2}$ tend weakly to zero as $|w| \to \infty$.  Since the linear span of $\{k_z^{(w)}\::\: z \in \Bbb{C}^n\}$ is dense in $H_{\Phi_2}$, it suffices to observe that by \eqref{eq.kww} we have, for each $z \in \Bbb{C}^n$,
\[
	\left\langle \frac{k_w^{(2)}}{\|k_w^{(2)}\|_{\Phi_2}}, k_z^{(2)}\right\rangle_{\Phi_2} = \frac{\pi^{n/2}}{|\det G_2|} k_w^{(2)}(z)e^{-\Phi_2(w)} \to 0, \quad |w| \to \infty.
\]

If $\iota$ is compact, then as the compact image of a sequence weakly converging to zero, $\iota^* k_w^{(2)}/\|k_w^{(2)}\|_{\Phi_2}$ converges strongly to zero as $|w|\to \infty$, and by the previous calculations,
\[
	\left\|\iota^* \frac{k_w^{(2)}}{\|k_w^{(2)}\|_{\Phi_2}}\right\|_{\Phi_1} = \frac{\|k_w^{(1)}\|_{\Phi_1}}{\|k_w^{(2)}\|_{\Phi_2}} = |\det G_1^{-1}G_2| \exp(\Phi_1(w)-\Phi_2(w)).
\]
Since this quantity tends to zero as $|w| \to \infty$, this proves that
\begin{equation}\label{eq.change.Phi.compactness}
	\lim_{|w|\to\infty} \exp(\Phi_1(w) - \Phi_2(w)) = 0.
\end{equation}
Using again that $\Phi_1(w) - \Phi_2(w)$ is quadratic, this implies that \eqref{eq.embedding.inequality} must hold strictly on $\{|z| = 1\}$, completing the proof of the proposition.
\end{proof}

To study compactness, it is natural to study a similar class of solution operators to those considered in Proposition \ref{prop.solve.exptP}, except acting on $H_\Psi$ with $\Psi$ from \eqref{eq.def.Psi}.  Realizing these operators via conjugation with $\mathcal{U}$ from \eqref{eq.def.U}, it becomes clear that their effect is to modify the Hermitian part of the weight $\Phi$. In Proposition \ref{prop.Hermite}, we see that there is a correspondence between the special case $P_0$, defined below in \eqref{eq.def.P0}, and $Q_0$ the harmonic oscillator in \eqref{eq.def.h.o}; this relates to to the more or less classical picture in which decay for functions in $H_\Phi$ corresponds to decay and smoothness for functions in $L^2(\Bbb{R}^n)$.

\begin{proposition}\label{prop.h.o.semigroup}
Let $\Phi$ obey \eqref{eq.Phi.assumptions} and be decomposed as in \eqref{eq.decompose.Phi}. Recall also the definition \eqref{eq.def.U} of $\mathcal{U}:H_\Psi \to H_\Phi$ , and for $B \in \Bbb{M}_{n\times n}(\Bbb{C})$, let
\[
	\begin{aligned}
	\mathcal{Q}_B &= Bz\cdot(\partial_z + h'(z))
	\\ &= \mathcal{U}(GBG^{-1}z\cdot \partial_z)\mathcal{U}^*.
	\end{aligned}
\]
Also, for any $\tau \in \Bbb{C}$, let $\exp(\tau \mathcal{Q}_B) = \mathcal{U}\exp(\tau GBG^{-1} z\cdot \partial_z)\mathcal{U}^*$ be defined as in Proposition \ref{prop.solve.exptP}.

Then with 
\begin{equation}\label{eq.def.Phi.tau.B}
	\begin{aligned}
	\Phi^{(\tau), B}(z) & = \frac{1}{2}|Ge^{-\tau B}z|^2 - \Re h(z)
	\\ & = \Phi(z) + \frac{1}{2}\left(|Ge^{-\tau B}z|^2 - |Gz|^2\right),
	\end{aligned}
\end{equation}
we have
\[
	\|\exp(\tau \mathcal{Q}_B)u\|_\Phi = \exp(-\Re(\tau \opnm{Tr} B))\|u\|_{\Phi^{(\tau),B}}.
\]
Furthermore, $\mathcal{Q}_B$ is self-adjoint (resp.\ normal) if and only if $GBG^{-1}$ is self-adjoint (resp.\ normal).
\end{proposition}

\begin{remark*}
This is operator particularly useful when $B$ is a constant times the identity matrix, or at the very least when $GBG^{-1}$ is positive semi-definite Hermitian.  When $B$ is the identity matrix, we omit $B$ and define, for $\delta \in \Bbb{R}$,
\begin{equation}\label{eq.def.Phi.delta}
	\begin{aligned}
	\Phi^{(\delta)}(z) &= \frac{e^{-2\delta}}{2}|Gz|^2 - \Re h(z)
	\\ & = \Phi(z) + \frac{e^{-2\delta}-1}{2}|Gz|^2.
	\end{aligned}
\end{equation}
This case corresponds to a reference harmonic oscillator adapted to the spaces $H_\Phi$, as shown in Proposition \ref{prop.Hermite}. To refer to this operator throughout, we define
\begin{equation}\label{eq.def.P0}
	P_0 = z\cdot(\partial_z + h'(z))
\end{equation}
and note that, with $e_\alpha$ as in \eqref{eq.def.e.alpha},
\[
	P_0 e_\alpha = |\alpha|e_\alpha.
\]
It is clear that $P_0$ is self-adjoint, and from Proposition \ref{prop.solve.exptP}, we have the norm relation
\[
	\|\exp(\delta P_0)u\|_{\Phi} = e^{-\delta n}\|u\|_{\Phi^{(\delta)}}.
\]
We remark that this relation may also be checked directly on expansions in the orthogonal sets $\{e_\alpha\}_{\alpha \in \Bbb{N}^n}$ and $\{\exp(\delta P_0)e_\alpha\}_{\alpha \in \Bbb{N}^n}$ via a change of variables.
\end{remark*}

\begin{proof}
The alternate expression of $\mathcal{Q}_B$ follows from writing $\mathcal{U} = \mathcal{W}_{-h}\mathcal{V}_G$ and the relations \eqref{eq.trans.V} and \eqref{eq.trans.W}.  Having reduced to an operator acting on $H_\Psi$, we recall that $z_j^* = \partial_{z_j}$ as operators on $H_\Psi$ (more general formulas for adjoints may be found in \cite[Sec.~4.2]{Vi2013}). Therefore, working on $H_\Psi$,
\[
	(GBG^{-1}z\cdot \partial_z)^* = z \cdot \overline{GBG^{-1}}\partial_z = (GBG^{-1})^*z\cdot \partial_z.
\]
Therefore $\mathcal{Q}_B$ is self-adjoint if and only if $GBG^{-1}$ is self-adjoint. For any $M_1, M_2 \in \Bbb{M}_{n\times n}(\Bbb{C})$, we compute the commutator
\[
	[M_1z\cdot \partial_z, M_2z\cdot \partial_z] = -[M_1, M_2]z\cdot \partial_z,
\]
from which it follows that $\mathcal{Q}_B$ is normal if and only if $GBG^{-1}$ is normal.

Using Proposition \ref{prop.solve.exptP} and the definition \eqref{eq.def.U} of $\mathcal{U}$, we then check that
\[
	\exp(\tau \mathcal{Q}_B)u(z) = u(e^{\tau B}z)e^{h(e^{\tau B}z)-h(z)}.
\]
We can then compute the norm equivalence using simple operations and the decomposition \eqref{eq.decompose.Phi} of $\Phi$:
\[
	\begin{aligned}
	\|\exp(\tau \mathcal{Q}_B)u\|_\Phi^2 &= \int |u(e^{\tau B})|^2e^{2\Re h(e^{\tau B}z) -2 \Re h(z) - 2(\frac{1}{2}|Gz|^2 - \Re h(z))}\,dL(z)
	\\ &= \int |u(e^{\tau B} z)|^2 e^{-2(\frac{1}{2}|Gz|^2 - \Re h(e^{\tau B} z))}\,dL(z)
	\\ &= e^{-2\Re \tau \opnm{Tr} B}\int |u(z)|^2 e^{-2(\frac{1}{2}|Ge^{-\tau B}z|^2 - \Re h(z))}\,dL(z).
	\end{aligned}
\]
The definition \eqref{eq.def.Phi.tau.B} of the weight $\Phi^{(\delta), B}$ can be read off from the exponential factor.

Alternately, using the definitions \eqref{eq.def.V} and \eqref{eq.def.W}, we have
\[
	\mathcal{W}_{-h}\mathcal{V}_{e^{-\tau B}}\mathcal{W}_h \exp(\tau \mathcal{Q}_B) u = e^{-\Re \tau \opnm{Tr} B}u
\]
with the image lying in the space $H_{\Phi^{(\tau),B}}$.
\end{proof}

We can now see that the embedding \eqref{eq.def.iota} is not only compact but even has exponentially decaying singular values, so long as \eqref{eq.embedding.inequality} holds strictly on $\{|z|=1\}$.  We here say that a compact operator $A$ has exponentially decaying singular values $\{s_j(A)\}_{j=1}^\infty$ if there exists $C > 0$ such that
\begin{equation}\label{eq.exp.decaying}
	s_j(A) \leq C\exp\left(-\frac{j^{1/n}}{C}\right).
\end{equation}
The dependence on the dimension is unavoidable, since the estimate is sharp for $\exp(-Q_0)$ with $Q_0$ from \eqref{eq.def.h.o}. Note that this implies that $\iota$ is in any Schatten class $\mathfrak{S}_p$, $p \in (0, \infty)$.

\begin{corollary}\label{cor.sing.vals.general}
Let $\Phi_j, j = 1,2$ both satisfy \eqref{eq.Phi.assumptions} and suppose that
\[
	\Phi_2(z) > \Phi_1(z), \quad \forall |z| = 1.
\]
Then the embedding $\iota:H_{\Phi_1}\to H_{\Phi_2}$ used in \eqref{eq.def.iota} is compact and has exponentially decaying singular values in the sense of \eqref{eq.exp.decaying}.
\end{corollary}

\begin{proof}
By Proposition \ref{prop.h.o.semigroup}, it is easy to see that there exists $\delta > 0$ such that
\[
	\exp(\delta P_0)\iota:H_{\Phi_1} \to H_{\Phi_2}
\]
is bounded, with $P_0:H_{\Phi_2}\to H_{\Phi_2}$ defined as in \eqref{eq.def.P0} (and depending on the weight $\Phi_2$).  Therefore
\[
	\iota = \exp(-\delta P_0)\exp(\delta P_0)\iota
\]
expresses $\iota$ as the product of a bounded operator from $H_{\Phi_1}$ to $H_{\Phi_2}$ and a compact positive self-adjoint operator on $H_{\Phi_2}$ with
\[
	\opnm{Spec}(\exp(-\delta P_0)) = \{e^{-\delta |\alpha|} \::\: \alpha \in \Bbb{N}^n\},
\]
where the equality includes repetition according to multiplicity.

Since
\[
	\# \left\{\alpha \::\: |\alpha| \leq N\right\} = \frac{1}{n!}N^n(1+\BigO(N^{-1}))
\]
as $N \to \infty$, the singular values of $\exp(-\delta P_0)$ decay exponentially in the sense of \eqref{eq.exp.decaying}.  Since $s_j(AB) \leq s_j(A)\|B\|$ for any operators $A,B$ for which $A$ is compact and $B$ is bounded, this completes the proof of the corollary.
\end{proof}

We turn to the extension of operators on $H_\Phi$ given by changes of variables from their restriction to the space of polynomials. This is motivated by the fact that the space of polynomials appears as the span of the generalized eigenfunctions of $P$, and at least on any element of the span of the generalized eigenfunctions of $P$, the definition of $\exp(\tau P)$ may be realized as a matrix exponential. Since the solution to the evolution equation is unique in the space of holomorphic functions, this realization must agree with the definition in Proposition \ref{prop.solve.exptP}.

We recall that an unbounded operator $A$ acting on a Hilbert space $\mathcal{H}$ with domain $\mathcal{D}_A$ has the set $K \subset \mathcal{H}$ as a core if the closure of the graph
\[
	\{(x, Ax)\::\: x \in K\}
\]
in $\mathcal{H}\times \mathcal{H}$ is 
\[
	\{(x, Ax)\::\: x \in \mathcal{D}_A\}.
\]
Note that this implies that $A$ is a closed operator when equipped with the domain $\mathcal{D}_A$.

Let $F \in GL_n(\Bbb{C})$ be an invertible matrix and define
\begin{equation}\label{eq.def.C.F}
	C_F u(z) = u(Fz), \quad z \in \Bbb{C}^n,
\end{equation}
considered as acting on $H_\Phi$ for $\Phi$ obeying \eqref{eq.Phi.assumptions}. Its maximal domain is
\begin{equation}\label{eq.def.D.F}
	\mathcal{D}_F = \{u\in H_\Phi \::\: C_F u \in H_\Phi\},
\end{equation}
which is closed with respect to the graph norm given by the inner product
\begin{equation}\label{eq.def.i.p.F}
	\begin{aligned}
	\langle u, v\rangle_F &= \langle u, v\rangle_\Phi + \langle C_F u, C_F v\rangle_\Phi
	&= \langle u, v\rangle_\Phi + |\det F|^{-2}\langle u, v\rangle_{\Phi(F\cdot)},
	\end{aligned}
\end{equation}
by the same reasoning as in Proposition \ref{prop.solve.exptP}.

We start with a lemma on the strong continuity of bounded change of variables operators considered as functions depending on the matrix $F$.

\begin{lemma}\label{lem.strong.continuity}
Assume that $\Phi$ obeys \eqref{eq.Phi.assumptions} and recall the definition \eqref{eq.def.C.F}. Let $\{F_k\}_{k\in \Bbb{N}}$ be a sequence in $GL_n(\Bbb{C})$ converging to $F \in GL_n(\Bbb{C})$.  Assume furthermore that $\Phi(F_k z) \leq \Phi(z)$ for all $z \in \Bbb{C}$ and all $k \in \Bbb{N}$.  Then $C_{F_k}$ converges to $C_F$ in the strong operator topology on $\mathcal{L}(H_\Phi)$.
\end{lemma}

\begin{proof}
Because $F_k \to F$, we have that also $\Phi(Fz) \leq \Phi(z)$ for all $z \in \Bbb{C}^n$.  The same change of variables as \eqref{eq.norm.change.vars} gives that
\[
	\|C_{F_k}u\|_{\Phi}^2 = \int_{\Bbb{C}^n} e^{-2\Re \opnm{Tr} F_k}|u(z)|^2 e^{-2\Phi(F_k^{-1}z)}\,dL(z)
\]
(and similarly for $C_F$).
Since $\Phi(F_k^{-1}z)\geq \Phi(z)$ and $\Phi(F^{-1}z)\geq \Phi(z)$ for all $z \in \Bbb{C}^n$, this shows that $C_{F_k}, C_F \in \mathcal{L}(H_\Phi)$. Furthermore, we may dominate the integrand by 
\[
	e^{-2\Re \opnm{Tr} F_k}|u(z)|^2 e^{-2\Phi(F_k^{-1}z)} \leq A|u|^2 e^{-2\Phi(z)}
\]
uniformly in $k$ for some $A > 0$.  Therefore, by the dominated convergence theorem,
\[
	\forall u \in H_\Phi, \quad \lim_{k\to\infty} \|C_{F_k}u\|_\Phi = \|C_F u\|_\Phi.
\]

Since $F_kz\to Fz$ as $k\to \infty$ for each $z \in \Bbb{C}^n$, we see that $C_{F_k}u \to C_F u$ pointwise, which means $\langle C_{F_k}u, k_w\rangle_\Phi \to \langle C_F u, k_w\rangle_\Phi$ for $k_w$ any reproducing kernel at $w\in\Bbb{C}$ for $H_\Phi$. Since the sequence $\{C_{F_k}u\}_{k \in \Bbb{N}}$ is bounded in $H_\Phi$ for each $u$ and the span of reproducing kernels is dense, this means that $C_{F_k}u\to C_F u$ weakly. Therefore, by the Banach-Steinhaus theorem, $C_{F_k} \to C_F$ strongly.
\end{proof}

We may then prove that the polynomials form a core for every operator on $H_\Phi$ given by an invertible linear change of variables, whether or not it is bounded.

\begin{proposition}\label{prop.polys.core}
Let $F \in GL_n(\Bbb{C})$ and let $\Phi$ obey \eqref{eq.Phi.assumptions}. The polynomials form a core on $H_\Phi$ for $C_F$, defined in \eqref{eq.def.C.F}, on its maximal domain $\mathcal{D}_F$, defined in \eqref{eq.def.D.F}.
\end{proposition}

\begin{proof}
We begin by considering the dilations
\[
	T_\zeta u(z) = u(\zeta z).
\]
Let
\[
	\Omega = \{\zeta \in \Bbb{C} \backslash \{0\} \::\: |z| = 1 \implies \Phi(z) > \Phi(\zeta z)\},
\]
and note that $\Omega$ is an open subset of $\Bbb{C}\backslash \{0\}$. By Lemma \ref{lem.strong.continuity} we see that, on $\overline{\Omega}\backslash \{0\}$, the the map $\zeta \mapsto T_\zeta$ gives a strongly continuous family of operators from $H_\Phi$ to $H_\Phi$.  It is furthermore clear that $T_\zeta u$ is a holomorphic function of $\zeta \in \Omega$, and by strict convexity of $\Phi$, it is clear that $\Omega$ contains the interval $(0,1)$.

Recall the definition of $\Psi$ from \eqref{eq.def.Psi}. By strict convexity of $\Phi$, there exists some $C_0 > 0$ such that
\begin{equation}\label{eq.Phi.Psi.inequality}
	\frac{1}{C_0}\Psi(z) \leq \Phi(z) \leq C_0 \Psi(z), \quad \forall z \in \Bbb{C}^n;
\end{equation}
since $F$ is invertible, we may take $C_0$ sufficiently large to also ensure that
\begin{equation}\label{eq.PhiF.Psi.inequality}
	\frac{1}{C_0}\Psi(z) \leq \Phi(F^{-1}z) \leq C_0 \Psi(z), \quad \forall z \in \Bbb{C}^n.
\end{equation}
Therefore, so long as $0 < |\zeta| < \frac{1}{C_0}$, we have for $|z| = 1$ that
\[
	\Phi(z) \geq \frac{1}{C_0} \Psi(z) > C_0 \Psi(\zeta z) \geq \Phi(\zeta z),
\]
proving that $\Omega$ contains the punctured neighborhood $\{0 < |\zeta| < 1/C_0\}$.

What is more, when $0 < |\zeta| < 1/C_0$ and $u \in \mathcal{D}_F$, we have that both $T_\zeta u$ and $T_\zeta C_F u$ are in $H_{\Psi/C_0}$, a space in which monomials form an orthogonal basis; see \eqref{eq.def.e.alpha}.  Therefore
\begin{equation}\label{eq.T.zeta.polys}
	T_\zeta u = \lim_{N\to \infty} \sum_{|\alpha| \leq N} \frac{\langle T_\zeta u, z^\alpha\rangle_{\Psi/C_0}}{\|z^\alpha\|_{\Psi/C_0}} z^\alpha
\end{equation}
as a limit in $H_{\Psi/C_0}$.  Since convergence in $H_{\Psi/C_0}$ implies convergence in $H_\Phi$ by \eqref{eq.Phi.Psi.inequality} and in $H_{\Phi(F^{-1}\cdot)}$ by \eqref{eq.PhiF.Psi.inequality}, we see that \eqref{eq.T.zeta.polys} also holds as a limit with respect to the norm $\|\cdot \|_F$ given by \eqref{eq.def.i.p.F} for $0 < |\zeta| < 1/C_0$.

Therefore if $u \in \mathcal{D}_F$ is orthogonal to every polynomial with respect to the inner product \eqref{eq.def.i.p.F}, we have that $\langle T_\zeta u, u\rangle_F$ is a holomorphic function for $\zeta \in \Omega$, continuous on $\overline{\Omega}$, and for $0 < |\zeta| < 1/C_0$,
\[
	\langle T_\zeta u, u\rangle_F = \lim_{N\to\infty} \left\langle \sum_{|\alpha| \leq N} \frac{\langle T_\zeta u, z^\alpha\rangle_{\Psi/C_0}}{\|z^\alpha\|_{\Psi/C_0}} z^\alpha, u \right\rangle_F = 0.
\]
This shows that the function vanishes identically on $\Omega$, and upon taking the limit as $\zeta \to 1$ from within $\Omega$, we see that $u = 0$.  This proves that $\{(p, C_F p)\::\: p \textnormal{ a polynomial}\}$ is dense in $\{(u, C_F u)\::\: u \in \mathcal{D}_F\}$ as subsets of $H_\Phi \times H_\Phi$, which suffices to prove the proposition.
\end{proof}

\subsection{Identification of boundedness and compactness}

We proceed to the following precise description of the set of $\tau \in \Bbb{C}$ for which the map $\exp(\tau P)$ is bounded or compact.

\begin{theorem}\label{thm.boundedness.0}
Let the matrix $M$, the weight $\Phi$, and the operators $P$ and $\exp(\tau P)$ be as in Proposition \ref{prop.solve.exptP}. Then $\exp(\tau P)$ is bounded if and only if
\begin{equation}\label{eq.Phi.increases}
	\Phi(e^{-\tau M}z) \geq \Phi(z), \quad \forall z \in \Bbb{C}^n
\end{equation}
and is compact if and only if the inequality is strict on $\{|z| = 1\}$, in which case $\exp(\tau P)$ has exponentially decaying singular values in the sense of \eqref{eq.exp.decaying}.  On the set of $\tau \in \Bbb{C}$ for which this inequality holds, the family of operators $\exp(\tau P)$ is strongly continuous in $\tau$ and obeys
\begin{equation}\label{eq.norm.bound}
	\|\exp(\tau P)\|\leq e^{-\Re \tau \opnm{Tr}M}.
\end{equation}
\end{theorem}

\begin{proof}
The norm bound follows immediately from Proposition \ref{prop.solve.exptP}.  The characterization of boundedness and compactness is the special case $\delta = 0$ of the following more general theorem, which places the image of $\exp(\tau P)$ within the family of spaces $\{\exp(\delta P_0)H_\Phi\}_{\delta \in \Bbb{R}}$.  That the family of operators, where bounded, is strongly continuous in $\tau$ follows from Lemma \ref{lem.strong.continuity}.
\end{proof}

We continue with a more general theorem relating the boundedness properties of $\exp(\tau P)$ with those of $\exp(\delta P_0)$ for $P_0$ from \eqref{eq.def.P0}.  While this is natural and very useful to prove properties such as compactness, our principal interest is in the question of boundedness. Therefore, most results throughout may be read for $\delta = 0$, as done in Theorem \ref{thm.boundedness.0} above.

\begin{theorem}\label{thm.boundedness.delta}
Let the matrix $M$, the weight $\Phi$, and the operators $P$ and $\exp(\tau P)$ be as in Proposition \ref{prop.solve.exptP}. For $\Phi^{(\delta)}$ as in \eqref{eq.def.Phi.delta}, let $\delta_0 = \delta_0(\tau) \in \Bbb{R}$ be defined by
\begin{equation}\label{eq.def.delta.0}
	\delta_0 = \sup\{ \delta \in \Bbb{R} \::\: \forall z \in \Bbb{C}^n,~~ \Phi^{(\delta)}(e^{-\tau M}z) \geq \Phi(z)\}
\end{equation}
Then the operator
\begin{equation}\label{eq.compose.bounded.l}
	\exp(\delta P_0)\exp(\tau P),
\end{equation}
with $P_0$ as in \eqref{eq.def.P0}, is bounded on $H_\Phi$ if and only if $\delta \leq \delta_0$ and is compact if and only if $\delta < \delta_0$, in which case it has exponentially decaying singular values in the sense of \eqref{eq.exp.decaying}.
\end{theorem}

\begin{proof}
From Propositions \ref{prop.solve.exptP} and \ref{prop.h.o.semigroup} we see that
\[
	\begin{aligned}
	\|\exp(\delta P_0)\exp(\tau P)u\|_\Phi &= e^{-\delta n}\|u(e^{\tau M}z)\|_{\Phi^{(\delta)}}
	\\ &= e^{-\delta n -\Re(\tau \opnm{Tr} M)}\|u\|_{\Phi^{(\delta)}(e^{-\tau M}\cdot)}.
	\end{aligned}
\]
Therefore the operator \eqref{eq.compose.bounded.l} is, up to a unitary transformation, a factor times the embedding from $H_\Phi$ to $H_{\Phi^{(\delta)}(e^{-\tau M}\cdot)}$. This embedding is bounded if and only if 
\begin{equation}\label{eq.Phi.increases.delta}
	\Phi^{(\delta)}(e^{-\tau M}z) \geq \Phi(z)
\end{equation}
for all $z \in \Bbb{C}^n$ by Proposition \ref{prop.embedding}, which also gives that the inequality must be strict on $\{|z| = 1\}$ in order for the map to be compact. On the other hand, the map is compact with decaying singular values in the sense of \eqref{eq.exp.decaying} if the inequality holds strictly on $\{|z| = 1\}$ by Corollary \ref{cor.sing.vals.general}.

For $\tau \in \Bbb{C}$ and $z\in\Bbb{C}^n$ fixed, $\Phi^{(\delta)}(e^{-\tau M}z)$ is a decreasing function of $\delta$ which tends to $-\Re h(e^{-\tau M}z)$ as $\delta \to \infty$ and to $\infty$ as $\delta \to -\infty$.  As a harmonic function, $-\Re h(e^{-\tau M})$ cannot be positive definite, so the set defining $\delta_0$ must be bounded from above since $\Phi^{(\delta)}(e^{-\tau M}z)$ fails to dominate the strictly convex function $\Phi(z)$ for $\delta$ sufficiently large.  (See also Proposition \ref{prop.maximal.smoothing}.) Since $\Phi^{(\delta)}(e^{-\tau M}z) \to \infty$ as $\delta \to -\infty$, the set defining $\delta_0$ is bounded from below.  Therefore $\delta_0 \in \Bbb{R}$, and from the fact that $\Phi^{(\delta)}(e^{-\tau M}z)$ is decreasing and continuous in $\delta$ we have that \eqref{eq.Phi.increases.delta} holds for $\delta \leq \delta_0$ and holds strictly on $\{|z|=1\}$ for $\delta < \delta_0$, which suffices to identify when the operator \eqref{eq.compose.bounded.l} is bounded or 
compact with exponentially decaying singular values.
\end{proof}

\begin{remark*}
Continuing to use certain standard simple unitary transformations, we may make explicit the unitary transformation relating $\exp(\delta P_0)\exp(\tau P)$ to the (possibly unbounded) embedding from $H_\Phi$ to $H_{\Phi^{(\delta)}(e^{-\tau M}\cdot)}.$ Using the unitary transformation \eqref{eq.def.U} along with Propositions \ref{prop.solve.exptP} and \ref{prop.h.o.semigroup}, we see that
\[
	\begin{aligned}
	\exp(\delta P_0)\exp(\tau P)u(z) &= \mathcal{U}\exp(\delta z\cdot \partial_z)\mathcal{U}^* u(e^{\tau M}z)
	\\ &= \mathcal{U}u(e^\delta e^{\tau M}G^{-1}z)e^{h(e^{\delta}G^{-1}z)}
	\\ &= u(e^{\delta}e^{\tau M}z)e^{h(e^{\delta}z) - h(z)}.
	\end{aligned}
\]
(What is more, we see that $\exp(\delta P_0)$ is particularly convenient precisely because $e^\delta$ commutes with all matrices.) We may then check using \eqref{eq.def.V} and \eqref{eq.def.W} that, with $\iota u = u$ the natural embedding,
\[
	\mathcal{V}_{e^{-\tau M}}\mathcal{W}_{-h}\mathcal{V}_{e^{-\delta}} \mathcal{W}_{h}  \exp(\delta P_0)\exp(\tau P) 
	= e^{-\delta n - \Re(\tau \opnm{Tr} M)}\iota:H_\Phi \to H_{\Phi^{(\delta)}(e^{-\tau M}\cdot)}.
\]
\end{remark*}

We next consider the question of when the solution operator $\exp(-tP)$ is bounded for all $t \geq 0$.  For these operators on Fock spaces, the question is reduced to the question of positivity of a real quadratic form which corresponds to the classical notion of the real part of the symbol of a differential operator (see Remark \ref{rem.Fock.phase.space}).

\begin{theorem}\label{thm.boundedness.Theta}
Let the matrix $M$, the weight $\Phi$, and the operators $P$ and $\exp(\tau P)$ be as in Proposition \ref{prop.solve.exptP}. Then, $\exp(-tP)$ is bounded on $H_\Phi$ for all $t \in [0, \infty)$ if and only if
\begin{equation}\label{eq.ellipticity}
	\Theta(z) \geq 0, \quad \forall z \in \Bbb{C}^n,
\end{equation}
for
\begin{equation}\label{eq.def.Theta}
	\Theta(z) = 2\Re P\Phi(z) = 2\Re\left((Mz)\cdot \Phi'_z(z)\right).
\end{equation}
Moreover, using the decomposition \eqref{eq.decompose.Phi} and with $\delta_0$ defined in \eqref{eq.def.delta.0},
\begin{equation}\label{eq.delta.0.deriv}
	\delta_0(-t) = t\inf_{|z|=1}\frac{\Theta(z)}{|Gz|^2} + \BigO(t^2), \quad t \to 0^+.
\end{equation}
\end{theorem}

\begin{proof}
Since $\Phi$ is real-valued, $\Phi'_{\bar{z}} = \overline{\Phi'_z}$, so we compute
\begin{equation}\label{eq.deriv.first}
\begin{aligned}
	\frac{d}{dt}\Phi(e^{tM}z) &= Me^{tM}z\cdot \Phi'_z(e^{tM}z) + \overline{Me^{tM}z}\cdot \Phi'_{\bar{z}}(e^{tM}z)
	\\ &= \Theta(e^{tM}z).
\end{aligned}
\end{equation}
If $\Theta(z_0) < 0$ for some $z_0 \in \Bbb{C}^n$, then \eqref{eq.Phi.increases} fails at $z_0$ for $\tau = -t$ and $t > 0$ small.  If, on the other hand, \eqref{eq.ellipticity} holds, then $\Phi(e^{tM}z)$ is nondecreasing in $t$ for all $z \in \Bbb{C}$, so \eqref{eq.Phi.increases} holds for $\tau = -t$ and any $t \geq 0$. Therefore \eqref{eq.Phi.increases} holds for all $\tau = -t$ with $t > 0$ if and only if \eqref{eq.ellipticity} holds.

From \eqref{eq.deriv.first} and a direct calculation we have that
\[
	\left.\frac{\partial}{\partial t}\Phi^{(\delta)}(e^{tM}z)\right|_{t = \delta = 0} = \Theta(z) \textnormal{  and  } \left.\frac{\partial}{\partial \delta}\Phi^{(\delta)}(e^{tM}z)\right|_{t = \delta = 0} = -|Gz|^2.
\]
Using the fact that $\Phi$ is quadratic along with the Taylor expansion for $e^{2\delta}$ and $e^{tM}$, we estimate
\begin{equation}\label{eq.Theta.degree.1.expand}
	\Phi^{(\delta)}(e^{tM}z) = \Phi(z) + t\Theta(z) - \delta|Gz|^2 + \BigO((\delta^2+t^2)|z|^2)
\end{equation}
for $\delta, t$ small and with error bound uniform for $z \in \Bbb{C}^n$.  Let
\[
	R = \inf_{|z|=1} \frac{\Theta(z)}{|Gz|^2}.
\]
If $\delta = R t - Ct^2$, then
\[
	\Phi^{(\delta)}(e^{tM}z) =  \Phi(z) + C|Gz|^2 t^2 + (\Theta(z)-R|Gz|^2)t + \BigO((t^2 + \delta^2)|z|^2).
\]
By the definition of $R$, the coefficient of $t$ is positive and $\delta^2 = \BigO(t^2)$.  Using also that $|Gz|^2/|z|^2$ is bounded away from zero on $\{|z|=1\}$ because $G$ is invertible, if $C$ is sufficiently large and $t$ is sufficiently small and positive we have that \eqref{eq.Phi.increases.delta} holds with $\tau = -t$. 

On the other hand, by continuity we may select $z_0 \in \Bbb{C}^n$ with $|z_0| = 1$ and where $\Theta(z_0)/|Gz_0|^2 = R$. Taking instead $\delta = R t + Ct^2$ gives
\[
	\Phi^{(\delta)}(e^{tM}z_0) =  \Phi(z_0) - C|Gz_0|^2 t^2 + \BigO(t^2 + \delta^2),
\]
so \eqref{eq.Phi.increases.delta} with $\tau = -t$ fails if $C$ is sufficiently large and $t$ is sufficiently small and positive.  Using again that $\Phi^{(\delta)}(e^{tM}z)$ is decreasing in $\delta$, we conclude that, for some $C$ and for $t$ sufficiently small and positive,
\[
	\delta_0(-t) \in \left[Rt-Ct^2, Rt+Ct^2\right],
\]
which completes the proof of the theorem.
\end{proof}

\begin{remark*}
One could also reverse the order of $P_0$ and $P$ in Theorem \ref{thm.boundedness.delta} and analyze the operator
\[
	\exp(\tau P)\exp(\delta P_0)u(z) = u(e^\delta e^{\tau M} z)e^{h(e^\delta e^{\tau M}z)-h(e^{\tau M}z)}.
\]
We may check boundedness for this operator by using that
\[
	\mathcal{W}_{-h}\mathcal{V}_{e^{-\delta}}\mathcal{W}_h \mathcal{V}_{e^{-\tau M}}\exp(\tau P)\exp(\delta P_0) = e^{-n\delta - \Re \tau \opnm{Tr} M}\iota : H_\Phi \to H_{\tilde{\Phi}}
\]
for 
\[
	\tilde{\Phi}(z) = \tilde{\Phi}(z; \delta, \tau, M) = e^{-2\delta} \Phi(e^{-\tau M}z) + (e^{-2\delta}-1)\Re h(z).
\]
Therefore $\exp(\tau P)\exp(\delta P_0)$ is bounded if and only if $\tilde{\Phi}\geq \Phi$.

This weight $\tilde{\Phi}$ seems less convenient than $\Phi^{(\delta)}(e^{-\tau M}z)$, which is in part explained by the way in which the change of variables associated with $\exp(\tau P)$ changes the harmonic part $\Re h(z)$ of the weight.  Nonetheless, the same reasoning can show that if
\[
	\tilde{\delta}_0(\tau) = \sup \{\delta \::\: \exp(\tau P)\exp(\delta P_0) \in \mathcal{L}(H_\Phi)\},
\]
then
\[
	\tilde{\delta}_0(-t) = t\inf_{|z|=1}\frac{\Theta(z)}{|Gz|^2} + \BigO(t^2), \quad t \to 0^+,
\]
similarly to \eqref{eq.delta.0.deriv}.
\end{remark*}

We now show that the span of the generalized eigenfunctions of $P$ form a core for $\exp(\tau P)$ by identifying those eigenfunctions and observing that their span is the set of polynomials.

To fix notation, let $\tilde{G}$ be an invertible matrix such that $\tilde{G}^{-1}M\tilde{G}$ is in Jordan normal form. Let $\lambda_1,\dots,\lambda_n$ be the spectrum of $M$, repeated for algebraic multiplicity, so that
\begin{equation}\label{eq.Jordan}
	\tilde{M} = \tilde{G}^{-1}M\tilde{G} = \left(\begin{array}{cccc}
		\lambda_1 & \gamma_1 & 0 & 0
		\\ 0 & \ddots & \ddots & 0
		\\ 0 & 0 & \lambda_{n-1} & \gamma_{n-1}
		\\ 0 & 0 & 0 & \lambda_n
		\end{array}\right)
\end{equation}
for $\gamma_j \in \{0,1\}$ for all $j=1,\dots,n-1$. For $e_j$ the standard basis vector with $1$ in the $j$-th position and $0$ elsewhere, let $r_j$ be the order of the generalized eigenvector $e_j$ of $\tilde{M}$, meaning that
\begin{equation}\label{eq.def.rj}
	r_j = \min\{k \in \Bbb{N}^* \::\: (\tilde{M}-\lambda_j)^k e_j = 0\}.
\end{equation}
We define the complementary notion of the distance to the end of the Jordan block:
\begin{equation}\label{eq.def.rj.tilde}
	\tilde{r}_j = \max\{k \in \Bbb{N} \::\: e_j \in (\tilde{M}-\lambda_j)^k(\Bbb{C}^n)\},
\end{equation}
with the usual convention that $(M-\lambda_j)^0 = 1$, the identity matrix.  (These notions do not depend essentially on the Jordan normal form, so long as $e_j$ is replaced by a generalized eigenvector and $\lambda_j$ is replaced by the corresponding eigenvalue.) The definition of $\tilde{r}_j$ becomes useful since the action of $M-\lambda_j$ on a monomial is in the opposite direction from the action of $M-\lambda_j$ on the $e_j$, as we will see shortly.

In the Jordan normal form case, we note that $r_j + \tilde{r}_j$ is the size of the Jordan block containing $e_j$ and that $\gamma_j = 1$ implies that $\lambda_{j+1} = \lambda_j$.  Furthermore, $\gamma_j = 1$ if and only if $r_{j+1} > 1$ if and only if $\tilde{r}_j \geq 1$, and in this case $r_{j+1} = r_j + 1$ and $\tilde{r}_j = \tilde{r}_{j+1} + 1$.

In the following theorem, we identify the complete set of eigenfunctions of $P$, which can be traced back to \cite[Sec.~3]{Sj1974}, and show that the span of these eigenfunctions forms a core for $\exp(\tau P)$, which is novel and follows directly from Proposition \ref{prop.polys.core}.

\begin{theorem}\label{thm.eigen.core}
Let the matrix $M$, the weight $\Phi$, and the operators $P$ and $\exp(\tau P)$ be as in Proposition \ref{prop.solve.exptP}. Furthermore let the matrix $\tilde{G}$ be such that $\tilde{G}^{-1}M\tilde{G}$ is in Jordan normal form; also let the eigenvalues $\{\lambda_j\}_{j=1}^n$, repeated for algebraic multiplicity, and the orders $\{r_j\}_{j=1}^n$ and $\{\tilde{r}_j\}_{j=1}^n$ be as above.  Then
\[
	\{(\tilde{G}^{-1}z)^\alpha\}_{\alpha \in \Bbb{N}^n}
\]
form a complete set in $H_\Phi$ of generalized eigenvectors of $P$ with eigenvalues
\begin{equation}\label{eq.def.lambda.alpha}
	\lambda_\alpha = \sum_{j=1}^n \lambda_j \alpha_j
\end{equation}
and orders
\[
	r_\alpha = 1+\sum_{j=1}^n \tilde{r}_j\alpha_j.
\]
The span of these eigenfunctions (that is, the polynomials) form a core of $\exp(\tau P)$ considered on its maximal domain
\[
	\mathcal{D}_{\exp(\tau P)} = \{u\in H_\Phi \::\: \|u(e^{\tau M}\cdot)\|_{\Phi} < \infty \} = \{u \in H_\Phi \::\: \|u\|_{\Phi_{-\tau}} < \infty\}.
\]
\end{theorem}

\begin{proof}
By conjugating $P$ by $\mathcal{V}_{\tilde{G}}$ as in \eqref{eq.def.V}, it suffices to consider $M$ already in Jordan normal form as in \eqref{eq.Jordan}.  Then
\[
	P = \sum_{j=1}^n \lambda_j z_j \partial_{z_j} + \sum_{j=1}^{n-1} \gamma_j z_{j+1}\partial_{z_j},
\]
so
\[
	\begin{aligned}
	(P-\lambda_\alpha)z^\alpha &= \sum_{j=1}^{n-1} \gamma_j\alpha_j z^{\alpha-e_j+e_{j+1}}
	\\ &= \sum_{j\::\: \alpha_j \neq 0, \tilde{r}_j \neq 0} \alpha_j z^{\alpha_j +e_{j+1} -e_j},
	\end{aligned}
\]
using that $\gamma_j = 1$ if and only if $\tilde{r}_j \neq 0$.

We see that $(P-\lambda_\alpha)z^\alpha = 0$ if and only if $r_\alpha = 1$ and that otherwise $(P-\lambda_\alpha)z^\alpha$ is a linear combination, with coefficients in $\Bbb{N}^*$, of those monomials $z^{\alpha + e_{j+1}-e_j}$ for which
\[
	r_{\alpha + e_{j+1} - e_j} - r_\alpha = \tilde{r}_{j+1} - \tilde{r}_j = -1.
\]
When repeating this expansion, there can be no cancellation since the coefficients at each stage are positive, and we conclude by an induction argument that $r_\alpha$ is the minimal $N$ for which $(P-\lambda_\alpha)^N z^\alpha = 0$.  Therefore, for a combinatorial constant $C_\alpha \in \Bbb{N}^*$ which we do not compute here,
\begin{equation}\label{eq.gen.eigen.vanishing}
	(P-\lambda_\alpha)^{r_\alpha - 1}z^\alpha = C_\alpha z^{\tilde{\alpha}}
\end{equation}
for $\tilde{\alpha}$ the multi-index formed by pushing each $\alpha_j$ to the end of the corresponding Jordan block:
\[
	\tilde{\alpha}_j = \left\{\begin{array}{ll} 0, & \tilde{r}_j \neq 0, \\ \sum_{k=0}^{r_j-1} \alpha_j, & \tilde{r}_j = 0.\end{array}\right.
\]

For $M$ already in Jordan normal form, it is automatic that the span of the monomials $\{z^\alpha\}_{\alpha \in \Bbb{N}^n}$ is the set of polynomials.  Conjugation with $\mathcal{V}_{\tilde{G}}$ does not change this, since $\mathcal{V}_{\tilde{G}}$ is an isomorphism on the set of polynomials (or even on each set of homogeneous polynomials of fixed degree). For the claim that the polynomials form a complete set in $H_\Phi$, see \cite[Rem.~2.5]{Vi2013}, which relies essentially on \cite[Lem.~3.12]{Sj1974}.

That the polynomials form a core for $\exp(\tau P)$ is the content of Proposition \ref{prop.polys.core}.
\end{proof}

\subsection{Consequences}

We continue by deducing several consequences of our results on the operators $\exp(\tau P)$. These include necessary conditions for boundedness of $\exp(\tau P)$ based on the spectrum of $M$, a precise description of those $\tau \in \Bbb{C}$ for which $\exp(\tau P)$ is bounded as $|\tau| \to \infty$, a relationship between the Hermitian part of $\Phi$ and the decay of $\exp(-t P)u$ as $t \to \infty$, an analysis of the fragile case when $\opnm{Spec} M \cap i\Bbb{R} \neq \varnothing$, and an extension of the analysis whereby $P$ may essentially absorb linear terms with minimal changes to the character of the family of solution operators.

\begin{proposition}\label{prop.expansive}
Let the matrix $M$, the weight $\Phi$, and the operators $P$ and $\exp(\tau P)$ be as in Proposition \ref{prop.solve.exptP}. Let $\delta_0$ be as in \eqref{eq.def.delta.0}, and let the matrix $G$ be as in the decomposition \eqref{eq.decompose.Phi} of $\Phi$. Then
\begin{equation}\label{eq.expansive.bound}
	\delta_0 \leq -\log\|Ge^{\tau M}G^{-1}\|.
\end{equation}
In particular,
\begin{equation}\label{eq.expansive.Spec}
	\opnm{Spec}(\tau M) \subset \{\Re \lambda \leq -\delta_0\}.
\end{equation}

In addition, if $\exp(-tP)$ is bounded for all $t \in [0, \infty)$, then
\[
	\Re \langle GMG^{-1}z, z\rangle \geq 0, \quad \forall z \in \Bbb{C}^n.
\]
\end{proposition}

\begin{remark*}
As a special case, we have that if $\exp(\tau P)$ is bounded, then $Ge^{\tau M}G^{-1}$ is a contraction (in the sense that its norm is at most one). In particular, $\exp(\tau P)$ can only be bounded if $\opnm{Spec}(\tau M) \subset \{\Re \lambda \leq 0\}$, as may be seen by testing $Ge^{\tau M}G^{-1}$ on $G$ applied to each eigenvector of $\tau M$.

It is also helpful to make a comparison with the case of a normal operator: if $A$ were a normal operator on a Hilbert space $\mathcal{H}$ with $\opnm{Spec} A$ equal to the eigenvalues of $P$ in \eqref{eq.def.lambda.alpha}, then \eqref{eq.expansive.Spec} with $\delta_0 = 0$ would be an exact description of the boundedness of the solution operator for $A$ in the sense that
\[
	\{\tau \::\: \Re (\tau \lambda_j) \leq 0, ~~j = 1,\dots, n\} = \{\tau \::\: e^{\tau A} \in \mathcal{L}(\mathcal{H})\}.
\]

Finally, note that as a special case of Theorem \ref{thm.bounded.tau}, we have a partial converse: if all eigenvalues of $M$ have strictly positive real parts, then $\exp(-tP)$ is bounded for all $t$ real and sufficiently large.
\end{remark*}

\begin{proof}
It is clear that the Hermitian part, defined in Lemma \ref{lem.decomposition}, of $\Phi^{(\delta_0)}(e^{-\tau M}z) - \Phi(z)$ when $\Phi$ is written using \eqref{eq.decompose.Phi} is
\begin{equation}\label{eq.expansive.proof}
	\frac{e^{-2\delta_0}}{2}|Ge^{-\tau M}z|^2 - \frac{1}{2}|Gz|^2.
\end{equation}
Recall from Lemma \ref{lem.decomposition} and Theorem \ref{thm.boundedness.delta} that this quantity must be nonnegative. Setting $y = Ge^{-\tau M}z$ gives
\[
	\frac{1}{2}\left(e^{-2\delta_0} |y|^2 - |Ge^{\tau M}G^{-1} y|^2\right) \geq 0,
\]
from which 
\[
	\|Ge^{\tau M}G^{-1}\|\leq e^{-\delta_0}.
\]
The estimate \eqref{eq.expansive.bound} follows.

If $\tau M v = \lambda v$ for $v \neq 0$, then
\[
	-\log \frac{|Ge^{\tau M}G^{-1} Gv|}{|Gv|} = -\Re \lambda.
\]
Therefore, by \eqref{eq.expansive.bound}, we see that $\Re \lambda \leq -\delta_0$ for all $\lambda \in \opnm{Spec}(\tau M)$, proving \eqref{eq.expansive.Spec}.

Similarly, the second claim follows from the calculation
\begin{equation}\label{eq.decompose.Theta}
	\begin{aligned}
	\Theta(z) &= 2\Re Mz \cdot \partial_z\left(\frac{1}{2}z\cdot \overline{G^*Gz} - \frac{1}{2}(h(z) + \overline{h(z)})\right)
	\\ &= \Re \left(Mz\cdot \left(\overline{G^*G z} - h'(z)\right)\right).
	\end{aligned}
\end{equation}
Since $Mz\cdot h'(z)$ is quadratic and holomorphic in $z$, the Hermitian part of $\Theta$ is
\[
	\frac{1}{2}\left(\Theta(z) + \Theta(iz)\right) = \Re \left(Mz \cdot \overline{G^*G z}\right),
\]
which must be positive semidefinite since $\Theta$ is by Theorem \ref{thm.boundedness.Theta} The second claim follows from writing this quantity as an inner product, moving the adjoint $G^*$ to the other side, a change of variables $y = Gz$.
\end{proof}

We continue with an observation that, since $\Phi$ is strictly convex, the matrix norm $\|e^{\tau M}\|$ can play a deciding role in determining whether $\exp(\tau P)$, or even $\exp(\delta P_0)\exp(\tau P)$, is bounded as in Theorems \ref{thm.boundedness.0} and \ref{thm.boundedness.delta}.  To begin, it is useful to identify the maximum $\delta$ such that $\Phi^{(\delta)}(z)$ is convex.

\begin{lemma}\label{lem.Delta.0}
	Let $\Phi$ obey \eqref{eq.Phi.assumptions}. Using the decomposition \eqref{eq.decompose.Phi}, we define the matrix
	\begin{equation}\label{eq.def.H}
		H = (G^{-1})^{\top}h''G^{-1}.
	\end{equation}
	For $\Phi^{(\delta)}$ in \eqref{eq.def.Phi.delta}, let $\Delta_0$ be defined by
	\[
		\Delta_0 = \sup\{\delta \in \Bbb{R} \::\: \forall z \in \Bbb{C}^n,~~\Phi^{(\delta)}(z) \geq 0\}.
	\]

	Then
	\begin{equation}\label{eq.def.Delta.0}
		\Delta_0 = -\frac{1}{2}\log\|H\|.
	\end{equation}
\end{lemma}

\begin{proof}
The lemma follows from recalling that $\Phi^{(\delta)}(z)$ is decreasing in $\delta$ and noting that
\[
	\Phi^{(\Delta_0)}(G^{-1}z) = \frac{1}{2}(\|H\||z|^2 - \Re (Hz\cdot z))
\]
is convex by the Cauchy-Schwarz inequality but is not strictly convex by part (\ref{it.decomposition.Sigma}) in Lemma \ref{lem.decomposition} (which is essentially Takagi's factorization).
\end{proof}

\begin{proposition}\label{prop.norm.necc.suff}
Let $\Phi$ obey \eqref{eq.Phi.assumptions}, fix $\delta < \Delta_0$ with $\Delta_0$ defined in \eqref{eq.def.Delta.0}, and recall the definition \eqref{eq.def.Phi.delta} of the weight $\Phi^{(\delta)}$. Let
\[
	C_0(\delta) = \sqrt{\frac{\inf_{|z|=1} \Phi(z)}{\sup_{|z|=1}\Phi^{(\delta)}(z)}}
\]
and
\[
	C_1(\delta) = \sqrt{\frac{\sup_{|z|=1} \Phi(z)}{\inf_{|z|=1}\Phi^{(\delta)}(z)}}.
\]
Then, in order to have $\delta_0 \geq \delta$ for $\delta_0$ in \eqref{eq.def.delta.0}, it is necessary that
\begin{equation}\label{eq.norm.etM.necc}
	\|e^{\tau M}\| \leq \frac{1}{C_0(\delta)}
\end{equation}
and sufficient that
\begin{equation}\label{eq.norm.etM.suff}
	\|e^{\tau M}\| \leq \frac{1}{C_1(\delta)}.
\end{equation}
\end{proposition}

\begin{remark*}
Note that if $\delta = 0$ in the lemma above, then $C_1 = \frac{1}{C_0}$ and we obtain a necessary condition and a sufficient condition in order for $\exp(\tau P)$ to be bounded.  For general $\delta$, we obtain a necessary condition and a sufficient condition for the operator \eqref{eq.compose.bounded.l} to be bounded.
\end{remark*}

\begin{proof}
By Lemma \ref{lem.Delta.0}, and that $\Phi^{(\delta)}$ is decreasing in $\delta$, we have that $\Phi^{(\delta)}$ is strictly convex whenever $\delta < \Delta_0$, so the definitions of $C_0(\delta)$ and $C_1(\delta)$ give positive real numbers.

We note that the inequality \eqref{eq.Phi.increases.delta} from the definition \eqref{eq.def.delta.0} of $\delta_0$ is equivalent to the statement
\begin{equation}\label{eq.Phi.decreases.ratio}
	\frac{\Phi(e^{\tau M}z)}{\Phi^{(\delta)}(z)} \leq 1, \quad \forall z \in \Bbb{C}^n\backslash \{0\}.
\end{equation}
We reduce to a comparison on the unit sphere by writing
\[
	\frac{\Phi(e^{\tau M}z)}{\Phi^{(\delta)}(z)} = \frac{|e^{\tau M}z|^2}{|z|^2}\Phi\left(\frac{e^{\tau M}z}{|e^{\tau M}z|}\right)\left(\Phi^{(\delta)}\left(\frac{z}{|z|}\right)\right)^{-1}.
\]
If there exists some $z_0 \in \Bbb{C}^n\backslash \{0\}$ for which $|e^{\tau M}z_0| > C_0(\delta) |z_0|$, then
\[
	\frac{\Phi(e^{\tau M}z_0)}{\Phi^{(\delta)}(z_0)} > \frac{1}{C_0(\delta)^2} \frac{\inf_{|z|=1}\Phi(z)}{\sup_{|z|=1} \Phi^{(\delta)}(z)} = 1,
\]
violating \eqref{eq.Phi.decreases.ratio}.  This proves that \eqref{eq.norm.etM.necc} is necessary to have $\delta_0 \geq \delta$.  On the other hand, if \eqref{eq.norm.etM.suff} holds, then for all $w \in \Bbb{C}^n$ we see that
\[
	\frac{\Phi(e^{\tau M}w)}{\Phi^{(\delta)}(w)} \leq \frac{1}{C_1(\delta)^2}\frac{\sup_{|z|=1} \Phi(z)}{\inf_{|z|=1}\Phi^{(\delta)}(z)} = 1.
\]
This proves sufficiency and completes the proof of the proposition.
\end{proof}

We now show that $\Delta_0$ from \eqref{eq.def.Delta.0} gives the maximal possible decay, in terms of $\exp(\delta P_0)$, for functions in the range of $\exp(\tau P)$. We also show that this maximal decay is attained in the limit whenever $\|e^{\tau M}\| \to 0$.

\begin{proposition}\label{prop.maximal.smoothing}
	Let $\Phi$ obey \eqref{eq.Phi.assumptions}, let $\delta_0 = \delta_0(\tau)$ be as defined in \eqref{eq.def.delta.0}, and recall the definition \eqref{eq.def.Delta.0} of $\Delta_0$. Then, for any $\tau \in \Bbb{C}$,
	\begin{equation}\label{eq.max.delta.0}
		\delta_0(\tau) < \Delta_0
	\end{equation}
	and if $\{\tau_k\}_{k\in \Bbb{N}}$ is a sequence of complex numbers for which $\|e^{\tau_k M}\| \to 0$, then 
	\[
		\lim_{k \to \infty} \delta_0(\tau_k) = \Delta_0.
	\]
\end{proposition}

\begin{proof}
Since $\Phi(z)$ is strictly convex, $\Phi^{(\Delta_0)}(z)$ is not convex by Lemma \ref{lem.Delta.0}, and $e^{\tau M}$ is a linear bijection on $\Bbb{C}^n$, it is impossible to have $\Phi^{(\Delta_0)}(e^{\tau M}z) \geq \Phi(z)$ for all $z \in \Bbb{C}^n$ as in \eqref{eq.def.delta.0}. Therefore $\delta_0 < \Delta_0$. 

To prove the second claim, fix any $\delta < \Delta_0$. Since $\Phi^{(\delta)}(z)$ is strictly decreasing as a function of $\delta$ for $z \neq 0$, we see that $\Phi^{(\delta)}$ is strictly convex. Therefore, by Proposition \ref{prop.norm.necc.suff}, $\delta_0 \leq \delta$ for $\|e^{\tau_k M}\|$ sufficiently small, so the final claim of the proposition follows.
\end{proof}

These results motivate our interest in the set of $\tau$ for which $\|e^{\tau M}\|$ becomes small. Because $e^{\tau M}$ is always invertible, we can only have $\|e^{\tau M}\| \to 0$ as $|\tau|\to \infty$.

It is useful at this point to compute explicitly the matrix exponential of $M$ applied to a generalized eigenvector.  We refer to the definitions preceding Theorem \ref{thm.eigen.core}, including the definition of the order $r$ of a generalized eigenvector.

\begin{lemma}\label{lem.exptM.ej}
Let $M \in \Bbb{M}_{n\times n}$ and let $v\in\Bbb{C}^n$ be a generalized eigenvector of order $r$ with eigenvalue $\lambda$.  Then, as $|\tau| \to \infty$,
\[
	e^{\tau M}v = \frac{1}{(r-1)!}e^{\tau\lambda}\tau^{r-1}((M-\lambda)^{r-1}v + \BigO(|\tau|^{-1})).
\]
\end{lemma}

\begin{proof}
We write
\[
	\begin{aligned}
	e^{\tau M}v &= e^{\tau \lambda}e^{\tau(M-\lambda)}v
	\\ &= e^{\tau \lambda}\sum_{j=0}^\infty \frac{\tau^j}{j!} (M-\lambda)^j v.
	\end{aligned}
\]
By definition of the order $r$, the term $\frac{\tau^{r-1}}{(r-1)!}(M-\lambda)^{r-1}v$ in the sum is the nonvanishing term with the largest power of $\tau$, and the lemma follows.
\end{proof}

In particular, if $M$ is in Jordan normal form for which each standard basis vector $e_j$ is a generalized eigenvector of order $r_j$ with eigenvalue $\lambda_j$, then for all $j=1,\dots,n$
\[
	e^{\tau M}e_j = \frac{1}{(r_j-1)!}e^{\tau \lambda_j}\tau^{r_j-1}\left(e_{j-r_j+1} + \BigO(|\tau|^{-1})\right)
\]
as $|\tau| \to \infty$.

Since we now have a simple expansion for $e^{\tau M}e_j$ as $|\tau| \to \infty$, we can obtain a rather precise description of those $\tau$ for which $\exp(\tau P)$ is a bounded operator on $H_\Phi$, as $|\tau| \to \infty$, using the elementary inequality
\begin{equation}\label{eq.exptM.triangle}
	\max_{j=1,\dots,n} |e^{\tau M}e_j| \leq \|e^{\tau M}\| \leq \sqrt{n}\max_{j=1,\dots, n} |e^{\tau M}e_j|.
\end{equation}
Since
\begin{equation}\label{eq.exptM.ej.expansion}
	|e^{\tau M}e_j| = \exp\left(-\log((r_j-1)!) + (r_j-1)\log|\tau| + \Re(\tau \lambda_j)\right)(1+\BigO(|\tau|^{-1})),
\end{equation}
we see that if, for some $j$, we have $\Re(\tau \lambda_j) \gg \log |\tau|$, then $\|e^{\tau M}\| \to \infty$, so $\delta_0$ of \eqref{eq.def.delta.0} tends to $-\infty$ thanks to Proposition \ref{prop.norm.necc.suff}. Similarly, if $\Re(\tau \lambda_j) \ll -\log|\tau|$ for all $j = 1,\dots, n$, then $\|e^{\tau M}\| \to 0$, so $\delta_0 \to \Delta_0 = -\frac{1}{2}\log\|H\|$ as in Lemma \ref{lem.Delta.0}.

Therefore, if $\opnm{Spec} M$ is not contained in a half-plane, then $\|e^{\tau M}\| \to \infty$ as $|\tau| \to \infty$ regardless. The case where $\opnm{Spec} M$ is contained in a half-plane but no smaller sector is considered in Theorem \ref{thm.imaginary.eigenvalues}. By shifting the argument of $\tau$ if necessary, we assume for what follows that $\opnm{Spec} M \subset \{\Re \lambda > 0\}$. Writing
\[
	\lambda_j = \rho_je^{i\theta_j}
\]
for $\theta_j\in (-\pi/2, \pi/2)$, we may then define
\begin{equation}\label{eq.def.theta.pm}
	\begin{aligned}
	\theta_+ &= \max_{j=1,\dots,n} \theta_j,
	\\ \theta_- &= \min_{j=1,\dots,n} \theta_j.
	\end{aligned}
\end{equation}
If we also write
\[
	\tau = |\tau|e^{i\varphi},
\]
we have 
\begin{equation}\label{eq.Re.tau.lambda.cos}
	\Re(\tau \lambda_j) = |\tau|\rho_j\cos(\varphi + \theta_j).
\end{equation}
In supposing that $\cos(\varphi + \theta_j)$ is negative or small for each $j$, we assume that $\varphi+\theta_j \in [\pi/2-\delta, 3\pi/2+\delta]$ for all $j$ and for $\delta > 0$ small.  As a result,
\begin{equation}\label{eq.cos.maximized}
	\max_{j=1,\dots,n} \cos(\varphi+\theta_j) = \max\{\cos(\varphi + \theta_+), \cos(\varphi+\theta_-)\}.
\end{equation}

Of those eigenvalues $\lambda_j$ for which $\theta_j = \theta_+$ or $\theta_j = \theta_-$, we can identify the largest coefficient of the logarithmic correction coming from \eqref{eq.exptM.ej.expansion}:
\begin{equation}\label{eq.def.b.pm}
	b_\pm = \max_{\{j\::\:\theta_j = \theta_{\pm}\}} \frac{r_j - 1}{\rho_j}.
\end{equation}
In the regime $|\tau|\to \infty$, we record how the leading term of this expansion can determine whether $\|e^{\tau M}\| \to 0$ or $\|e^{\tau M}\|\to \infty$, depending principally on the argument of $\tau$.

\begin{proposition}\label{prop.exptM.asymptotic}
Suppose that $M \in \Bbb{M}_{n\times n}(\Bbb{C})$ is an invertible matrix in Jordan normal form for which $\opnm{Spec} M \subset \{\Re \lambda > 0\}$. Therefore write 
\[
	\opnm{Spec} M = \{\lambda_j = \rho_j e^{i\theta_j}\::\: j=1,\dots,n\},
\]
repeated for algebraic multiplicity, with $\theta_j \in (-\pi/2, \pi/2)$ and with orders of generalized eigenvectors $\{r_j\}_{j=1}^n$ as in \eqref{eq.def.rj}.  Let $\theta_\pm$ be as in \eqref{eq.def.theta.pm} and $b_\pm$ be as in \eqref{eq.def.b.pm}.

Then, for every $C_0 > 0$, there exists some $R_0, R_1 > 0$ such that
\[
	\|e^{\tau M}\| \leq \frac{1}{C_0}
\]
whenever $|\tau| \geq R_0$ and, for both signs,
\[
	\cos(\varphi + \theta_\pm) \leq \frac{1}{|\tau|}(-b_\pm \log|\tau| - R_1).
\]

Similarly, for every $C_0 > 0$, there exists some $R_0, R_1 > 0$ such that
\[
	\|e^{\tau M}\| \geq C_0
\]
whenever $|\tau| \geq R_0$ and, for at least one sign,
\[
	\cos(\varphi + \theta_\pm) \geq \frac{1}{|\tau|}(-b_\pm \log|\tau| + R_1).
\]
\end{proposition}

\begin{remark*}
We may dispense with the hypothesis that $M$ is in Jordan normal form by taking into account the condition number of a matrix $\tilde{G}$ such that $\tilde{G}M\tilde{G}^{-1}$ is in Jordan normal form.  So long as the spectrum of $M$ is in a proper half-plane, we may obtain similar asymptotics by applying the proposition to $e^{i\theta_0} M$ for some $\theta_0 \in [0, 2\pi)$.  If the spectrum of $M$ is not contained in a half-plane, then $e^{\tau M} \to \infty$ exponentially rapidly as $|\tau| \to \infty$ since then there exists some $C > 0$ where every $\tau$ admits a $j$ with $\Re \tau \lambda_j \geq |\tau|/C$.  Some discussion of the situation when $\opnm{Spec}M$ is contained in a half-plane but no smaller sector appears in Theorem \ref{thm.imaginary.eigenvalues}. If $0 \in \opnm{Spec}M$ then $\|e^{\tau M}\| \geq 1$ always, and if a Jordan block corresponds to the zero eigenvalue, then $\|e^{\tau M}\| \to \infty$ at least polynomially rapidly as $|\tau| \to \infty$.
\end{remark*}

\begin{proof}
Since otherwise $\|e^{\tau M}\|\to \infty$, we may certainly assume that $\cos(\varphi+\theta_j) \in [-1, 1/2]$, in which case \eqref{eq.cos.maximized} holds.  By the expansion \eqref{eq.exptM.ej.expansion} and \eqref{eq.Re.tau.lambda.cos}, 
\[
	|e^{\tau M}e_j| = \exp\left(\frac{1}{\rho_j|\tau|}\left(\frac{r_j - 1}{\rho_j}\frac{\log|\tau|}{|\tau|} + \cos(\varphi+\theta_j)+\BigO(|\tau|^{-1})\right)\right)(1+\BigO(|\tau|^{-1})).
\]
As $|\tau| \to \infty$, the maximum of this quantity, ignoring the $\BigO(|\tau|^{-1})$ terms, for $j=1,\dots, n$ is attained for some $j$ where $\theta_j \in \{\theta_+, \theta_-\}$ and where $\frac{r_j-1}{\rho_j} = b_\pm$ accordingly.  The result then follows from \eqref{eq.exptM.triangle}.
\end{proof}

Up to shifting by constants, this allows us to describe the set of $\tau$ with $|\tau|$ large for which $\exp(\tau P)$ is bounded as in Theorem \ref{thm.boundedness.0} or even bounded after composing with $\exp(\delta P_0)$ as in Theorem \ref{thm.boundedness.delta}.

\begin{theorem}\label{thm.bounded.tau} Let the matrix $M$, the weight $\Phi$, and the operators $P$ and $\exp(\tau P)$ be as in Proposition \ref{prop.solve.exptP}. Recall the definitions \eqref{eq.def.delta.0} of $\delta_0$ and \eqref{eq.def.Delta.0} of $\Delta_0$.  Suppose in addition that $\opnm{Spec} M \subset \{\Re \lambda > 0\}$. For every $\delta \in (-\infty, \Delta_0)$ there exists $C_1, C_2 \in \Bbb{R}$ and $C_0 > 0$ such that $\delta_0 \geq \delta$ whenever $|\tau| \geq C_0$ and, for both signs,
\[
	\cos(\varphi + \theta_\pm) \leq \frac{1}{|\tau|}(-b_\pm \log|\tau| - C_1)
\]
and $\delta_0 \leq \delta$ whenever $|\tau| \geq C_0$ and, for at least one sign,
\[
	\cos(\varphi + \theta_\pm) \geq \frac{1}{|\tau|}(-b_\pm \log|\tau| + C_2).
\]
\end{theorem}

\begin{remark*}
We again compare with the case of a normal operator $A$ on a Hilbert space $\mathcal{H}$ for which $\opnm{Spec}A$ is equal to the set of eigenvalues of $P$ given by \eqref{eq.def.lambda.alpha}. In this case,
\[
	\{\tau \::\: e^{\tau A} \in \mathcal{L}(\mathcal{H})\} = \{\tau = |\tau|e^{i\varphi} \::\: \cos(\varphi + \theta_+) \leq 0 \textnormal{  and  } \cos(\varphi + \theta_-) \leq 0\}.
\]
We therefore see that the set of $\tau$ for which $\exp(\tau P)$ is bounded and $\tau$ is large is substantially similar to the same set where $P$ is replaced by a normal operator sharing the eigenvalues of $P$.

In Figure \ref{fig.tau.Example} we have an diagram of a typical region in the complex plane indicated by the theorem.  We have set 
\[
	\opnm{Spec} M = \left\{\frac{5}{3}e^{i \pi/4}, 2e^{-i\pi/6}, \frac{5}{2}\right\},
\]
and the eigenvalues of $P$ are indicated by dots (with circles indicating the eigenvalues of $M$). We suppose that the eigenvalue $\frac{5}{3}e^{i \pi/4}$ is associated with a Jordan block of size 3 while the eigenvalue $2e^{-i\pi/6}$ is not associated with any nontrivial Jordan block.  Then the light grey area indicates the set of $\bar{\tau} \in \Bbb{C}$ where we know that $\exp(\tau P)$ is unbounded, and the dark grey area is the set of $\bar{\tau} \in \Bbb{C}$ where we know that $\exp(\tau P)$ is bounded, with constants $C_0, C_1,$ and $C_2$ chosen by hand.

In order to clarify that the boundary of the sets indicated are effectively the graphs of a logarithm for $|\tau|$ large, we consider
\begin{equation}\label{eq.large.tau.example}
	\left\{\tau = |\tau|e^{i\varphi} \::\: \cos(\varphi + \theta_+) = \frac{1}{|\tau|}(-b_+ \log|\tau| - C_1)\right\}
\end{equation}
for $\theta_+ = 0$ and $\Im \tau > 0$ as $|\tau|\to \infty$.  We therefore have $\cos(\varphi+\theta_+) = \cos(\varphi) = \Re \tau/|\tau|$, and we can write
\[
	\frac{\Im \tau}{|\tau|} = \sqrt{1-\frac{(\Re \tau)^2}{|\tau|^2}} = 1 + \BigO\left(\frac{(\log|\tau|)^2}{|\tau|}\right).
\]
Seeing that, in this case, $\Im \tau \approx |\tau|$ as $|\tau|\to \infty$, we get that the boundary \eqref{eq.large.tau.example} is contained, for $|\tau|$ sufficiently large, in the set
\[
	\left\{\tau\::\: \Re \tau = (-b_+ \log(\Im \tau) - C_1)\left(1+\BigO(\frac{(\Im \tau)^2}{|\tau|^2})\right)\right\}.
\]
\end{remark*}

\begin{proof}
The claim is immediate from Proposition \ref{prop.norm.necc.suff}, Proposition \ref{prop.exptM.asymptotic}, and the identification of $\Delta_0$ in Lemma \ref{lem.Delta.0}.
\end{proof}

\begin{figure}
\centering
	\includegraphics[width=0.70\textwidth]{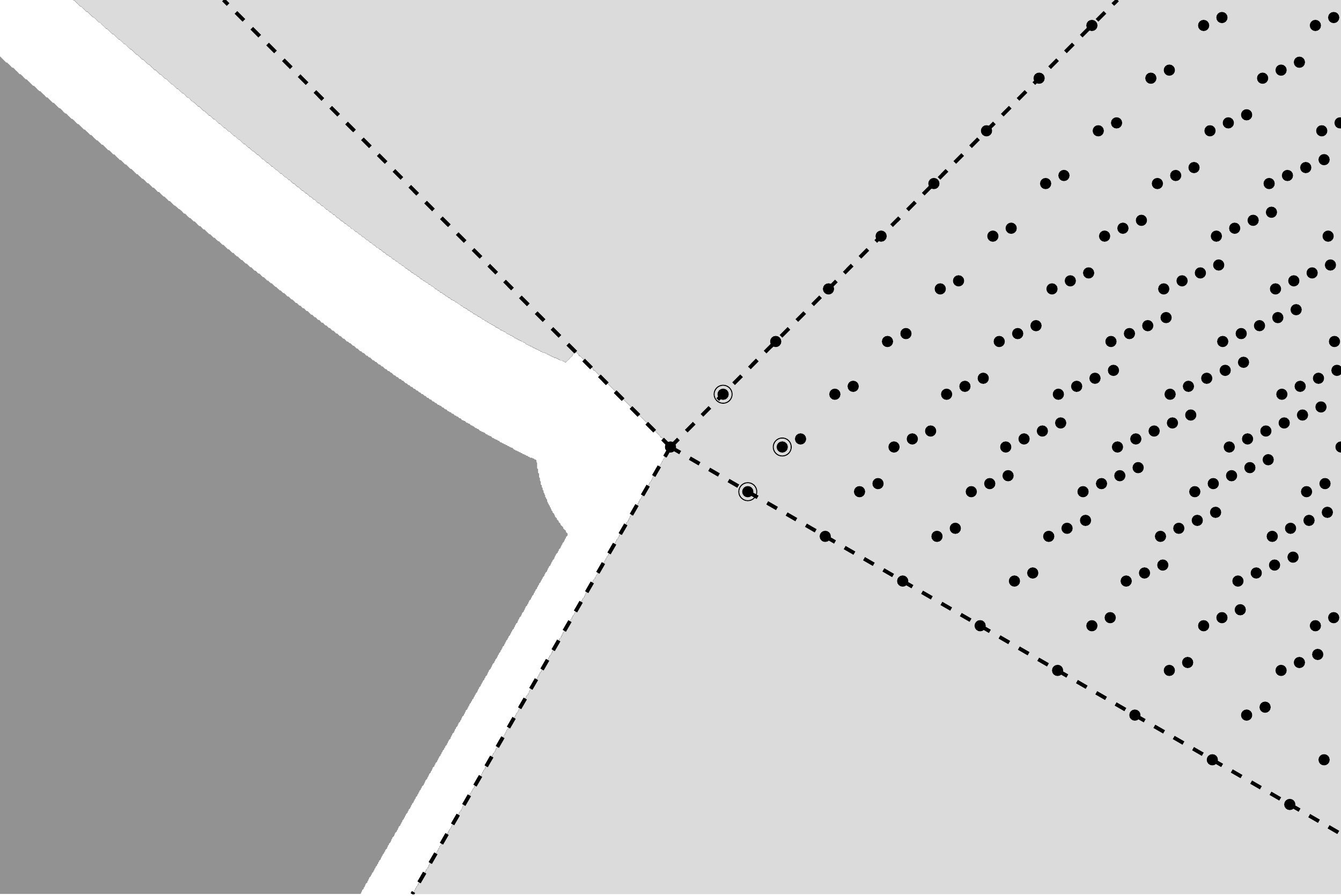}
	\caption{Diagram of a typical set of $\bar{\tau} \in \Bbb{C}$ for boundedness of $\exp(\tau P)$ as indicated in Theorem \ref{thm.bounded.tau}; see the remark following the theorem}
	\label{fig.tau.Example}
\end{figure}

We turn to the question of how imaginary eigenvalues of $M$ affect boundedness of $\exp(-tP)$, particularly for short times $t \to 0^+$ as in Theorem \ref{thm.boundedness.Theta}. We show here that this only occurs when $P$ is skew-adjoint in the variables in $\Bbb{C}^n$ corresponding to imaginary eigenvalues of $M$; see also \cite[Prop.~2.0.1,~(iii)]{HiPS2009} for a similar result in terms of quadratic operators on $L^2(\Bbb{R}^n)$.

We recall from Theorem \ref{thm.boundedness.Theta} that $\exp(-tP)$ is bounded for all $t > 0$ if and only if $\Theta \geq 0$ as in \eqref{eq.ellipticity}, and then from Proposition \ref{prop.expansive} we can conclude that $\opnm{Spec}M \subset \{\Re \lambda \geq 0\}$.  We therefore decompose $\Bbb{C}^n$ into the subspaces of generalized eigenvectors of $M$ corresponding to eigenvalues which are purely imaginary and those which have positive real parts:
\begin{equation}\label{eq.def.imag.V}
	V := \bigoplus_{\lambda \in (\opnm{Spec}M) \cap i\Bbb{R}} \ker(M-\lambda)^n
\end{equation}
and
\begin{equation}\label{eq.def.imag.W}
	W := \bigoplus_{\lambda \in (\opnm{Spec}M) \cap \{\Re \lambda > 0\}} \ker(M-\lambda)^n
\end{equation}

\begin{theorem}\label{thm.imaginary.eigenvalues}
Let $P$ be as in \eqref{eq.def.P} acting on $H_\Phi$ with $\Phi$ verifying \eqref{eq.Phi.assumptions}. Suppose that $\Theta$ from \eqref{eq.def.Theta} obeys \eqref{eq.ellipticity} and therefore define $V$ and $W$ as in \eqref{eq.def.imag.V} and \eqref{eq.def.imag.W} as well as the projection $\pi_W$ such that $\pi_W z \in W$ and $(1-\pi_W)z \in V$.  Let the matrix $G$ and the function $h$ be as in the decomposition \eqref{eq.decompose.Phi}.

Then $GMG^{-1}|_{GV}$ is skew-adjoint, $GV \perp GW$, and 
\begin{equation}\label{eq.M.h.pi.W}
	Mz\cdot \partial_z h(z) = Mz\cdot \partial_z h(\pi_W z)
\end{equation}
with $\pi_W$ the projection onto $W$ defined by $\Bbb{C}^n = V \oplus W$. Furthermore, $\Theta|_V = 0$ and $\Theta|_W \geq 0$.
\end{theorem}

\begin{remark*}
The proof implies a reduction to a normal form in which the action of $P$ on the $V$ variables becomes very simple. Specifically, letting
\[
	g(z) = h(z) - h(\pi_W z),
\]
we have that conjugation with $\mathcal{W}_g$ as in \eqref{eq.def.W} eliminates dependence of $h$ on the $V$ variables, and then conjugation with $\mathcal{V}_G^*$ as in \eqref{eq.def.V} reduces $G$ to the identity matrix and replaces $M$ with $GMG^{-1}$.  A final change of variables $\mathcal{V}_U$ for a unitary matrix $U$ then reduces $GV$ to $\{(z',0)\}$ and $GW$ to $\{(0, z'')\}$ while diagonalizing the skew-adjoint matrix $GMG^{-1}|_{GV}$.

As a result we have a unitary equivalence between $P = Mz\cdot \partial_z$ acting on $H_\Phi$ and $M'z\cdot\partial_z$ acting on $H_{\tilde{\Phi}}$. Writing $\opnm{Spec}M \cap i\Bbb{R} = \{i\rho_1, \dots, i\rho_J\}$, counted for multiplicity and with $J = \dim V$, we have
\[
M' = \left(
\begin{array}{ccc|c}
i\rho_1 & & 0 & 0
\\ & \ddots & & \vdots
\\ 0 & & i\rho_J & 0
\\ \hline
0 & \cdots & 0 & M''
\end{array}\right)
\]
for some matrix $M'' \in \Bbb{M}_{(n-J)\times (n-J)}(\Bbb{C})$ and, writing $z = (z', z'') \in \Bbb{C}^J \times \Bbb{C}^{n-J}$,
\[
	\tilde{\Phi}(z) = \frac{1}{2}|z'|^2 + \tilde{\Phi}_2(z'').
\]

After the proof, we illustrate the situation discussed and complications which may arise with two examples.
\end{remark*}

\begin{proof}
To simplify the exposition, let $\tilde{M} = GMG^{-1}$; then $\tilde{V} = GV$ and $\tilde{W} = GW$ form the sums of generalized eigenspaces of $\tilde{M}$ corresponding to purely imaginary eigenvalues and eigenvalues with positive real parts. By Proposition \ref{prop.expansive},
\[
	\Re \langle \tilde{M}z, z\rangle \geq 0, \quad \forall z \in \Bbb{C}^n.
\]

Let $v, x \in \Bbb{C}^n$ and suppose that $\tilde{M}v = i\rho v$ for $\rho \in \Bbb{R}$.  Then for $\alpha, \beta \in \Bbb{C}$,
\[
	\begin{aligned}
	\Re \langle \tilde{M}(\alpha v + \beta x), \alpha v + \beta x\rangle &= \Re\left(i\rho |\alpha v + \beta x|^2 + \beta \langle (\tilde{M} - i\rho)x, \alpha v + \beta x\rangle\right)
	\\ & = \Re \left(\bar{\alpha}\beta\langle (\tilde{M}-i\rho)x, v\rangle + \BigO(|\beta|^2)\right).
	\end{aligned}
\]
This quantity must be non-negative for all $\alpha, \beta \in \Bbb{C}$, so it is clear from allowing $\alpha$ to vary that
\[
	\langle (\tilde{M}-i\rho)x, v\rangle = 0.
\]

This gives the following immediate consequences.  If $\tilde{M}v = i\rho v$ with $\rho \in \Bbb{R}$ and $\tilde{M}\tilde{v} = i\rho \tilde{v} + v$, then 
\[
	|v|^2 = \langle (\tilde{M}-i\rho)\tilde{v}, v\rangle = 0.
\]
Therefore every generalized eigenvector of $\tilde{M}|_{\tilde{V}}$ is an eigenvector, which is to say that $\tilde{M}_{\tilde{V}}$ is diagonalizable.  Similarly, if $\tilde{M}v = i\rho v$ with $\rho \in \Bbb{R}$ and $\tilde{M}w = \mu w$ for $\mu \neq v$, then $v \perp w$.  Therefore $\tilde{M}|_{\tilde{V}}$ has an orthonormal basis of eigenvectors and $\tilde{V}$ is orthogonal to any eigenvector of $M$ lying in $\tilde{W}$.  If we assume that $\tilde{M}v = i\rho v$ with $\rho \in \Bbb{R}$, that $\tilde{M}\tilde{w} = \mu\tilde{w} + w$ for $\mu \neq i\rho$, and that $w \perp v$, then $v \perp \tilde{w}$.  In this way, we see that every such $v$ is orthogonal to every generalized eigenvector of $\tilde{M}$ with a different eigenvalue, and therefore $\tilde{V} \perp \tilde{W}$.

Since $\tilde{M}|_{\tilde{V}}$ has an orthonormal basis of eigenvectors and 
\[
	\opnm{Spec}\tilde{M}|_{\tilde{V}} = \opnm{Spec} M|_V \subset i\Bbb{R},
\]
we see that $\tilde{M}|_{\tilde{V}}$ is skew-adjoint. From the definitions of the matrix $\tilde{M} = GMG^{-1}$ and the subspaces $\tilde{V} = GV$ and $\tilde{W} = GW$, all that remains is to prove that $\Theta|_V = 0$ and \eqref{eq.M.h.pi.W}.

For any $v \in V$ and $w \in W$, we may write $\Theta(v+w)$ as
\[
	\Theta(v+w) = \Re \left(\langle GM(v+w), G(v+w)\rangle - M(v+w)\cdot h'(v+w)\right).
\]
We have that $\langle GMG^{-1}Gv, Gv\rangle$ is purely imaginary since $\tilde{M}|_{\tilde{V}}$ is skew-adjoint.  Since $\tilde{V}$ and $\tilde{W}$ are $\tilde{M}$-invariant and orthogonal,
\[
	\langle GMG^{-1}Gw, Gv\rangle = \langle GMG^{-1}Gv, Gw\rangle = 0.
\]
Therefore, for fixed vectors $v \in V$ and $w \in W$ and for $\alpha, \beta \in \Bbb{C}$,
\[
	\Theta(\alpha v + \beta w) = \Re\left(-\alpha^2 Mv\cdot h'(v) - \alpha\beta (Mv\cdot h'(w) + Mw\cdot h'(v)) + \BigO(|\beta|^2)\right).
\]
Letting the argument of $\alpha$ vary and letting $\beta \to 0$, we discover first that $Mv\cdot h'(v) = 0$, implying that $\Theta|_V = 0$, and then that $Mv\cdot h'(w) + Mw\cdot h'(v) = 0$.  Expanding out $Mz\cdot h'(z)$ for $z = (1-\pi_W)z+\pi_W z \in V \oplus W$, we see that
\[
	Mz\cdot h'(z) = M\pi_W w\cdot h'(\pi_W w).
\]
Since $V$ and $W$ are $M$-invariant, $[M, \pi_W] = 0$. Furthermore, 
\[
	\pi_W^\top h'(\pi_W z) = \partial_z h(\pi_W z),
\]
and this suffices to prove \eqref{eq.M.h.pi.W}.
\end{proof}

\begin{example}
The conclusions of Theorem \ref{thm.imaginary.eigenvalues} do not necessarily hold if one assumes only that $\exp(-tP)$ is unbounded for some, or even infinitely many, positive times.  The natural example is
\[
	P = iz\cdot \partial_z
\]
acting on $H_\Phi$ with
\[
	\Phi(z) = \frac{1}{2}(|z|^2 - a\Re z^2)
\]
for some $a \in (0,1)$.  Following \eqref{eq.RHO.weight}, we see that $P$ is unitarily equivalent to $\frac{i}{2}(Q_\theta-e^{i\theta})$ with $Q_\theta$ from \eqref{eq.RHO} and $\theta = \arcsin a$.

It is then easy to check from Theorem \ref{thm.boundedness.0} that $\exp(-tP)$ is unbounded unless $t/\pi = j \in \Bbb{Z}$, and in this case $\exp(-\pi j P)u = (-1)^ju$.
\end{example}

\begin{example}\label{ex.nontrivial.trivial.h}
While the conclusion of Theorem \ref{thm.imaginary.eigenvalues} does not say that the function $h$ (representing the pluriharmonic part of the weight $\Phi$) does not depend on the $V$ variables, it does say that, due to cancellation from $Mz$, the role of these variables in $h$ does not affect $P$ and may be eliminated with a unitary transformation of type \eqref{eq.def.W}.

A natural, if somewhat degenerate, example, is given by
\[
	M = \left(\begin{array}{cc} i & 0 \\ 0 & -i \end{array} \right)
\]
and
\[
	\Phi(z) = \frac{1}{2}|z|^2 - a\frac{1}{2}\Re z_1z_2, \quad a \in (-1, 1).
\]
Since, in this case,
\[
	Mz\cdot h'(z) = (iz_1, -iz_2)\cdot \frac{1}{2}(z_2, z_1) = 0,
\]
the reduction of Theorem \ref{thm.imaginary.eigenvalues} gives that $P = Mz\cdot \partial_z$ acting on $H_\Phi$ is unitary and unitarily equivalent to $Mz\cdot \partial_z$ acting on $H_\Psi$ with $\Psi(z) = \frac{1}{2}|z|^2$.

We say this example is somewhat degenerate because $\opnm{Spec}P = i\Bbb{Z}$ and, for any $j, k \in \Bbb{Z}$ with $j \geq 0$ and $j \geq k$,
\[
	z_1^{k+j}z_2^j \in \ker(Mz\cdot \partial_z -ik),
\]
and so $\dim \ker (Mz \cdot \partial_z - ik) = \infty$. What is more, so long as $f$ is an entire function on $\Bbb{C}$ for which $z_1^k f(z_1 z_2) \in H_\Phi$, clearly $z_1^k f(z_1z_2) \in \ker (Mz\cdot \partial_z - ik)$.

Setting $g(z) = az_1z_2$ and writing $\mathcal{W}_g$ as in \eqref{eq.def.W} gives that $\mathcal{W}_g:H_\Phi \to H_{\Psi}$ is unitary and that
\[
	\mathcal{W}_g^* Mz\cdot \partial_z \mathcal{W}_g = e^{-az_1z_2}i(z_1\partial_{z_1} - z_2\partial_{z_2})e^{az_1z_2} = Mz\cdot \partial_z.
\]
Again, the fact that $Mz\cdot\partial_z$ is unchanged under conjugation by $\mathcal{W}_g$ is quite special and reflects that $Mz\cdot \partial_z (az_1z_2) = 0$, as in \eqref{eq.M.h.pi.W} with $\pi_W = 0$. After conjugation by $\mathcal{W}_g$, it is clear also that $P$ is unitarily equivalent to $i$ times a harmonic oscillator in the $x_1$ variable plus $-i$ times a harmonic oscillator in the $x_2$ variable, acting on $L^2(\Bbb{R}^2)$, since the classical Bargmann transform relates the harmonic oscillator $Q_0$ to $z\cdot\partial_z$ acting on $H_\Psi$.
\end{example}

We consider finally a more general class of operators where we allow the inclusion of terms which are first-order in $(z,\partial_z)$. It is clear that introducing a constant term would not affect whether the operator $\exp(\tau P)$ is bounded or not; apart for the boundary case where equality holds in \eqref{eq.Phi.increases}, terms which are first-order in $(z,\partial_z)$ do not either.  We do not attempt a particularly deep analysis, and instead content ourselves with a brief illustration that certain more general operators may be analyzed by the approach used in the present work.  We remark that this class of operators corresponds to the Weyl quantization acting on $L^2(\Bbb{R}^n)$ of any degree-2 polynomial in $(x,\xi)$ for which the quadratic part obeys the hypotheses of Proposition \ref{prop.which.q} and for which $0 \notin \opnm{Spec} F$.

\begin{proposition}\label{prop.linear}
Let $M \in GL_n(\Bbb{C})$ be an invertible matrix and let $a,b \in \Bbb{C}^n$. Define
\[
	L = Mz\cdot \partial_z + a\cdot z + b\cdot \partial_z.
\]
Then the evolution equation
\[
	\left\{\begin{array}{l} \partial_t u + Lu = 0,\\ u(0,z) = u_0 \in H_\Phi,\end{array}\right.
\]
for $\Phi$ obeying \eqref{eq.Phi.assumptions}, admits a unique holomorphic solution $u(t,z) = \exp(-tL)u_0$ where
\begin{equation}\label{eq.linear.terms.sol}
	\exp(\tau L)u(z) = e^{a \cdot (M^{-1}(e^{\tau M} -1)(z+ M^{-1}b) - \tau M^{-1} b)}u(e^{\tau M}z + (e^{\tau M}-1)M^{-1}b).
\end{equation}
This operator is bounded on $H_\Phi$ whenever
\begin{equation}\label{eq.linear.Phi.increases}
	\liminf_{|z|\to\infty}\tilde{\Phi}(e^{-\tau M}z) - \tilde{\Phi}(z) > - \infty,
\end{equation}
with
\begin{equation}\label{eq.def.linear.Phi}
	\tilde{\Phi}(z) = \Phi(z - M^{-1}b) + \Re a\cdot M^{-1}z.
\end{equation}
Furthermore, with $P = Mz\cdot \partial_z$ and $\exp(\tau P)$ defined as in Proposition \ref{prop.solve.exptP}, we have that if $\exp(\tau P)$ is unbounded on $H_\Phi$, then $\exp(\tau L)$ is also unbounded, and if $\exp(\tau P)$ is compact with exponentially decaying singular values as in \eqref{eq.exp.decaying}, then $\exp(\tau L)$ is also compact with exponentially decaying singular values.
\end{proposition}

\begin{proof}
We proceed by a unitary reduction to the case of $\exp(\tau P)$ already studied beginning with Proposition \eqref{prop.solve.exptP}.  For $v \in \Bbb{C}^n$ fixed, we introduce the unitary shift map
\begin{equation}\label{eq.def.Shift}
	\mathcal{S}_v : H_\Phi \ni u(z) \mapsto u(z+v) \in H_{\Phi(\cdot + v)}
\end{equation}
for which, with $P$ as in \eqref{eq.def.P},
\begin{equation}\label{eq.trans.Shift}
	\mathcal{S}_v P \mathcal{S}_v^* = (Mz + Mv)\cdot \partial_z.
\end{equation}
Let $v = -M^{-1}b$ and $g(z) = (M^{-1})^\top a \cdot z$,
\[
	\mathcal{W}_g \mathcal{S}_{v} L \mathcal{S}_{v}^* \mathcal{W}_g^* = P - a\cdot M^{-1}b.
\]

We then may define $\exp(\tau P)u(z) = u(e^{\tau M}z)$ as in Proposition \ref{prop.solve.exptP} and
\[
	\exp(\tau L) = e^{-\tau M^{-1}b \cdot a}\mathcal{S}_v^* \mathcal{W}_g^* \exp(\tau P)\mathcal{W}_g \mathcal{S}_v,
\]
which gives the formula \eqref{eq.linear.terms.sol}.  Therefore $\exp(\tau L) : H_\Phi \to H_\Phi$ may be analyzed as an operator via the relation
\[
	\mathcal{V}_{e^{-\tau M}} \mathcal{W}_g \mathcal{S}_v \exp(\tau L) \mathcal{S}_v^* \mathcal{W}_g^*u(z) = e^{-\tau M^{-1}b\cdot a - \Re \tau \opnm{Tr} M}u(z).
\]
In order to have $\mathcal{S}_v^* \mathcal{W}_g^* u \in H_\Phi$, we take $u \in \mathcal{W}_g \mathcal{S}_v H_\Phi$ which, following \eqref{eq.def.W} and \eqref{eq.def.Shift}, is $H_{\tilde{\Phi}}$ for $\tilde{\Phi}$ in \eqref{eq.def.linear.Phi}. Similarly, the norm of the image is in $\mathcal{V}_{e^{-\tau M}} \mathcal{W}_g \mathcal{S}_v H_\Phi$ which is $H_{\tilde{\Phi}(e^{-\tau M} \cdot)}$.

The same analysis of the reproducing kernel by following the unitary transformations shows that $\exp(\tau L)$ is bounded if and only if \eqref{eq.linear.Phi.increases} holds; and a similar operator $P_0$ shows that $\exp(\tau L)$ is compact with decreasing singular values whenever there exists $C > 0$ such that
\begin{equation}\label{eq.linear.increases.at.infinity}
	\tilde{\Phi}(e^{-\tau M}z) - \tilde{\Phi}(z) \geq \frac{1}{C}|z|^2 - C, \quad \forall z \in \Bbb{C}^n.
\end{equation}

To prove that $\exp(\tau L)$ is unbounded or compact whenever $\exp(\tau P)$ is unbounded or compact, we only need to use that 
\[
	\tilde{\Phi}(e^{-\tau M}z) - \tilde{\Phi}(z) = \Phi(e^{-\tau M}z) - \Phi(z) + \BigO(1+|z|).
\]
Therefore when $\Phi(e^{-\tau M}z_0) < \Phi(z_0)$ for some $z_0 \in \Bbb{C}^n$, then $\tilde{\Phi}(e^{-\tau M}rz_0) < \Phi(rz_0)$ for $r > 0$ sufficiently large.  Similarly, if $\Phi(e^{-\tau M}z) > \Phi(z)$ on the unit sphere $\{|z| = 1\}$, then a scaling argument shows that \eqref{eq.linear.increases.at.infinity} holds.
\end{proof}

\section{Real-side equivalence}\label{sec.real}

The operators given by \eqref{eq.def.P} are unitarily equivalent (up to the addition of a constant) to certain operators on $L^2(\Bbb{R}^n)$ given by the Weyl quantization of quadratic forms.  In this section, we begin by recalling basic definitions and facts about these Weyl quantizations.  We then discuss the aforementioned unitary equivalence with the operators on Fock spaces considered in the previous section. Afterwards, we consider the purely self-adjoint question of comparing the semigroups of two operators of harmonic oscillator type.  Then, for reference, we present a corollary collecting many results from Section \ref{sec.Fock} applied to real-side operators.  Finally, we perform explicit computations and discuss illustrations related to the examples in Section \ref{subsec.examples}.

\subsection{Real-side quadratic operators}\label{subsec.real-side}

 Much of the following discussion can be found in previous works including \cite{Sj1974}, \cite{HiPS2009}, \cite{HiSjVi2013}, and \cite{Vi2013}. Let $q(x,\xi):\Bbb{R}^{2n}\to\Bbb{C}$ be a quadratic form. We define the Weyl quantization by replacing the $\xi$ variables with the self-adjoint derivatives $D_x = -i\partial_x$ as follows:
\begin{equation}\label{eq.Weyl}
	q^w(x,D_x) = \sum_{|\alpha + \beta| = 2} \frac{q''_{\alpha\beta}}{2}(x^\alpha D_x^\beta + D_x^\beta x^\alpha).
\end{equation}

For comparison, our operator $P$ in \eqref{eq.def.P} may also be realized as a Weyl quantization:
\begin{equation}\label{eq.P.symbol}
	\begin{aligned}
	P &= p^w(z,D_z) - \frac{1}{2}\opnm{Tr} M,
	\\ p(z,\zeta) &= (Mz)\cdot(i\zeta).	
	\end{aligned}
\end{equation}

The Weyl quantization of quadratic forms are often studied under an ellipticity hypothesis
\begin{equation}\label{eq.real.ell.hypothesis}
	q(x,\xi) = 0 \implies (x,\xi) = 0
\end{equation}
and the additional assumption in dimension $n = 1$
\begin{equation}\label{eq.real.ell.hypothesis.dim1}
	q(\Bbb{R}^2) \neq \Bbb{C}.
\end{equation}
Following \cite[Lem.~2.1]{PS2007}, we have that multiples of rotated harmonic oscillators $-(d/dx)^2+ e^{2i\theta}x^2$ are the only possible dimension-one operators satisfying the ellipticity assumption; this continues to be true for the operators considered here, since any weight in dimension one can be reduced to a weight of the form $\eqref{eq.RHO.weight}$ after a change of variables.

We turn to the spectral theory for quadratic operators obeying either \eqref{eq.real.ell.hypothesis} and, in dimension one, \eqref{eq.real.ell.hypothesis.dim1} or obeying \eqref{eq.real.semidef} and \eqref{eq.real.subell.hypothesis} introduced below. Under these assumptions, the spectral decomposition of the operator is determined by the spectral decomposition of the fundamental matrix
\begin{equation}\label{eq.def.F}
	F = F(q) = \frac{1}{2}\left(\begin{array}{cc} q''_{\xi x} & q''_{\xi\xi} \\ -q''_{xx} & -q''_{x \xi}\end{array}\right),
\end{equation}
described in for instance \cite[Sec.\ 21.5]{HoALPDO3}.  The role of the fundamental matrix is analogous to that of the Hessian matrix of second derivatives of $q$, except the usual inner product is replaced by the symplectic inner product
\begin{equation}\label{eq.def.sigma}
	\sigma((x,\xi),(y,\eta)) = \xi\cdot y - \eta \cdot x.
\end{equation}
The matrix $F$ is then determined uniquely by the conditions that
\begin{equation}\label{eq.F.sigma1}
	\sigma((x,\xi), F(x,\xi)) = q(x,\xi), \quad \forall (x,\xi) \in \Bbb{R}^{2n}
\end{equation}
and
\begin{equation}\label{eq.F.sigma2}
	\sigma((x,\xi), F(y,\eta)) = -\sigma(F(x,\xi), (y,\eta)), \quad \forall (x,\xi), (y,\eta) \in \Bbb{R}^{2n}.
\end{equation}

For our analysis of the eigenspaces of $F$, it is essential to introduce the concept of a positive or negative definite Lagrangian plane.  A Lagrangian plane $\Lambda$ is an $n$-dimensional subspace of $\Bbb{C}^{2n}$ for which $\sigma|_{\Lambda \times \Lambda} \equiv 0$; nondegeneracy of $\sigma$ implies that $\Lambda$ is maximal with respect to the vanishing of $\sigma$.  We say that a Lagrangian plane $\Lambda$ is positive if
\[
	-i\sigma((x,\xi), \overline{(x,\xi)}) > 0, \quad \forall (x,\xi) \in \Lambda \backslash \{0\}.
\]
This is equivalent to requiring that 
\begin{equation}\label{eq.Lagrangian.graph}
	\Lambda = \{(x,Ax) \::\: x \in \Bbb{C}^n\}
\end{equation}
for some $A \in \Bbb{M}_{n\times n}(\Bbb{C})$ which is symmetric, $A^\top = A$, and has positive definite imaginary part, $\Im A > 0$.  Negative Lagrangian planes are defined analogously with inequalities reversed.

It is a deep fact proven in \cite[Prop.~3.3]{Sj1974} that for $q(x,\xi):\Bbb{R}^{2n}\to \Bbb{C}$ quadratic obeying \eqref{eq.real.ell.hypothesis} and, in dimension one, \eqref{eq.real.ell.hypothesis.dim1} there exist Lagrangian planes $\Lambda^\pm$ which are $F$-invariant and where $\Lambda^+$ is positive and $\Lambda^-$ is negative.  Specifically, $\Lambda^+$ may be realized as the span of the generalized eigenspaces of $F$ corresponding to eigenvalues with $\lambda/i$ in $q(\Bbb{R}^{2n})$, and $\Lambda^-$ is similarly the span of the generalized eigenspaces of $F$ corresponding to eigenvalues which obey $-\lambda/i$ in $q(\Bbb{R}^{2n})$.  The proof can be adapted to cover the case of weakly elliptic operators obeying \eqref{eq.real.semidef} and \eqref{eq.real.subell.hypothesis} introduced below; details may be found in \cite[Prop.~2.1]{Vi2013}.  In Proposition \ref{prop.which.q} below, we prove that it is precisely the presence of these subspaces $\Lambda^\pm$ which determines whether we can construct a unitary 
equivalence between 
$q^w(
x,D_x)$ acting on $L^2(\Bbb{R}^n)$ and an operator $P$ as in \eqref{eq.def.P} acting on a space $H_\Phi$ for $\Phi$ obeying \eqref{eq.Phi.assumptions}.

In order to study certain operators such as the Fokker-Planck quadratic model, the hypotheses of ellipticity need to be weakened, as discussed in such works as \cite{HiPS2009} and \cite{HeSjSt2005}.  In this setting, one retains the hypothesis
\begin{equation}\label{eq.real.semidef}
	\Re q(x,\xi) \geq 0, \quad \forall (x,\xi) \in \Bbb{R}^{2n},
\end{equation}
but one only assumes definiteness of $\Re q$ after averaging along the flow of the Hamilton vector field $H_{\Im q} = 2\Im F$.  In \cite{HiPS2009}, this condition was put in terms of an index depending on the fundamental matrix \eqref{eq.def.F}:
\begin{equation}\label{eq.real.subell.index}
	J(x,\xi) = \min\{k \in \Bbb{N} \::\: \Re F (\Im F)^k (x,\xi) \neq 0\}.
\end{equation}
Under the hypothesis 
\begin{equation}\label{eq.real.subell.hypothesis}
	J(x,\xi) < \infty, \quad \forall (x,\xi) \in \Bbb{R}^{2n}\backslash \{0\},
\end{equation}
the semigroup $\exp(-tq^w(x,D_x))$, for $t > 0$, possesses strong regularization properties.

In Section \ref{subsec.subellipticity}, we arrive at a natural weak ellipticity condition in terms of the dynamics of $\Phi(e^{tM}z)$ as a function of $t$.  It is unsurprising, but worthy of note, that these two conditions are identical and their associated coefficients are closely related, as formulated in Proposition \ref{prop.subell.relation} below.

To finish the discussion of operators on the real side, we demonstrate, by appealing to a well-known pseudomode construction, the non-existence of the resolvent for a quadratic operator for which the so-called bracket condition fails at some $(x_0, \xi_0) \in q^{-1}(\{0\})$. Many of the essential ideas were present in the fundamental work of H\"ormander \cite{Ho1960}, as noted in \cite{Zw2001}, and here we rely on the celebrated work \cite{DeSjZw2004}. We recall that the Poisson bracket of two symbols $f,g : \Bbb{R}^{2n}\to \Bbb{C}$ is
\[
	\{f,g\} = \sum_{j=1}^n \frac{\partial f}{\partial \xi}\cdot \frac{\partial g}{\partial x} - \frac{\partial g}{\partial \xi} \cdot \frac{\partial f}{\partial x}.
\]
We recall from \cite[Lem.~2]{PS2008a} that this has a simple expression in the quadratic case using the fundamental matrix defined in \eqref{eq.def.F}: if $q_1, q_2:\Bbb{R}^{2n} \to \Bbb{C}$ are quadratic, then
\[
	F(\{q_1, q_2\}) = -2[F(q_1), F(q_2)].
\]
When the symbol $f$ of a Weyl quantization is homogeneous (and obeys appropriate hypotheses) and $\{\Im f, \Re f\} > 0$ for all $(x,\xi) \in \Omega$ an appropriate open set, a scaling argument and \cite[Thm.~1.2]{DeSjZw2004} shows that the resolvent of $f^w(x,D_x)$ either has a rapidly-growing norm or does not exist in $h^{-1}f(\Omega)$ as $h \to 0^+$. Following this route, we see that the resolvent of the Weyl quantization of a quadratic form $q$ cannot exist anywhere if the bracket fails to vanish on $q^{-1}(\{0\})$.

\begin{theorem}\label{thm.spectrum.C}
Let $q:\Bbb{R}^{2n} \to \Bbb{C}$ be a quadratic form such that there exists $(x_0, \xi_0) \in \Bbb{R}^{2n}$ for which
\[
	q(x_0, \xi_0) = 0
\]
and
\begin{equation}\label{eq.Twist.Condition}
	\{\Im q, \Re q\}(x_0, \xi_0) \neq 0.
\end{equation}
Then, for the maximal realization of $q^w(x,D_x)$ on $L^2(\Bbb{R}^n)$,
\[
	\opnm{Spec} q^w(x,D_x) = \Bbb{C}.
\]
\end{theorem}

\begin{proof}
	We show that if $\{\Im q, \Re q\}(x_0, \xi_0) > 0$, then $\|(z-q^w(x,D_x))^{-1}\| = \infty$ for all $z \in \Bbb{C}$.  If $\{\Im q, \Re q\}(x_0, \xi_0) < 0$, then we recall \cite[p.~426]{Ho1995} that the symbol of the adjoint $q^w(x,D_x)^*$ is $\overline{q(x,\xi)}$. Since $\{\Im \bar{q}, \Re \bar{q}\} = - \{\Im q, \Re q\}$, we see that $\|(\bar{z} - q^w(x,D_x)^*)^{-1}\| = \infty$ for all $z \in \Bbb{C}$, which suffices to show that the resolvent set is empty.
	
We therefore assume that
	\begin{equation}\label{eq.Twist.Condition.2}
		\{\Im q, \Re q\}(x_0, \xi_0) > 0.
	\end{equation}
	As a consequence, $\nabla \Im q(x_0, \xi_0)$ and $\nabla \Re q(x_0, \xi_0)$ are linearly independent.  Using also that \eqref{eq.Twist.Condition.2} is an open condition in $(x_0, \xi_0)$, let $r_0, r_1, c > 0$ be sufficiently small such that
	\[
		\{\Im q, \Re q\}(x, \xi) \geq c, \quad \forall (x,\xi) \in B((x_0, \xi_0), r_0)
	\]
	and such that
	\[
		B(0, r_1) \subset q(B((x_0, \xi_0), r_0)) \subset \Bbb{C}.
	\]

	Then, by \cite[Thm.~1.2]{DeSjZw2004}, there exist $h_0 > 0$ sufficiently small and $C > 0$ sufficiently large such that
	\[
		\|(q^w(x,hD_x) - z)^{-1}\| \geq \frac{1}{C}e^{1/(Ch)}, \quad \forall h \in (0, h_0], \quad \forall z \in B(0,r_1).
	\]
	(As usual, we write $\|(A-z)^{-1}\| = +\infty$ if $z \in \opnm{Spec}A$.) Using the standard scaling
	\[
		\tilde{\mathcal{V}}_{\sqrt{h}} u(x) = h^{n/4}u(\sqrt{h}x),
	\]
	which is unitary on $L^2(\Bbb{R}^n)$ and for which
	\[
		\tilde{\mathcal{V}}_{\sqrt{h}} q^w(x,hD_x) \tilde{\mathcal{V}}_{\sqrt{h}}^* = hq^w(x,D_x),
	\]
	we see that
	\[
		\|(q^w(x,D_x) - z)^{-1}\| = \left\| \left(\frac{1}{h}(q^w(x,hD_x) - hz)\right)^{-1}\right\| \geq \frac{h}{C}e^{1/(Ch)}
	\]
	so long as $|z| < r_1/h$ and $0 < h \leq h_0$.  Since $(h/C)e^{1/(Ch)}\to \infty$ and $r_1/h \to \infty$ as $h \to 0^+$, this shows that the resolvent cannot be a bounded operator for any $z \in \Bbb{C}$.
\end{proof}

\subsection{Unitary equivalence with Fock spaces}

We now summarize a method of reducing certain quadratic operators $q^w(x,D_x)$ acting on $L^2(\Bbb{R}^n)$ to operators on Fock spaces $H_\Phi$ of the form $P = Mz\cdot\partial_z$ as in \eqref{eq.def.P}, up to an additive constant.  If such a reduction exists, as determined in Proposition \ref{prop.which.q}, one can apply the results of Section \ref{sec.Fock} to find the eigenvalues of $q^w(x,D_x)$ as well as the weak definition of $\exp(\tau q^w(x,D_x))$ for $\tau \in \Bbb{C}$ and its properties.

For $\Phi$ obeying \eqref{eq.Phi.assumptions} decomposed as in \eqref{eq.decompose.Phi}, let the symmetric matrix $H$ be as in \eqref{eq.def.H} so that
\[
	\Phi(G^{-1}z) = \frac{1}{2}\left(|z|^2 - \Re (z\cdot H z)\right).
\]
In order to associate the space $H_\Phi$ with $L^2(\Bbb{R}^n)$, we follow \cite[Sec.~2.2,~4.1]{Vi2013} in creating an adapted Fourier-Bros-Iagolnitzer (FBI) transform.  For details as well as deeper analysis and applications, the reader may consult among others the works \cite[Ch.~13]{ZwBook}, \cite{MaBook}, or \cite{SjSAM}.

To define this transform, let
\begin{equation}\label{eq.def.FBI.A}
	A = i(1+H)^{-1}(1-H),
\end{equation}
where it follows automatically that $\Im A > 0$ in the sense of positive definite matrices because $\|H\| < 1$; see Lemma \ref{lem.decomposition}. Let the holomorphic quadratic phase $\varphi$ be defined by
\[
	\varphi(z,x) = \frac{i}{2}(z-x)^2 - \frac{1}{2}z\cdot\left((1-iA)^{-1}A z\right).
\]
Then for $v \in L^2(\Bbb{R}^n)$, we define the FBI transform
\begin{equation}\label{eq.FBI.initial}
	\mathcal{T}_0 v(z) = C_\varphi \int_{\Bbb{R}^n} e^{i\varphi(z,x)}v(x)\,dx.
\end{equation}
For the correct choice of $C_\varphi$, the map $\mathcal{T}_0$ is unitary from $L^2(\Bbb{R}^n)$ to $H_{\Phi_0}(\Bbb{C}^n)$ with
\[
	\begin{aligned}
	\Phi_0(z) &= \sup_{x \in \Bbb{R}^n} \left(-\Im \varphi(z,x)\right)
	\\ &= \frac{1}{4}\left(|z|^2 - \Re (z\cdot Hz)\right).
	\end{aligned}
\]
We may compose this transform with the unitary change of variables $\mathcal{V}_{\sqrt{2}G}$ as in \eqref{eq.def.V} to arrive at $\Phi$ as in \eqref{eq.decompose.Phi}. We therefore let
\begin{equation}\label{eq.def.FBI.final}
	\mathcal{T} = \mathcal{V}_{\sqrt{2}G} \circ \mathcal{T}_0 : L^2(\Bbb{R}^n) \to H_\Phi.
\end{equation}

The role here of conjugation by the FBI transform is to simplify the symbols of Weyl quantizations.  From \cite[Eq.~(12.37)]{SjLoR} we have that
\[
	\mathcal{T}_0 a^w(x,D_x)\mathcal{T}_0^* = (a\circ \kappa_0^{-1})^w(z,D_z)
\]
for symbols $a:\Bbb{R}^{2n}\to \Bbb{C}$ in standard symbol classes, certainly including polynomials of degree two, with the canonical transformation $\kappa_0$ defined via
\[
	\kappa_0(x, -\varphi'_x(x,z)) = (z, \varphi'_z(x,z)), \quad \forall x,z \in \Bbb{C}^n.
\]
Conjugating with the change of variables $\mathcal{V}_{\sqrt{2}G}$ can be seen a more elementary fashion to act on symbols by composing with the canonical transformation $\kappa_{\sqrt{2}G}(z,\zeta) = ((\sqrt{2}G)^{-1}z, (\sqrt{2}G)^\top \zeta)$.  Composing the two, we get
\begin{equation}\label{eq.canonical.relation}
	\begin{aligned}
	q_1^w(x,D_x) &= \mathcal{T}^* p^w(z,D_z)\mathcal{T},
	\\ q_1 &= p\circ \kappa,
	\end{aligned}
\end{equation}
for the complex linear canonical transformation
\begin{equation}\label{eq.def.kappa}
	\kappa = \frac{1}{\sqrt{2}}\left(\begin{array}{cc} G^{-1} & -iG^{-1}
		\\ -2 G^\top A(1-iA)^{-1} & 2 G^\top(1-iA)^{-1}\end{array}\right)
\end{equation}
with $A$ as in \eqref{eq.def.FBI.A}. For future reference, we therefore write
\begin{equation}\label{eq.def.Q}
	Q_1 = \mathcal{T}^*P\mathcal{T} + \frac{1}{2}\opnm{Tr} M.
\end{equation}

This may be regarded as a partial analogue, for complex linear canonical transformations, of the well-known fact \cite[Lem.~18.5.9]{HoALPDO3} that, when $\chi$ is a real linear canonical transformation, we may always find a simple unitary operator $\mathcal{U}_\chi: L^2(\Bbb{R}^n) \to L^2(\Bbb{R}^n)$ such that
\begin{equation}\label{eq.canonical.real}
	\mathcal{U}_\chi q^w(x,D_x)\mathcal{U}_\chi^* = (q\circ\chi^{-1})^w(x,D_x).
\end{equation}
More specifically, this operator can be decomposed as a composition of changes of variables, multiplication by exponentials of imaginary quadratic forms, and partial Fourier transforms.

\begin{remark}\label{rem.Fock.phase.space}
We recall that there is a classical equivalence between the values of the symbol on the real and the Fock space sides: for any $(x,\xi) \in \Bbb{R}^{2n}$, we have that
\[
	\kappa(x,\xi) = (z,-2i\Phi'_z(z))
\]
for some $z \in \Bbb{C}^n$, and in fact the map $(x,\xi)\mapsto z$ formed by composing $\kappa$ with projection onto the first coordinate is a real-linear bijection; see \cite[Sec.\ 1]{Sj1996}.  This shows that conditions \eqref{eq.ellipticity} and \eqref{eq.real.semidef} are equivalent if the symbols $p(z,\zeta) = (Mz)\cdot(i\zeta)$ and $q_1(x,\xi)$ are related by \eqref{eq.canonical.relation}.  Furthermore, \eqref{eq.real.semidef} is invariant under composition of $q$ with real canonical transformations, so \eqref{eq.ellipticity} and \eqref{eq.real.semidef} are equivalent.
\end{remark}

We now have established the required vocabulary to identify the real-side symbols which may be treated in the framework of this paper.

\begin{proposition}\label{prop.which.q}
Let $q(x,\xi):\Bbb{R}^{2n}\to\Bbb{C}$ be quadratic.  Then the following are equivalent:
\begin{enumerate}[(i)]
	\item \label{it.which.q.reduction} there exists a unitary transformation $\mathcal{U}_\chi:L^2(\Bbb{R}^n) \to L^2(\Bbb{R}^n)$ of the form in \eqref{eq.canonical.real} and an FBI transform $\mathcal{T}$ of the form in \eqref{eq.def.FBI.final} such that
	\begin{equation}\label{eq.q.p.unitary}
		\mathcal{T}\mathcal{U}_\chi q^w(x,D_x) \mathcal{U}^*_\chi \mathcal{T}^* = p^w(z,D_z)
	\end{equation}
	for $p(z,\zeta) = Mz\cdot(i\zeta)$ as in \eqref{eq.P.symbol},
	\item \label{it.which.q.subspaces} there exist two invariant subspaces $\Lambda^+$ and $\Lambda^-$ of the fundamental matrix $F = F(q)$ which are positive and negative definite Lagrangian planes as in \eqref{eq.Lagrangian.graph}, and
	\item there exist matrices $A_\pm \in \Bbb{M}_{n\times n}(\Bbb{C})$, with $A_\pm^\top = A_\pm$ and $\pm \Im A_\pm > 0$ in the sense of positive definite matrices, and a matrix $B \in \Bbb{M}_{n\times n}(\Bbb{C})$ for which
	\begin{equation}\label{eq.real.symbol.ip.form}
		q(x,\xi) = B(\xi-A_- x)\cdot (\xi - A_+ x).
	\end{equation}
\end{enumerate}
\end{proposition}

\begin{remark*}
Since the intersection of a positive and a negative Lagrangian plane must be trivial, it follows automatically that $\Lambda^+ \oplus \Lambda^- = \Bbb{C}^{2n}$.
\end{remark*}

\begin{remark*}
Following Proposition \ref{prop.linear}, we may also obtain some results for the Weyl quantization of any polynomial of degree 2 including linear and constant terms, so long as the quadratic part satisfies the hypotheses of Proposition \ref{prop.which.q} above.
\end{remark*}

\begin{proof}
From \eqref{eq.F.sigma1} and \eqref{eq.F.sigma2}, if $\mathcal{K}$ is a canonical linear transformation, then
\begin{equation}\label{eq.F.transform.K}
	F(q\circ \mathcal{K}) = \mathcal{K}^{-1} F(q) \mathcal{K}.
\end{equation}
The property of being a Lagrangian subspace is preserved by all linear canonical transformations; the property that a Lagrangian plane is positive or negative definite is preserved by all real linear canonical transformations (meaning those that preserve $\Bbb{R}^{2n}$ or equivalently those given by matrices with real entries).

We note that, for $p(z,\zeta)$ in \eqref{eq.P.symbol}, we have
\begin{equation}\label{eq.F.p}
	F(p) = \frac{1}{2}\left(\begin{array}{cc} M & 0 \\ 0 & -M^\top \end{array}\right)
\end{equation}
which has the invariant subspaces $\{(z,\zeta)\::\: \zeta = 0\}$ and $\{(z,\zeta) \::\: z = 0\}$. If the reduction in (\ref{it.which.q.reduction}) exists, $F(q_1) = F(p\circ \kappa)$ has invariant subspaces
\[
	\begin{aligned}
	\Lambda_1^+ &:= \kappa^{-1}(\{\zeta = 0\}) = \{(x,Ax)\}_{x \in \Bbb{C}^n},
	\\ \Lambda_1^- &:= \kappa^{-1}(\{z = 0\}) = \{(x, -ix)\}_{x \in \Bbb{C}^n}.
	\end{aligned}
\]
That $\Im A > 0$ is equivalent to strict convexity of $\Phi$; see \cite[Eq.~(2.8)]{Vi2013}. That $\Lambda^\pm_1$ are positive and negative definite Lagrangian planes then follows from \eqref{eq.Lagrangian.graph}.  These properties persist for $\Lambda^\pm := \chi^{-1}(\Lambda^\pm_1)$, which are invariant subspaces of $F(q)$, proving that the existence of $\Lambda^\pm$ is a necessary condition for the reduction to an operator $P$ described in the statement of the proposition.

Conversely, if $\Lambda^\pm$ exist, the construction of $\chi$ and $\kappa$ for which
\begin{equation}\label{eq.Lambda.pm.reduction}
	\begin{aligned}
	\kappa\circ\chi(\Lambda^+) &= \{\zeta = 0\},
	\\ \kappa\circ\chi(\Lambda^-) &= \{z = 0\}
	\end{aligned}
\end{equation}
may be found in \cite[Sec.~2]{HiSjVi2013} or with a few more details in \cite[Prop.~2.2]{Vi2013}; both essentially follow the ideas of \cite[Sec.~3]{Sj1974}.  The fact that $p := q \circ \chi^{-1}\circ \kappa^{-1}$ is of the form $Mz\cdot i\zeta$ follows from checking through \eqref{eq.F.sigma1} that $p''_{zz} = p''_{\zeta\zeta} = 0$ since $\{\zeta = 0\}$ and $\{z = 0\}$ are Lagrangian and $F(p)$-invariant.  If desired, one may put $M$ in Jordan normal form through a change of variables.

In order to establish that it is necessary and sufficient that $q(x,\xi)$ can be put in the form \eqref{eq.real.symbol.ip.form}, begin by supposing that the decomposition \eqref{eq.real.symbol.ip.form} holds and let
\[
	\ell_{\pm}(x,\xi) = \xi - A_\pm x,
\]
noting that these are linear maps of rank $n$ from $\Bbb{C}^{2n}$ to $\Bbb{C}^n$ with kernels
\[
	\ker \ell_{\pm}(x,\xi) = \Lambda^\pm := \{\xi = A_\pm x\}.
\]
Therefore, using $\chi$ and $\kappa$ from \eqref{eq.Lambda.pm.reduction},
\[
	k_\pm(z,\zeta) := \ell_\pm\circ\chi^{-1}\circ\kappa^{-1}(z,\zeta)
\]
are two rank-$n$ linear forms from $\Bbb{C}^{2n}$ to $\Bbb{C}^n$ with kernels $\ker k_+ = \{\zeta = 0\}$ and $\ker k_- = \{z = 0\}.$ Therefore $k_+ = F_+\zeta$ and $k_- = F_-z$ for some invertible matrices $F_\pm$, proving that
\[
	q\circ\chi^{-1}\circ\kappa^{-1}(z,\zeta) = (F_+^\top B F_-z)\cdot \partial_z,
\]
establishing that (\ref{it.which.q.reduction}) is satisfied.

Alternatively, we compute that, under the form \eqref{eq.real.symbol.ip.form},
\[
	F(q) = \frac{1}{2}\left(\begin{array}{cc} -B^\top A_+ - B A_- & B + B^\top \\ -A_+ B A_- - A_- B^\top A_+ & A_+ B + A_- B^\top\end{array}\right).
\]
From there it is easy to check directly that $\{(x,A_\pm x)\}$ are invariant subspaces of $F(q)$, because for instance
\[
	F(q)(x,A_+ x) = \frac{1}{2}(B(A_+-A_-)x, A_+B(A_+-A_-)x).
\]
This establishes (\ref{it.which.q.subspaces}) instead.

Conversely, supposing that (\ref{it.which.q.reduction}) holds, we simply reverse the process with $\tilde{k}_+ = \zeta$ and $\tilde{k}_- = z$.  With
\[
	\tilde{\ell}_\pm(x,\xi) = \tilde{k}_\pm\circ\kappa\circ\chi(x,\xi)
\]
we have two rank-$n$ linear forms with kernels 
\[
	\ker \tilde{\ell}_+ = \chi^{-1} \kappa^{-1} (\{\zeta = 0\})
\]
and
\[
	\ker \tilde{\ell}_- = \chi^{-1} \kappa^{-1} (\{z = 0\}).
\]
Since these must be positive and negative definite Lagrangian planes, we can write
\[
	\Lambda^\pm := \ker \tilde{\ell}_\pm = \{\xi = A_\pm x\}
\]
for symmetric matrices $A_\pm$ with sign-definite imaginary parts.  As a consequence, $G_\pm := (\tilde{\ell}_\pm)'_\xi$ must be invertible, so we can check that 
\[
	G_\pm^{-1}\tilde{\ell}_\pm(x,\xi) = \xi - A_\pm x
\]
since the coefficient of $\xi$ is the identity matrix and the coefficient of $x$ is then identified by the kernel. Since $p = M\tilde{k}_- \cdot \tilde{k}_+$, we have that 
\[
	\begin{aligned}
	q(x,\xi) & = p \circ \kappa \circ \chi(x,\xi)
	\\ & = M\tilde{\ell}_-(x,\xi) \cdot \tilde{\ell}_+(x,\xi)
	\\ & = MG_-(\xi - A_-x) \cdot G_+(\xi - A_+x).
	\end{aligned}
\]
This proves that \eqref{eq.real.symbol.ip.form} holds with $B = G_+^\top MG_-$.
\end{proof}

\begin{corollary}\label{cor.Lambda.pm.spectra}
If $q(x,\xi) : \Bbb{R}^{2n} \to \Bbb{C}$ is a quadratic form obeying condition (\ref{it.which.q.subspaces}) in Proposition \ref{prop.which.q}, then, with the fundamental matrix $F$ as in \eqref{eq.def.F},
\[
	\opnm{Spec} F|_{\Lambda^+} = -\opnm{Spec} F|_{\Lambda^-},
\]
including algebraic and geometric multiplicities.  Furthermore, under the relation between $q(x,\xi)$ and $p(z,\zeta) = Mz\cdot i\zeta$ in part (\ref{it.which.q.reduction}) of Proposition \ref{prop.which.q}, we have 
\[
	\opnm{Spec}M = \frac{2}{i}\opnm{Spec}F|_{\Lambda^+}.
\]
\end{corollary}

\begin{proof}
Using the reduction (\ref{it.which.q.reduction}) from Proposition \ref{prop.which.q} and writing $\mathcal{K} = \kappa \circ \chi$, we have that
\[
	F(p) = \frac{i}{2}\left(\begin{array}{cc} M & 0 \\ 0 & -M^\top \end{array}\right) = \mathcal{K} F(q) \mathcal{K}^{-1}.
\]
Since $\mathcal{K}:\Lambda^+ \to \{\zeta = 0\}$ and $\mathcal{K}:\Lambda^- \to \{z = 0\}$ are linear bijections, we have that $F|_{\Lambda^+}$ is similar to $\frac{i}{2}M$ and $F|_{\Lambda^-}$ is similar to $-\frac{i}{2}M^\top$.  The result follows.
\end{proof}

Under the natural assumption that $\opnm{Spec} F|_{\Lambda^+}$ is contained in a proper half-plane --- which appears in, for instance, Proposition \ref{prop.exptM.asymptotic} --- we have that the hypothesis in Proposition \ref{prop.which.q} is stable.

\begin{corollary}
Let $q(x,\xi) : \Bbb{R}^{2n} \to \Bbb{C}$ be a quadratic form obeying the conditions in Proposition \ref{prop.which.q} and for which 
\[
	\opnm{Spec}F|_{\Lambda^+}(q) \subset \{\Re e^{i\theta}\lambda > 0\}
\]
for some $\theta \in \Bbb{R}$. Then there exists some $\eps > 0$ such that, if $\tilde{q}:\Bbb{R}^{2n}\to\Bbb{C}$ is another quadratic form with $\|\tilde{q}''\| \leq \eps$, then $q+\tilde{q}$ also obeys the conditions in Proposition \ref{prop.which.q}.
\end{corollary}

\begin{proof}
We follow \cite[p.~97]{Sj1974}. We may assume without loss of generality that $\theta = 0$, and by Corollary \ref{cor.Lambda.pm.spectra} we have
\[
	\opnm{Spec}F(q)|_{\Lambda^\pm} = \opnm{Spec}F(q) \cap \{\pm \Re \lambda > 0\}.
\]	
Then $\Lambda^+(q)$ may be realized as the image of
\[
	P(q) = \frac{1}{2\pi i}\int_\Gamma (z-F(q))^{-1}\,dz
\]
for $\Gamma = i[-R, R] \cup \{|z| = R, \Re z > 0\}$ for $R$ sufficiently large that $\Gamma$ surrounds all the eigenvalues of $F|_{\Lambda^+}$.  We can express $\Lambda^-$ similarly.  That $\Lambda^+$ and $\Lambda^-$ are positive and negative Lagrangian planes is an open condition in $F$ (again referring to \cite[p.~97]{Sj1974}), as is the fact that the eigenvalues of $F|_{\Lambda^+}$ are contained in the right half-plane.  Therefore a sufficiently small change in the coefficients of $q$ cannot change condition (\ref{it.which.q.subspaces}) in Proposition \ref{prop.which.q}, and the corollary follows.
\end{proof}

As an illustration of \eqref{eq.canonical.relation} and to understand how decay in Fock spaces is related to smoothness and decay on the real side, we study the Hermite functions
\begin{equation}\label{eq.def.Hermite}
	h_\alpha(x) = \frac{1}{\sqrt{2^{|\alpha|} \alpha! \sqrt{\pi^n}}}(x-\partial_x)^\alpha e^{-x^2/2},
\end{equation}
which form an orthonormal basis of eigenfunctions for the harmonic oscillator $Q_0$ defined in \eqref{eq.def.h.o}.

\begin{proposition}\label{prop.Hermite}
With $\mathcal{T}$ in \eqref{eq.def.FBI.final}, the Hermite functions $\{h_\alpha\}_{\alpha \in \Bbb{N}^n}$ in \eqref{eq.def.Hermite}, and the orthonormal basis $\{e_\alpha\}_{\alpha \in \Bbb{N}^n}$ defined in \eqref{eq.def.e.alpha}, there exists some constant $c \in \Bbb{C}$ with $|c| = 1$ such that
\begin{equation}\label{eq.prop.hermite.1}
	\mathcal{T}h_\alpha = ce_\alpha.
\end{equation}
Furthermore, with $Q_0$ from \eqref{eq.def.h.o} and $P_0$ from \eqref{eq.def.P0},
\begin{equation}\label{eq.prop.hermite.2}
	\mathcal{T}Q_0\mathcal{T}^* = P_0.
\end{equation}
\end{proposition}

\begin{proof}
The Hermite functions are uniquely determined, up to a constant multiple of modulus one, by the creation operators $(x-\partial_x)$, regarded as an $n$-vector of operators, and the fact that $h_0$ is an $L^2(\Bbb{R}^n)$-normalized function in the kernel of the annihilation operators $(x+\partial_x)$.

Inverting $\kappa$ in \eqref{eq.def.kappa}, we see that the Weyl symbol of the creation operators is
\begin{equation}\label{eq.creation.Fock}
	\begin{aligned}
	\left.(x-i\xi)\right|_{(x,\xi) = \kappa^{-1}(z,\zeta)} & = \sqrt{2}\left((1-iA)^{-1} Gz + \frac{i}{2}(G^\top)^{-1}\zeta\right) 
		\\ &~\quad - i\sqrt{2}\left(A(1-iA)^{-1} Gz + \frac{1}{2}(G^\top)^{-1}\zeta\right)
	\\ &= \sqrt{2}G z.
	\end{aligned}
\end{equation}
Recalling the definition of $A$ in \eqref{eq.def.kappa}, the Weyl symbol of the annihilation operators may be computed similarly: 
\begin{equation}\label{eq.annihilation.Fock}
	\left.(x+i\xi)\right|_{(x,\xi) = \kappa^{-1}(z,\zeta)} = \sqrt{2}\left(HGz+i(G^\top)^{-1}\zeta\right).
\end{equation}

From \eqref{eq.decompose.Phi} and the definition \eqref{eq.def.H} of $H$, we see that the annihilation operators
\[
	\sqrt{2}(G^\top)^{-1}\left(G^\top H G z + \partial_z\right) = \sqrt{2}(G^\top)^{-1}\left(h'(z) + \partial_z\right)
\]
applied to $e_0$ give zero and we know that $\|e_0\|_{\Phi} = 1$.  Therefore
\[
	e_0 = c\mathcal{T}h_0
\]
for some $c$ with $|c| = 1$. We therefore have \eqref{eq.prop.hermite.1} since
\[
	\begin{aligned}
	e_\alpha &= \frac{1}{\sqrt{2^{|\alpha|} \alpha!}}(\sqrt{2}Gz)^\alpha e_0
	\\ &= \frac{1}{\sqrt{2^{|\alpha|}\alpha!}}\mathcal{T}(x-\partial_x)^\alpha\mathcal{T}^* c\mathcal{T}h_0
	\\ &= c\mathcal{T}h_\alpha.
	\end{aligned}
\]
The equivalence \eqref{eq.prop.hermite.2} follows from the computation
\[
	\begin{aligned}
	\mathcal{T}Q_0\mathcal{T}^* &= \mathcal{T} \frac{1}{2}(x -\partial_x)\cdot(x+\partial_x)\mathcal{T}^*
	\\ &= \frac{1}{2}\mathcal{T}(x-\partial_x)\mathcal{T}^* \cdot \mathcal{T}(x+\partial_x)\mathcal{T}^*
	\\ &= \frac{1}{2}\sqrt{2}Gz \cdot \sqrt{2}(G^\top)^{-1}(\partial_z + h'(z))
	\\ &= P_0.
	\end{aligned}
\]
\end{proof}

We now state the equivalence between the real and Fock space weak ellipticity conditions.

\begin{proposition}\label{prop.subell.relation}
	Let $p(z,\zeta) = Mz\cdot (i\zeta)$ and $q(x,\xi)$ be related through $q = p \circ \mathcal{K}$ with $\mathcal{K} = \kappa \circ \chi$ as in part (\ref{it.which.q.reduction}) of Proposition \ref{prop.which.q}. Recall the definition of $\Theta$ from \eqref{eq.def.Theta}, the real-side index $J(x,\xi)$ from \eqref{eq.real.subell.index} above, and the Fock-space index $I(z)$ from \eqref{eq.def.subellipticity} below.  Assume that \eqref{eq.real.semidef}, or equivalently \eqref{eq.ellipticity}, holds.  
	
	Then, for every $(x,\xi) \in \Bbb{R}^{2n}\backslash \{0\}$,
	\[
		J(x,\xi) = I(z),
	\]
	where $(x,\xi)$ and $z$ are related by
	\begin{equation}\label{eq.subell.reln.K}
		(z, -2i\Phi'_z(z)) = \mathcal{K}(x,\xi),
	\end{equation}
	recalling that $(x,\xi) \stackrel{\kappa}\mapsto (z,\zeta) \mapsto z$ is a real linear bijection from $\Bbb{R}^{2n}$  to $\Bbb{C}^n$ and therefore so is $(x,\xi) \stackrel{\mathcal{K}}\mapsto (z,\zeta) \mapsto z$. Furthermore,
	\begin{equation}\label{eq.subell.equiv.value}
		\Re q((\Im F)^{J(x,\xi)}(x,\xi)) = 4^{-I(z)}\Theta(M^{I(z)}z).
	\end{equation}
\end{proposition}

In order to take advantage of tools introduced in Section \ref{subsec.subellipticity}, we reserve the proof for  Appendix \ref{sec.appendix.subell}.

\subsection{Comparison of operators of harmonic oscillator type}\label{subsec.compare.h.o}

Consider $q:\Bbb{R}^{2n}\to \Bbb{C}$ quadratic and satisfying the hypotheses of Proposition \ref{prop.which.q}. Combining Theorem \ref{thm.boundedness.delta} with Propositions \ref{prop.which.q} and \ref{prop.Hermite} allows us to describe the set of $\delta \in \Bbb{R}$ depending on $\tau \in \Bbb{C}$ for which
\[
	\exp(\delta \tilde{Q}_0)\exp(\tau q^w(x,D_x)) \in \mathcal{L}(L^2(\Bbb{R}^n)),
\]
with $\tilde{Q}_0$ a self-adjoint operator unitarily equivalent to the harmonic oscillator \eqref{eq.def.h.o}.  Specifically,
\[
	\tilde{Q}_0 = \mathcal{U}_\chi^* Q_0 \mathcal{U}_\chi
\]
with $\mathcal{U}_\chi$ taken from (\ref{it.which.q.reduction}) in Proposition \ref{prop.which.q}.  Since the Weyl symbol of $Q_0$ is 
\[
	\frac{1}{2}(x^2 + \xi^2 - n),
\]
we conclude from \eqref{eq.canonical.real} that
\[
	\tilde{Q}_0 = \tilde{q}_0^w(x,D_x) - \frac{n}{2}
\]
with 
\[
	\tilde{q}_0(x,\xi) = \left.\frac{1}{2}(y^2 + \eta^2)\right|_{(y,\eta) = \chi(x,\xi)}.
\]

It is not immediately apparent how regularization properties of $\exp(\delta \tilde{Q}_0)$ depend on $\tilde{Q}_0$ and specifically $\chi$.  We therefore consider families of spaces $\{\exp(\delta Q) \::\: \delta \in \Bbb{R}\}$ for $Q$ of harmonic oscillator type, focusing on the question of whether and to what extent this family of spaces depends on the choice of $Q$.  When saying that $Q$ is of harmonic oscillator type, we here mean that $Q$ is the Weyl quantization as in \eqref{eq.Weyl} of a real-valued positive definite quadratic form on $\Bbb{R}^{2n}$.

For $Q_1, Q_2$ both of harmonic oscillator type, we consider $\delta_1, \delta_2 > 0$ and study sufficient conditions to have
\begin{equation}\label{eq.mixed.h.o}
	\exp(\delta_2 Q_2)\exp(-\delta_1 Q_1) \in \mathcal{L}(L^2(\Bbb{R}^n)).
\end{equation}
The operator $\exp(\delta_2 Q_2)$ is certainly unbounded but may be understood weakly either in the sense of Proposition \ref{prop.solve.exptP} after a conjugation like in Proposition \ref{prop.which.q} or as a formal sum extended from the span of its orthonormal basis of eigenvectors.  If \eqref{eq.mixed.h.o} holds, then 
\begin{equation}\label{eq.mixed.h.o.set.compare}
	\exp(-\delta_1 Q_1)L^2(\Bbb{R}^n) \subset \exp(-\delta_2 Q_2)L^2(\Bbb{R}^n),
\end{equation}
since we can realize any element in the set on the left-hand side as the product of $\exp(-\delta_2 Q_2)$ times the aforementioned bounded operator applied to an element of $L^2(\Bbb{R}^n)$.

We cannot perform a Fock-space reduction on both $Q_1$ and $Q_2$ simultaneously. We may, however, bridge the gap between $Q_1$ and $Q_2$ by introducing an operator $Q_3$, generally non-normal, where for certain $\delta_1, \delta_2, t \in \Bbb{R}$ we have
\begin{equation}\label{eq.mixed.h.o.right}
	\exp(tQ_3)\exp(-\delta_1 Q_1) = (\exp(-\delta_1 Q_1)\exp(tQ_3^*))^* \in \mathcal{L}(L^2(\Bbb{R}^n))
\end{equation}
and
\begin{equation}\label{eq.mixed.h.o.left}
	\exp(\delta_2 Q_2)\exp(-t Q_3) \in \mathcal{L}(L^2(\Bbb{R}^n)),
\end{equation}
from which \eqref{eq.mixed.h.o} follows. (In the proof which follows, we justify the equality in \eqref{eq.mixed.h.o.right} by checking against dense subsets of $L^2(\Bbb{R}^n)$.) This strategy, combined with the Fock-space analysis already established, yields the following theorem, which gives sufficient conditions for \eqref{eq.mixed.h.o} to hold for $\delta_1, \delta_2$ small and a sharp characterization of the maximum $\delta_2$ for which \eqref{eq.mixed.h.o} can hold.

\begin{theorem}\label{thm.mixed.h.o}
Let $q_j :\Bbb{R}^{2n}\to\Bbb{R}$, for $j = 1, 2$, be two real-valued quadratic forms which are positive definite in the sense that $q_j(x,\xi) > 0$ for all $(x,\xi) \in \Bbb{R}^{2n} \backslash \{0\}$.  Write $Q_j = q_j^w(x,D_x)$. Let $u_{0, j} \neq 0$ be ground states for the operators $Q_j$, meaning that
\[
	Q_j u_{0, j} = \mu_{0,j}u_{0,j}, \quad \mu_{0,j} = \min \opnm{Spec} Q_j.
\]
\begin{enumerate}[(i)]
\item\label{it.mixed.small} There exist constants $C, \delta_0 > 0$ such that
\begin{equation}\label{eq.mixed.h.o.result}
	\exp(\frac{\delta}{C} Q_2) \exp(-\delta Q_1) \in \mathcal{L}(L^2(\Bbb{R}^n)), \quad \forall \delta \in [0, \delta_0).
\end{equation}
\item\label{it.mixed.h.o.match} If $Q_1$ and $Q_2$ share ground states, meaning that $\opnm{span}u_{0,1} = \opnm{span} u_{0,2}$, then we may take $\delta_0 = \infty$ in \eqref{eq.mixed.h.o.result}.
\item\label{it.mixed.large} If $\opnm{span}u_{0,1} \neq \opnm{span} u_{0,2}$, then there exists $\tilde{\Delta}_0 > 0$ such that
\begin{equation}\label{eq.mixed.ground.state.explodes}
	\exp(\tilde{\Delta}_0 Q_2) u_{0,1} \notin L^2(\Bbb{R}^n)
\end{equation}
and such that, for every $\delta_2 < \tilde{\Delta}_0$, there exists $\delta_1 > 0$ such that \eqref{eq.mixed.h.o} holds.
\end{enumerate}
\end{theorem}

\begin{remark}\label{rem.Lipschitz}
The claim (\ref{it.mixed.small}) easily strengthens to a Lipschitz relation for $\delta_1, \delta_2$ near zero: specifically, if
\[
	\delta_2^*(\delta_1) = \sup\{\delta_2 \in \Bbb{R} \::\: \exp(\delta_2 Q_2)\exp(-\delta_1 Q_1) \in \mathcal{L}(L^2(\Bbb{R}^n))\},
\]
then
\[
	\delta_2^*(\delta_1) \asymp \delta_1, \quad \delta_1 \to 0^+,
\]
in the sense of the ratio being bounded from above and below by positive constants.  The lower bound is claim (\ref{it.mixed.small}). The upper bound follows from the same claim, which gives the existence of $C' > 0$ for which, when $\delta_2 > C'\delta_1$, the operator
\[
	\exp(\delta_1 Q_1)\exp(-\delta_2 Q_2) = \exp((\delta_1 - \frac{\delta_2}{C'}) Q_1 )\left(\exp(\frac{\delta_2}{C'}Q_1)\exp(-\delta_2 Q_2)\right)
\]
would give a compact inverse for $\exp(\delta_2 Q_2)\exp(-\delta_1 Q_1)$, which must therefore be unbounded.

We also observe that the small-time Lipschitz relation could also be analyzed via an FBI transform not specially adapted to the operators $Q_1$ and $Q_2$. As mentioned in Section \ref{subsec.paths.not.taken}, the small-time evolution is known to correspond on the FBI side to a change of weight where the weight $\Phi_t$ solves the Hamilton-Jacobi equation \eqref{eq.HJ.weight}, as discussed in \cite{Sj2010}, \cite{HiPS2009}, or \cite{HiPSVi2015b}. For any FBI transform $\mathcal{T}$ with quadratic phase of the type discussed here, expanding \eqref{eq.HJ.weight} to first order as $\delta_1, \delta_2 \in \Bbb{R}$ are small gives that
\[
	\mathcal{T}\exp(\delta_2Q_2)\exp(-\delta_1 Q_1)\mathcal{T}^*:H_\Phi \to H_{\Phi_{\delta_1, \delta_2}}
\]
is bounded, with
\begin{equation}\label{eq.HJ.mixed.h.o}
	\Phi_{\delta_1, \delta_2}(z) = \Phi(z) + \delta_1 \Re p_1(z, -2i\partial_z \Phi(z)) - \delta_2 \Re p_2(z, -2i\partial_z\Phi_z) + \BigO((\delta_1^2 + \delta_2^2)|z|^2).
\end{equation}
Here, $p_j(z, \zeta)$ are the FBI-side symbols of $Q_j$ obtained via composition with the canonical transformation corresponding to $\mathcal{T}$, and they are therefore positive definite along $\Lambda_\Phi = \{(z, -2i\partial_z \Phi(z))\}$. The relation \eqref{eq.mixed.h.o.result} follows, because for $C$ sufficiently large and $\delta >0$ sufficiently small we can guarantee that $\Phi_{\delta, \delta/C} \geq \Phi$.

We have a detailed proof of Theorem \ref{thm.mixed.h.o} below, including large-time behavior. Particularly in short times, however, the idea remains essentially the same, as may be seen by comparing \eqref{eq.HJ.mixed.h.o}, \eqref{eq.mixed.h.o.Phi2.expand}, and \eqref{eq.mixed.h.o.Phi1.expand}.
\end{remark}

\begin{remark*}
A consequence of claim (\ref{it.mixed.large}) is that, unless the ground states of $Q_1$ and $Q_2$ agree, we cannot take $\delta_1, \delta_2 \to \infty$ in \eqref{eq.mixed.h.o.set.compare}, because in fact
\[
	\exp(\tilde{\Delta}_0 Q_2)\exp(-\delta_1 Q_1) \notin \mathcal{L}(L^2(\Bbb{R}^n))
\]
and
\[
	\exp(-\delta_1 Q_1)L^2(\Bbb{R}^n) \not\subset \exp(-\tilde{\Delta}_0 Q_2)L^2(\Bbb{R}^n)
\]
for any $\delta_1 > 0$.

We also note that, when $\opnm{span}u_{0,1} = \opnm{span} u_{0,2}$, we demonstrate the exact characterization that \eqref{eq.mixed.h.o} holds if and only if
\begin{equation}\label{eq.mixed.h.o.match.exact}
	\|e^{\delta_2 \tilde{B}_2}e^{-\delta_1 \tilde{B}_1}\| \leq 1
\end{equation}
for certain positive definite Hermitian matrices $\tilde{B}_j$, $j=1,2$. Part (\ref{it.mixed.h.o.match}) is then an easy consequence.
\end{remark*}

\begin{proof}
The symbols $q_1$ and $q_2$ are elliptic, so by \cite{Sj1974} it is classical that they satisfy the hypotheses of Proposition \ref{prop.which.q}.  By \cite[Thm.~1.4]{Vi2013} their corresponding stable manifolds 
\[
	\Lambda^\pm(q_j) := \mathop{\bigoplus_{\lambda \in \opnm{Spec} F(q_j)}}_{\pm \Im \lambda > 0} \ker(F(q_j) - \lambda)^{2n},
\]
the same as in Proposition \ref{prop.which.q}, must be complex conjugates of one another, meaning that $\Lambda^+(q_j) = \overline{\Lambda^-(q_j)}.$ Therefore we appeal to the decomposition \eqref{eq.real.symbol.ip.form} and write henceforth 
\begin{equation}\label{eq.def.qj}
	q_j(x,\xi) = B_j(\xi - \overline{A_j}x)\cdot (\xi - A_j x), \quad j=1,2,
\end{equation}
for matrices $A_j, B_j \in \Bbb{M}_{n\times n}(\Bbb{C})$ with $A_j^\top = A_j$ and $\Im A_j > 0$ in the sense of positive definite matrices.  Since $q_j$ are real-valued and positive definite, we may take $B_j$ self-adjoint and positive definite.  We also recall from the proof of \cite[Thm.~3.5]{Sj1974} that the ground states of $Q_j$ are determined by the matrices $A_j$: there exist constants $a_j \in \Bbb{C}\backslash \{0\}$ such that
\begin{equation}\label{eq.mixed.Gaussians}
	u_{0,j}(x) = a_je^{\frac{i}{2}A_jx\cdot x}.
\end{equation}

In order to establish \eqref{eq.mixed.h.o.right} and \eqref{eq.mixed.h.o.left}, we introduce $Q_3 = q_3^w(x,D_x)$ for
\begin{equation}\label{eq.def.q3}
	q_3(x,\xi) = B_3(\xi - \overline{A_2}x)\cdot (\xi - A_1 x),
\end{equation}
where the matrix $B_3$ is to be determined.

Following the proof of Proposition \ref{prop.which.q}, there exists a strictly convex weight $\Phi_2$, a transformation $\mathcal{T}_2 : L^2(\Bbb{R}^n) \to H_{\Phi_2}$, and a choice of the matrix $B_3$ such that
\[
	\mathcal{T}_2 Q_3 \mathcal{T}_2^* = z\cdot \partial_z.
\]
The fact that the canonical transformation associated with $\mathcal{T}_2$ takes $\{(x, \overline{A_2} x)\}$ to $\{(0, \zeta)\}$ implies that, for some matrix $\tilde{B}_2$ and writing $h_2(z) = \frac{1}{2}z\cdot (\Phi_2)''_{zz}z$,
\[
	\mathcal{T}_2 Q_2 \mathcal{T}_2^* = \tilde{B}_2 z\cdot (\partial_z + h_2'(z)) + \mu_{0,2}.
\]
The eigenvalue $\mu_{0,2}$ appears because we can identify the ground state of $\mathcal{T}_2 Q_2 \mathcal{T}_2^*$ via 
\[
	\tilde{B}_2z\cdot(\partial_z + h_2'(z))e^{-h_2(z)} = 0.
\]
From the definition of the Weyl quantization, we can deduce that $\opnm{Tr}\tilde{B}_2 = 2\mu_{0, 2}$, but this can also be deduced from invariance of the spectrum of the fundamental matrix when $q_2$ is composed with a canonical transformation.

We remark similarly that, identifying the ground state $u_{0,1}(x) = a_1e^{\frac{i}{2}A_1x\cdot x}$ with the kernel of $D_x - A_1x$, we see that $\mathcal{T}_2 u_{0,1}$ lies in the kernel of $D_z$ and is therefore constant.

By modifying Theorem \ref{thm.boundedness.delta} to account for the matrix $\tilde{B}_2$, \eqref{eq.mixed.h.o.left} holds if and only if
\[
	\Phi^{(\delta_2), \tilde{B}_2}_2 (e^t z) \geq \Phi_2(z), \quad \forall z \in \Bbb{C}^n.
\]
Using the expression \eqref{eq.def.Phi.tau.B}, with $G_2 = ((\Phi_2)''_{\bar{z}z})^{1/2}$ and with $\delta, t \in \Bbb{R}$ and small, we obtain the following analogue of \eqref{eq.Theta.degree.1.expand}:
\begin{equation}\label{eq.mixed.h.o.Phi2.expand}
	\begin{aligned}
	\Phi_2^{(\delta_2), \tilde{B}_2}(e^t z) - \Phi_2(z) &= \frac{e^{2t}}{2}(|G_2e^{-\delta_2 \tilde{B}_2}z|^2 - |G_2z|^2) + (e^{2t}-1)\Phi_2(z)
	\\ &= 2t\Phi_2(z) - \delta_2 \Re\langle G_2z, \tilde{B}_2 G_2z\rangle + \BigO((\delta_2^2 + t^2)|z|^2).
	\end{aligned}
\end{equation}
Strict convexity of $\Phi_2$ means that we can ensure that \eqref{eq.mixed.h.o.left} holds for $\delta_2 = t/C$ for $0 \leq t \leq t_0$ sufficiently small.

Furthermore, as in Lemma \ref{lem.Delta.0}, let
\[
	\tilde{\Delta}_0 = \sup\{\delta \in \Bbb{R} \::\: \forall z \in \Bbb{C}^n,~~\Phi_2^{(\delta), \tilde{B}_2} \geq 0\}.
\]
Since $\mathcal{T}_2 u_{0,2}(z) = c_2 e^{-h_2(z)}$ and because $G_2\tilde{B}_2 G_2^{-1}$ is positive definite Hermitian following Proposition \ref{prop.h.o.semigroup}, we can easily check that $\tilde{\Delta}_0 = \infty$ if and only if $h_2(z) = 0$ if and only if $\opnm{span} u_{0,1} = \opnm{span} u_{0,2}$. In this special case, we have that $\mathcal{T}_2 Q_1\mathcal{T}_2 = \tilde{B_1}z\cdot \partial_z$ and we are free to take $\Phi_2(z) = \frac{1}{2}|z|^2$ since $\Phi_2$ has no pluriharmonic part; part (\ref{it.mixed.h.o.match}) of the theorem as well as \eqref{eq.mixed.h.o.match.exact} follow immediately.

Following Proposition \ref{prop.maximal.smoothing}, we see that there exists a $t > 0$ such that \eqref{eq.mixed.h.o.left} holds if and only if $\delta_2 < \tilde{\Delta}_0$. Recalling that $\mathcal{T}_2 u_{0,1}$ is constant, if $\tilde{\Delta}_0 \neq \infty$, then
\[
	\exp(\tilde{\Delta}_0 Q_2)u_{0,1} \notin L^2(\Bbb{R}^n).
\]

We turn to \eqref{eq.mixed.h.o.right}. Since
\[
	Q_3^* = B_3^*(D_x - \overline{A_1}x)\cdot (D_x - A_2 x)
\]
and since $Q_2$ is self-adjoint, we can reverse the process, finding a weight $\Phi_1$, a transformation $\tilde{T}_1:L^2(\Bbb{R}^n) \to H_{\Phi_1}$, and matrices $\tilde{B}_1, \tilde{B}_3$ such that, writing $h_1(z) = \frac{1}{2}z\cdot (\Phi_1)''_{zz}z$,
\[
	\mathcal{T}_1 Q_3^* \mathcal{T}_1^* = \tilde{B}_3 z\cdot \partial_z
\]
and
\[
	\mathcal{T}_1 Q_1 \mathcal{T}_1^* = \tilde{B}_1 z\cdot (\partial_z + h_1'(z)) + \mu_{0,1}.
\]
We do not seek to write a formula for the matrix $\tilde{B}_3$, but we remark that the symbol $\Re (\tilde{B}_3 z \cdot (\Phi_1)'_z(z))$ is elliptic in the sense of \eqref{eq.ellipticity}. This follows from the exact Egorov theorem and the observation that, on the space $H_{\Phi_2}$, the symbol of $Q_3$ is $\Re (z\cdot (\Phi_2)'_{z}(z)) = \Phi_2(z)$ which is strictly convex.

A similar computation to \eqref{eq.mixed.h.o.Phi2.expand} or \eqref{eq.Theta.degree.1.expand}, this time with $G_1 = ((\Phi_1)''_{\bar{z}z})^{1/2}$, gives that
\begin{multline}\label{eq.mixed.h.o.Phi1.expand}
	\Phi_1^{(-\delta_1), \tilde{B}_1}(e^{-t\tilde{B}_3} z) - \Phi_1(z) 
	\\ = \frac{1}{2}(|G_1 e^{\delta_1\tilde{B}_1}e^{-t\tilde{B}_3} z|^2 - |G_1 e^{-t\tilde{B}_3}z|^2) + \Phi(e^{-t\tilde{B}_3}z)-\Phi(z)
	\\ = \delta_1\Re \langle G_1 \tilde{B}_1 z, G_1 z\rangle - 2t \Re((\tilde{B}_3 z)\cdot (\Phi_1)'_z(z)) + \BigO((\delta_1^2 + t^2)|z|^2).
\end{multline}
Since $\tilde{B}_1$ corresponds to a (positive definite) harmonic oscillator, we have following Proposition \ref{prop.h.o.semigroup} that $G_1 \tilde{B}_1 G_1^{-1}$ is positive definite Hermitian. Therefore
\[
	\exp(-\delta_1 Q_1)\exp(tQ_3^*) \in \mathcal{L}(L^2(\Bbb{R}^n)).
\]
either taking $t = \delta_1/C$ for $C$ sufficiently large and $\delta_1$ sufficiently small, to establish (\ref{it.mixed.small}), or for $\delta_1$ sufficiently large for any $t$, to establish (\ref{it.mixed.large}).

Having already established \eqref{eq.mixed.h.o.left} for $\delta_2 = t/C$ and $t$ sufficiently small or for $t$ sufficiently large for any $\delta_2 < \tilde{\Delta}_0$, all that remains to prove the theorem is to justify the adjoint relation in \eqref{eq.mixed.h.o.right}. This follows by finding dense subsets of $L^2(\Bbb{R}^n)$ for which 
\begin{equation}\label{eq.mixed.adjoints}
	\langle \exp(-\delta_1 Q_1)\exp(t Q_3^*)u, v\rangle = \langle u, \exp(t Q_3)\exp(-\delta_1 Q_1)v\rangle.
\end{equation}
Since $\exp(-\delta_1 Q_1)$ is self-adjoint, it suffices to show that
\begin{equation}\label{eq.mixed.adjoint.check}
	\langle \exp(t Q_3^*)u, \exp(-\delta_1 Q_1) v\rangle = \langle u, \exp(tQ_3)\exp(-\delta_1 Q_1)v\rangle.
\end{equation}
Using the supersymmetric decompositions \eqref{eq.def.qj} and \eqref{eq.def.q3}, let $u$ be in in the span of the generalized eigenfunctions of $Q_3^*$ and let $v$ be in the span of the generalized eigenfunctions of $Q_3$ or of $Q_1$, since these are the same set. Concretely, this is equivalent to assuming that $u/u_{0,2}$ and $v/u_{0,1}$ are polynomials. These sets of $u$ and $v$ are dense and the actions of the semigroups above leave invariant the degree of the polynomial coefficient, so the relation \eqref{eq.mixed.adjoint.check} becomes easy to check. This completes the proof of the theorem.
\end{proof}

\section{Return to equilibrium and regularization for times long and short}\label{sec.return}

The question of return to equilibrium generally concerns the operator
\[
	e^{-tP}(1-\Pi_0),
\]
where $\Pi_0$ is the spectral projection associated with the eigenvalue $0 \in \Bbb{C}$; see for instance \cite[Ch.~6]{HeNiBook} or \cite{VillaniBook}. The operators $P$ given by \eqref{eq.def.P} are associated with natural projections 
\begin{equation}\label{eq.def.Pi}
	\Pi_N u(z) = \sum_{|\alpha| \leq N} \frac{\partial^\alpha u(0)}{\alpha!}z^\alpha:H_\Phi \to H_\Phi.
\end{equation}
It is clear from Theorem \ref{thm.eigen.core} that the image of $\Pi_0$, which is the set of constant functions, is the span of an eigenfunction of $P$ with eigenvalue zero; under a hypothesis such as that the spectrum of $M$ is strictly contained in a half-plane, this eigenfunction with eigenvalue zero is unique up to scalar multiples.

In general, up to some questions of multiplicity of eigenvalues --- and possible non-existence of the resolvent --- the $\Pi_N$ are sums of spectral projections of $P$; see \cite[Thm.~1.2]{Vi2013}. The images of the complements of these projections are the high-energy spaces
\begin{equation}\label{eq.def.MN}
	\begin{aligned}
	\mathcal{M}_{N+1} &= (1-\Pi_{N})H_\Phi
	\\ & = \{u \in H_\Phi \::\: \partial^\alpha u(0) = 0,\ \forall |\alpha| \leq N\}.
	\end{aligned}
\end{equation}
Naturally, we identify $\mathcal{M}_0$ with the space $H_\Phi$ itself. Where the weight needs to be emphasized, we will write $\mathcal{M}_N^\Phi$.

Section \ref{subsec.return} concerns sharp estimates for return to equilibrium for long times.  Roughly, as $|\tau| \to \infty$, the return to equilibrium is governed by $\|Ge^{\tau M}G^{-1}\|$, which following Lemma \ref{lem.exptM.ej} is largely determined by the spectral properties of $M$.  Next, in Section \ref{subsec.subellipticity}, we discuss short time estimates for the regularization $\exp(-tP)$ for $t > 0$ in terms of $\delta_0(-t)$, extending \eqref{eq.delta.0.deriv} in a natural way which turns out to be equivalent to a classical bracket condition.  Finally, in Section \ref{subsec.orthogonal}, we see that in an important special case considered more closely in \cite{AlVi2014a}, estimates for $\delta_0(-t)$ and estimates for return to equilibrium are identical.

\subsection{Return to equilibrium for long times}\label{subsec.return}

To discuss the long-time behavior of $\exp(-tP)$ on the spaces $\mathcal{M}_N$, we begin by using the unitary transformation $\mathcal{U}$ in \eqref{eq.def.U} to reduce to a study on $H_\Psi$ for $\Psi$ as in \eqref{eq.def.Psi}.  We compute that, for $\Phi$ as in \eqref{eq.decompose.Phi} and $\exp(\tau P)$ and $\exp(\delta P_0)$ as in Theorem \ref{thm.boundedness.delta},
\begin{equation}\label{eq.exptP.conjugated}
	\mathcal{U}^* \exp(\delta P_0) \exp(\tau P) \mathcal{U} u(z) = u(Ge^\delta e^{\tau M}G^{-1}z)e^{-h(e^\delta e^{\tau M}G^{-1}z) + h(e^\delta G^{-1}z)}.
\end{equation}
Motivated by the form of this operator, we turn to the following lemma.

\begin{lemma}\label{lem.return}
Fix $c_1 \in [0,1)$. Let $h_1, h_2$ be holomorphic quadratic forms on $\Bbb{C}^n$ and $B \in \Bbb{M}_{n\times n}(\Bbb{C})$ a matrix such that, for all $z \in \Bbb{C}^n$,
\begin{equation}\label{eq.return.lem.hyp}
	|Bz|^2 + 2\Re(h_1(Bz) + h_2(z)) \leq c_1 |z|^2.
\end{equation}
Then the operator
\[
	Su(z) = u(B z)e^{h_1(Bz)+h_2(z)}, \quad u\in H_\Psi,
\]
is bounded as an operator on $H_\Psi$ with $\Psi = \frac{1}{2}|z|^2$ as in \eqref{eq.def.Psi}.  Furthermore, for all $N \in \Bbb{N}$, there exists some $C = C(N,c_1) > 0$ such that
\begin{equation}\label{eq.lem.return.1}
	\|S u\|_\Psi \leq C\|B\|^N \|u\|_\Psi, \quad \forall u \in \mathcal{M}_N^\Psi.
\end{equation}
\end{lemma}

\begin{proof}
Let $\{f_\alpha\}_{\alpha \in \Bbb{N}^n}$ be the usual orthonormal basis for $H_\Psi$ defined in \eqref{eq.def.f.alpha}; it is easy to see that $\{f_\alpha\}_{|\alpha| \geq N}$ is an orthonormal basis for $\mathcal{M}_N$. (Throughout the proof, we take $\mathcal{M}_N = \mathcal{M}^\Psi_N$.)  We begin with a pointwise estimate for $|Su|$ when $u \in \mathcal{M}_N$.  Write 
\[
	u = \sum_{|\alpha| \geq N} \langle u, f_\alpha\rangle f_\alpha,
\]
and apply the Cauchy-Schwarz inequality to obtain
\[
	|Su(z)|^2 \leq \|u\|_{\Psi}^2 e^{2\Re(h_1(Bz)+h_2(z))}\sum_{|\alpha| \geq N} |f_\alpha(B z)|^2.
\]

We may check from the definition \eqref{eq.def.f.alpha} that $|f_{\alpha + \beta}| \leq \pi^{n/2}|f_\alpha|\,|f_\beta|$ and that $\sum_{\beta\in \Bbb{N}^n} |f_\beta|^2 = \pi^{-n}e^{2\Psi(z)}$. Furthermore, $|(Bz)_j| \leq |Bz| \leq \|B\|\,|z|$, so we compute that
\[
	\begin{aligned}
	\sum_{|\alpha| \geq N}|f_\alpha(Bz)|^2 &\leq \pi^n \sum_{|\alpha|=N} |f_\alpha(Bz)|^2\sum_{\beta\in \Bbb{N}^n} |f_\beta(Bz)|^2
	\\ &= e^{2\Psi(B z)}\sum_{|\alpha|=N} |f_\alpha(Bz)|^2
	\\ &\leq K_N e^{2 \Psi(B z)}\|B\|^{2N}|z|^{2N},
	\end{aligned}
\]
for some positive constant $K_N$ and all $z \in \Bbb{C}^n$.

Thus, for any $u \in \mathcal{M}_N$,
\begin{equation}\label{eq.S.pointwise}
	|S u(z)|^2 \leq 2^N K_N \|u\|_\Psi^2 e^{|Bz|^2 + 2\Re(h_1(Bz) + h_2(z))}\|B\|^{2N}\Psi(z)^N,
\end{equation}
from which we have the estimate
\[
	\frac{\|Su\|_{\Psi}^2}{\|u\|_{\Psi}^2} \leq \|B\|^{2N} 2^N K_N \|u\|_\Psi^2 \int_{\Bbb{C}^n} |z|^{2N} e^{|Bz|^2 + 2\Re(h_1(Bz) + h_2(z))-|z|^2}\,dL(z),
\]
also for all $u \in \mathcal{M}_N$. Therefore, so long as \eqref{eq.return.lem.hyp} holds, then \eqref{eq.lem.return.1} holds with
\[
	C^2 = 2^N K_N \int_{\Bbb{C}^n} |z|^{2N}e^{-(1-c_1)|z|^2}\,dL(z).
\]
The claim that $S$ is bounded is just the special case $N = 0$ of \eqref{eq.lem.return.1}.
\end{proof}

In view of \eqref{eq.exptP.conjugated}, we would like to apply Lemma \ref{lem.return} with the change of variables matrix
\[
	B = Ge^\delta e^{\tau M} G^{-1}
\]
and the harmonic functions
\[
	h_1(z) = -h(G^{-1}z), \quad h_2(z) = h(e^\delta G^{-1}z).
\]
The condition \ref{eq.return.lem.hyp} then becomes
\[
	|G e^\delta e^{\tau M} G^{-1} z|^2 - 2\Re h(e^\delta e^{\tau M} G^{-1}z) \leq c_1 |z|^2 - 2\Re h(e^\delta G^{-1}z).
\]
Making the change of variables $y = e^\delta e^{\tau M}G^{-1}z$, this is equivalent to
\[
	|Gy|^2 - 2\Re h(y) \leq c_1e^{-2\delta} |Ge^{-\tau M}y|^2 - 2\Re h(e^{-\tau M}y).
\]
We then note that the left-hand side is $2\Phi(y)$ and the right-hand side is 
\[
	2\Phi^{(\delta-(\log c_1)/2)}(e^{-\tau M}y).
\]
In conclusion, using the definition \eqref{eq.def.delta.0} of $\delta_0(\tau)$, the condition \eqref{eq.return.lem.hyp} applied to \eqref{eq.exptP.conjugated} is equivalent to
\[
	\delta_0(\tau) \geq \delta - \frac{1}{2}\log c_1.
\]

We arrive at the following theorem.

\begin{theorem}\label{thm.return.by.PiN}
	Let the matrix $M$, the weight $\Phi$, and the operators $P$ and $\exp(\tau P)$ be as in Proposition \ref{prop.solve.exptP}. Also recall the definitions \eqref{eq.def.P0} of $P_0$, \eqref{eq.def.delta.0} of $\delta_0(\tau)$, and \eqref{eq.def.Pi} of the projection $\Pi_N$. Fix any $c_0 > 0$ and $N \in \Bbb{N}$.  Then there exists some $C = C(c_0, N, \Phi) > 0$ for which, whenever
	\[
		\delta \leq \delta_0(\tau) - c_0,
	\]
	we have
	\begin{equation}\label{eq.return.by.PiN}
		\|\exp(\delta P_0)\exp(\tau P)(1-\Pi_N)\|_{\mathcal{L}(H_\Phi)} \leq C\|Ge^\delta e^{\tau M}G^{-1}\|^{N+1}.
	\end{equation}
\end{theorem}

\begin{proof}
As discussed, the hypothesis $\delta \leq \delta_0(\tau) - c_0$ allows us to apply Lemma \ref{lem.return} which gives that, for some $C_0 > 0$,
\[
	\|\mathcal{U}\exp(\delta P_0)\exp(\tau P)\mathcal{U}^* u\|_\Psi \leq C_0\|Ge^\delta e^{\tau M}G^{-1}\|^{N+1}\|u\|_\Psi, \quad \forall u \in \mathcal{M}^\Psi_{N+1}.
\]
Since $\mathcal{U}$ consists of multiplication by a holomorphic function and a change of variables, $\partial^\alpha u(0) = 0$ for all $|\alpha|\leq N$ if and only if $(\partial^\alpha \mathcal{U}u)(0) = 0$ for all $|\alpha|\leq N$.  That is, by \eqref{eq.def.MN}, 
\[
	\mathcal{U} (\mathcal{M}^\Psi_{N+1}) = \mathcal{M}^\Phi_{N+1}.
\]

Combining this with the triangle inequality, for any $u \in H_\Phi$,
\[
	\|\exp(\delta P_0)\exp(\tau P)(1-\Pi_N)u\|_\Phi = \leq C_0(1+\|\Pi_N\|)\|Ge^\delta e^{\tau M}G^{-1}\|^{N+1}\|u\|_\Phi.
\]
From \cite[Prop.~3.3]{HiSjVi2013} we have that $\|\Pi_N\|$ is bounded (with norm growing at most exponentially quickly in $N$), so the theorem follows.
\end{proof}

To complete this analysis, we describe the action of $\exp(\tau P)$ on the generalized eigenfunctions of $P$ for large $|\tau|$.  Since, by Theorem \ref{thm.eigen.core}, the functions $\{(Gz)^\alpha\}_{|\alpha|\leq N}$ for some invertible matrix $\tilde{G}$ are generalized eigenfunctions of $P$ and span the space $\Pi_N(H_\Phi)$, which is simply the space of polynomials in $n$ variables of degree $N$ or less, this suffices to describe $\exp(\tau P)\Pi_N$.

\begin{proposition}\label{prop.return.eigenfunctions}
	Let $\Phi$, $P$, and $\exp(\tau P)$ be as in Proposition \ref{prop.solve.exptP}, and let $\tilde{G}$ be such that $\tilde{G}^{-1}M\tilde{G}$ is in Jordan normal form. Let $\lambda_\alpha$, $r_\alpha$, $C_\alpha$, and $\tilde{\alpha}$ be as in Theorem \ref{thm.eigen.core} and its proof.  Then
	\[
		\exp(\tau P)(\tilde{G}z)^\alpha = \frac{C_\alpha}{(r_\alpha-1)!}e^{\tau \lambda_\alpha} \tau^{r_\alpha - 1}\left((\tilde{G}z)^{\tilde{\alpha}} + \BigO(|\tau|^{-1})\right)
	\]
	as $|\tau| \to \infty$.
\end{proposition}

\begin{proof}
	After conjugating by $\mathcal{V}_{\tilde{G}}$ as in the proof of Theorem \ref{thm.eigen.core}, the proposition is automatic from \eqref{eq.gen.eigen.vanishing} and Lemma \ref{lem.exptM.ej}.
\end{proof}

At this point, we have a complete description of the behavior of $\exp(\tau P)$ as $|\tau| \to \infty$ in such a way that $\|e^{\tau M}\| \to 0$; see Proposition \ref{prop.exptM.asymptotic} and the remark following for a discussion of this asymptotic regime.  To illustrate this, we consider the leading-order behavior for return to equilibrium of any order as $\tau = -t \to -\infty$.

For the purposes of notation, let
\[
	\rho(M) = \min_{\lambda \in \opnm{Spec}M} \Re \lambda
\]
be the spectral abscissa of $M$, let
\[
	r(\lambda, M) = \max\{r\in \Bbb{N} \::\: \ker(M-\lambda)^r \backslash \ker(M-\lambda)^{r-1} \neq \varnothing\}
\]
be the maximum size of a Jordan block associated with the eigenvalue $\lambda$, and let
\[
	R(M) = \max\{r(\lambda, M) \::\: \lambda \in \opnm{Spec}M, ~~ \Re \lambda = \rho(M)\}
\]
be the maximum size of a Jordan block associated with an eigenvalue with real part $\rho(M)$.

Finally, we define the natural decay factor
\[
	A(t) = t^{R(M)-1}e^{-t\rho(M)}.
\]
As a consequence of Lemma \ref{lem.exptM.ej} and the triangle inequality, if $\opnm{Spec} M \subset \{\Re \lambda > 0\}$, then, for some $C, T > 0$,
\begin{equation}\label{eq.exptM.A}
	\frac{1}{C} A(t) \leq \|e^{-tM}\| \leq C A(t), \quad \forall t \geq T.
\end{equation}
We see that this elementary asymptotic behavior for $\|e^{-tM}\|$ is repeated in return to equilibrium of every order for $\|\exp(-tP)\|$.

\begin{proposition}\label{prop.return.convergence}
Let the matrix $M$, the weight $\Phi$, and the operators $P$ and $\exp(-tP)$ be as in Proposition \ref{prop.solve.exptP}; assume furthermore that $\opnm{Spec}M \subset \{\Re \lambda > 0\}$.  Fix $N \in \Bbb{N}$ and recall the definition of $\Pi_N$ from \eqref{eq.def.Pi}.  Finally, let $A(t)$ be as above.

Then there exists $T_0, C_0 > 0$ sufficiently large such that, for all $t > T_0$,
\begin{equation}\label{eq.return.convergence.bounds}
	\frac{1}{C_0}A(t)^{N+1} \leq \|\exp(-tP)(1-\Pi_N)\|_{\mathcal{L}(H_\Phi)} \leq C_0A(t)^{N+1}.
\end{equation}
Furthermore, for there to exist $a \in \Bbb{R}$ such that
\[
	e^{iat}A(t)^{-N-1}\exp(-tP)(1-\Pi_N)
\]
converges in the weak operator topology as $t \to \infty$, it is necessary and sufficient that there is only one $\lambda \in \opnm{Spec} M$ for which $\Re \lambda = \rho(M)$ and $r(\lambda, M) = R(M)$; in this case, the convergence is in the operator norm topology.
\end{proposition}

\begin{proof}
As in the proof of Theorem \ref{thm.eigen.core}, we reduce to the case where $M$ is in Jordan normal form after a change of variables. Therefore let $\tilde{G}$ be such that $\tilde{G}^{-1}M\tilde{G}$ is in Jordan normal form, and for $\mathcal{V}_{\tilde{G}}$ from \eqref{eq.def.V}, let
\[
	\tilde{P} = \mathcal{V}_{\tilde{G}}^* P \mathcal{V}_{\tilde{G}} = \tilde{G}^{-1}M\tilde{G}z \cdot \partial_z.
\]
Note that $[\mathcal{V}_{\tilde{G}}, \Pi_N] = 0$, so the claims about $\exp(-tP)(1-\Pi_N)$ may be proven by studying $\exp(-t\tilde{P})(1-\Pi_N)$ instead. Note also that
\[
	\Pi_{N+1}u(z) - \Pi_N u(z) = \sum_{|\alpha| = N+1} \frac{\partial^\alpha u(0)}{\alpha!} z^\alpha.
\]

By Theorem \ref{thm.return.by.PiN}, the observation \eqref{eq.exptM.A}, and Proposition \ref{prop.return.eigenfunctions}, for $t$ sufficiently large,
\[
	\begin{aligned}
	\exp(-t\tilde{P})&(1-\Pi_N)u(z)  
	\\ & =\sum_{|\alpha| = N+1} \exp(-t\tilde{P}) \frac{\partial^\alpha u(0)}{\alpha !}z^\alpha + \exp(-t\tilde{P})(1-\Pi_{N+1})u(z)
	\\ & = \sum_{|\alpha| = N+1} \frac{\partial^\alpha u(0)}{\alpha !}\frac{C_\alpha}{(r_\alpha-1)!}e^{-t \lambda_\alpha} (-t)^{r_\alpha - 1}\left(z^{\tilde{\alpha}} + \BigO(t^{-1})\right) 
	\\ & \qquad \qquad + \BigO(A(t)^{N+2}\|u\|).
	\end{aligned}
\]
From Theorem \ref{thm.eigen.core} and recalling the definition \eqref{eq.def.rj.tilde}, it is clear that the $t$-dependent factor $|e^{-t \lambda_\alpha} (-t)^{r_\alpha - 1}|$ for $|\alpha| = N$ is maximized, as $t \to \infty$, when $\alpha$ is supported only on those indices corresponding to eigenvalues with real part $\rho(M)$ and with $\tilde{r}_j = R(M)-1$.  

Introducing the notation
\begin{equation}\label{eq.def.return.multiindices}
	S_{N+1} = \{\alpha \in \Bbb{N}^n \::\: |\alpha| = N+1,~~\alpha_j \neq 0 \implies (\Re \lambda_j = \rho(M) ~~\& ~~ \tilde{r}_j = R(M)-1)\},
\end{equation}
we see that $\alpha \in S_{N+1}$ if and only if $\Re \lambda_\alpha = (N+1)\rho(M)$ and 
\[
	r_\alpha = (N+1)(R(M)-1)+1.
\]
We see that, when $\alpha \in S_{N+1}$, for $t\geq 1$ and as $t \to \infty$,
\begin{equation}\label{eq.return.maximal.eigenfunction}
	\|\exp(-t\tilde{P})z^\alpha\| = \frac{C_\alpha}{((N+1)(R(M)-1))!} A(t)^{N+1}\|z^{\tilde{\alpha}}\|(1+\BigO(t^{-1}))
\end{equation}
and, when $\alpha \notin S_{N+1}$ but $|\alpha| = N+1$, then
\[
	\|\exp(-t\tilde{P})z^\alpha\| = \BigO(t^{-1}A(t)^{N+1}).
\]

Therefore, for $t$ sufficiently large,
\begin{equation}\label{eq.return.convergence.final}
	\begin{aligned}
	\exp(-tP)&(1-\Pi_N)u(z) 
	\\ &= \sum_{\alpha \in S_{N+1}} \frac{C_\alpha}{((N+1)(R(M)-1))!}A(t)^{N+1}e^{-it \Im \lambda_\alpha}\frac{\partial^\alpha u(0)}{\alpha !}z^{\tilde{\alpha}} 
	\\ &\qquad \qquad + \BigO(t^{-1}A(t)^{N+1}\|u\|)
	\end{aligned}
\end{equation}
Then \eqref{eq.return.convergence.bounds} follows from \eqref{eq.return.maximal.eigenfunction} and the triangle inequality.

The claim about weak convergence comes from the observation that, if $\exp(-t\tilde{P})$ converges weakly, then
\[
	\#\{\Im \lambda_\alpha \::\: \alpha \in S_{N+1}\} = 1;
\]
otherwise, $e^{iat}A(t)^{-N-1}$ contains oscillating factors.  Since this set can be expressed as the collection of sums (allowing repetition) of $N+1$ imaginary parts $\Im \lambda_j$ where $\Re \lambda_j = \rho(M)$ and $\tilde{r}_j = R(M)-1$, this collection consists of one value if and only if
\[
	\#\{\Im \lambda_j \::\: \Re \lambda_j = \rho(M) ~~\&~~ \tilde{r}_j = R(M)-1\} = 1.
\]
Note that, even if this is true, the eigenvalue $\rho(M) + i\Im \lambda_j$ could correspond to many Jordan blocks of the same size.

In the case that there is only one such $a = \Im \lambda_\alpha$, it is clear from \eqref{eq.return.convergence.final} that 
\begin{multline*}
	e^{iat}A(t)^{-N-1}\exp(-t\tilde{P})(1-\Pi_N)u(z) 
	\\ = \sum_{\alpha \in S_{N+1}} \frac{C_\alpha}{((N+1)(R(M)-1))!} \frac{\partial^\alpha u(0)}{\alpha !} z^{\tilde{\alpha}} + \BigO(t^{-1}\|u\|),
\end{multline*}
proving convergence in operator norm. Again, these statements for $\tilde{P} = \mathcal{V}_{\tilde{G}}^* P \mathcal{V}_{\tilde{G}}$ lead immediately to the corresponding statements for $P$, and so the proposition is proven.
\end{proof}

\begin{remark}
The projections
\[
	u(z) \mapsto \frac{\partial^\alpha u(0)}{\alpha !} z^\alpha
\]
can be seen to be bounded on $H_\Phi$ for the same reasons that each $\Pi_N$ from \eqref{eq.def.Pi} is a bounded projection.  A more detailed analysis is carried out in \cite{Vi2013}; while we recall that the norms of these rank-one projections must be bounded by $Ce^{C|\alpha|}$ for some $C > 0$ depending on $\Phi$, we do not pursue this question here. 
\end{remark}

\begin{remark}\label{rem.exact.norm}
When $\opnm{Spec}M \subset \{\Re \lambda > 0\}$, we have from the case $N = 0$ of Proposition \ref{prop.return.convergence} that, for $t$ sufficiently large,
\[
	\exp(-tP)u(z) = u(0) + \BigO(t^{R(M)-1}e^{-t\rho(M)}\|u\|),
\]
with error a function in $H_\Phi$; furthermore, the error estimate is sharp.  There is therefore a large gap between \eqref{eq.norm.bound}, which has exponential growth as an upper bound as $\tau = -t \to -\infty$, and the true behavior which is bounded with an exponentially small error.

This gap is explained under the hypotheses of Section \ref{subsec.orthogonal} where the value of the norm of $\|\exp(-tP)\|$ is known exactly, but the question remains open in the general case.  Because $u(z) \mapsto u(0)$ is an orthogonal projection only when the harmonic part $\Re h(z)$ of $\Phi$ from \eqref{eq.decompose.Phi} vanishes (see Proposition \ref{prop.h.vanishes}), we remark that if $\opnm{Spec} M \subset \{\Re \lambda > 0\}$, then
\[
	\lim_{t \to \infty}\|\exp(-tP)\| = 1 \iff h = 0.
\]
\end{remark}

\subsection{Weak ellipticity and small-time regularization}\label{subsec.subellipticity}

It is already apparent from Theorem \ref{thm.boundedness.0} that the rate of change of $\Phi(e^{tM}z) - \Phi(z)$ as a function of $t$ plays an important role in behavior of the solution operator for small times.  We begin by identifying that rate under the (non-strict) ellipticity hypothesis \eqref{eq.ellipticity}.

\begin{theorem}\label{thm.subelliptic.escape}
Let $\Phi$ satisfy \eqref{eq.Phi.assumptions}. Assume that the non-strict ellipticity condition \eqref{eq.ellipticity} holds for a matrix $M$ and fix $z \in \Bbb{C}^n$. Using the notation \eqref{eq.def.Theta}, let
\begin{equation}\label{eq.def.subellipticity}
	I = I(z) = \min\left \{ k \geq 0 \::\: \Theta(M^k z) \neq 0\right\}.
\end{equation}
Then either $I \leq 2n-2$ and, as $t \to 0$,
\begin{equation}\label{eq.Subelliptic.Escape}
	\Phi(e^{tM}z) - \Phi(z) = \frac{1}{(2I+1)!}\binom{2I}{I}\Theta(M^I z)t^{2I+1} + \BigO(t^{2I+2})|z|^2,
\end{equation}
or $I = \infty$ and 
\begin{equation}\label{eq.Subelliptic.No.Escape}
	\Phi(e^{tM}z) = \Phi(z), \quad \forall t \in \Bbb{R}.
\end{equation}
\end{theorem}

\begin{remark*}
When $I(z) = \infty$ for some $z \in \Bbb{C}^n \backslash \{0\}$, we conclude that $\exp(-tP)$ is never compact for any $t \in \Bbb{R}$ by Theorem \ref{thm.boundedness.0}.
\end{remark*}

\begin{proof}
Regarding $\Theta$ as a quadratic form in $2n$ real variables, we have that $\Theta$ is positive semidefinite by our assumption \eqref{eq.ellipticity} and therefore its zero set coincides with the kernel of its Hessian matrix.  This is a linear condition, so by the Cayley-Hamilton theorem we have that if $I \geq 2n-1$ then $I = \infty$ for $I$ in \eqref{eq.def.subellipticity}.

To analyze derivatives of $\Phi(e^{tM}z)$, particularly of higher order, it is convenient to associate $\Phi:\Bbb{C}^n \to \Bbb{R}$ with a natural real-valued real-bilinear form acting on $\Bbb{R}^{2n}$. That is, let
	\begin{equation}\label{eq.def.Phi.symmetric}
		\Phi(z,\zeta) = \Re(z\cdot\Phi'_z(\zeta))
	\end{equation}
	denote the unique symmetric real-bilinear form on $\Bbb{C}^{2n}$ such that $\Phi(z,z) = \Phi(z)$. Then we compute that
	\begin{equation}\label{eq.all.derivs}
		\frac{d^k}{dt^k} \Phi(e^{tM}z) = \sum_{j=0}^k \binom{k}{j}\Phi(M^je^{tM}z, M^{k-j}e^{tM}z).
	\end{equation}
It is also useful to similarly extend $\Theta$, which is here a positive semi-definite real-valued real-quadratic form thanks to \eqref{eq.ellipticity}, to a real-valued real-bilinear form.  We note that $\Theta(z) = 2\Phi(Mz, z)$ and therefore we may express the extension of $\Theta$ in terms of that of $\Phi$:
\begin{equation}\label{eq.Theta.symmetric}
	\Theta(z,\zeta) = \Phi(Mz,\zeta) + \Phi(z,M\zeta).
\end{equation}

To establish the theorem, we show that, for $k \in \Bbb{N}$,
\begin{equation}\label{eq.vanishing.derivs}
	\left.\frac{d^j}{dt^j}\Phi(e^{tM}z)\right|_{t=0} = 0, \quad \forall j = 1,\dots, 2k+1
\end{equation}
if and only if
\begin{equation}\label{eq.vanishing.Theta}
	\Theta(M^j z) = 0, \quad j = 0,\dots, k.
\end{equation}
This is obvious for $k = 0$ by \eqref{eq.all.derivs}, and so we proceed by an induction argument assuming that \eqref{eq.vanishing.derivs} and \eqref{eq.vanishing.Theta} are equivalent for $k$ and that either \eqref{eq.vanishing.derivs} or \eqref{eq.vanishing.Theta} holds for $k+1$, meaning that both \eqref{eq.vanishing.derivs} and \eqref{eq.vanishing.Theta} hold for $k$.

We rewrite \eqref{eq.all.derivs} in terms of $\Theta$, using that
\[
	\Theta\left(M^\ell z, M^{j-\ell-1}z\right) = \Phi\left(M^\ell z, M^{j-\ell}z\right) + \Phi\left(M^{\ell+1} z, M^{j-\ell-1}z\right).
\]
We see that
\begin{equation}\label{eq.derivs.Theta}
	\begin{aligned}
	\left.\frac{d^j}{dt^j}\Phi(e^{tM}z)\right|_{t=0} &= \sum_{j=0}^{j-1} a_j\Theta\left(M^\ell z,M^{j-\ell-1}z\right),
	\\ a_\ell &= \binom{j}{\ell} - a_{\ell-1} = \sum_{m=0}^\ell (-1)^{m-\ell}\binom{j}{m}.
	\end{aligned}
\end{equation}
For any $0 \leq \ell \leq k$ and $\zeta \in \Bbb{C}^n$, by the Cauchy-Schwarz inequality we have that
\[
	|\Theta(M^k z, \zeta)|^2 \leq \Theta(M^k z)\Theta(\zeta) = 0
\]
by our induction assumption which implies that \eqref{eq.vanishing.Theta} holds for $k$.  Therefore if $j = 2k+2$ or $j=2k+3$ the only term that survives in \eqref{eq.derivs.Theta} is when $j = 2k+3$ and $\ell = k+1$.  So
\[
	\left.\frac{d^{2k+2}}{dt^{2k+2}}\Phi(e^{tM}z)\right|_{t = 0} = 0,
\]
which also follows from the fact that $\Phi(e^{tM}z)$ is nondecreasing in $t$ by \eqref{eq.ellipticity}, and
\[
	\left.\frac{d^{2k+3}}{dt^{2k+3}}\Phi(e^{tM}z)\right|_{t = 0} = \left(\sum_{m=0}^{k+1} \binom{2k+3}{m}\right)\Theta(M^{k+1}z).
\]
By a standard combinatorial formula,
\[
	(-1)^{k+1}\sum_{m=0}^{k+1} (-1)^m\binom{2k+3}{m} = \binom{2k+2}{k+1}.
\]
Since this coefficient is nonzero, this suffices to prove that \eqref{eq.vanishing.derivs} and \eqref{eq.vanishing.Theta} are equivalent for all $k \in \Bbb{N}$.  What is more, this shows that the leading term in the Taylor expansion of $\Phi(e^{tM}z)-\Phi(z)$ is of order $t^{2I+1}$ and, through identifying the derivative, we have established \eqref{eq.Subelliptic.Escape}.

If $I = \infty$, then by bilinearity of $\Phi(z,\zeta)$ we see that
\[
	\Phi(e^{tM}z) = \sum_{j,k\in\Bbb{N}} \frac{t^{j+k}}{j! k!}\Phi(M^j z, M^k z).
\]
By the Cauchy-Schwarz inequality and the assumption that $I = \infty$, for any $(j,k) \neq (0,0)$ we have $\Phi(M^j z, M^k z) = 0$.  Equation \eqref{eq.Subelliptic.No.Escape} follows immediately, completing the proof of the theorem.
\end{proof}

We finish our analysis by using the rate of increase in Theorem \ref{thm.subelliptic.escape} to find the exact order, in $t$, of the small-time regularization properties $\exp(-tP)$, extending Theorem \ref{thm.boundedness.Theta}.  As we see later in Example \ref{ex.subell.not.sharp}, the bounds are of the correct order in $t$, but the value of the constant may not be given by the Taylor expansion established in Theorem \ref{thm.subelliptic.escape}.

\begin{theorem}\label{thm.subell.decay}
Let the matrix $M$, the weight $\Phi$, and the operators $P$ and $\exp(\tau P)$ be as in Proposition \ref{prop.solve.exptP}. Recall the definition \eqref{eq.def.delta.0} of $\delta_0$ and suppose furthermore that \eqref{eq.ellipticity} holds.

Let 
\[
	I_0 = \max\{I(z)\::\: |z|=1\}
\]
be the maximum of the $I(z)$ defined in \eqref{eq.def.subellipticity} for $|z|=1$.  Assume that $I_0 < \infty$ and let
\[
	k_1 = \frac{1}{(2I_0+1)!}\binom{2I_0}{I_0} \min\left\{\frac{\Theta(M^{I_0}z)}{|Gz|^2}\::\: |z| = 1,\ I(z) = I_0\right\}.
\]
Then, for $\delta_0$ defined in \eqref{eq.def.delta.0}, we have that there exists $C > 0$ for which
\begin{equation}\label{eq.subell.delta.0.bounds}
	 \frac{1}{C}t^{2I_0+1} \leq \delta_0(-t) \leq \frac{k_1}{4^{I_0}}t^{2I_0+1} + \BigO(t^{2I_0+2}), \quad \forall 0 \leq t \ll 1.
\end{equation}
\end{theorem}

\begin{proof}
	For the upper bound, let $z_0 \in \Bbb{C}^n$ with $|z_0| = 1$ attain the minimum in the definition of $k_1$; the existence of such a $z_0$ follows from continuity of $\Theta(M^{I_0}z)/|Gz|^2$ on $\Bbb{S}^{2n-1} \subset \Bbb{C}^n$. We abbreviate the leading coefficient in Theorem \ref{thm.subelliptic.escape} as
	\[
		k_0(z) = \frac{1}{(2I+1)!}\binom{2I}{I}\Theta(M^I z),
	\]
	so $k_0(z_0) = k_1$.
	Then from Theorem \ref{thm.subelliptic.escape}, for $s$ sufficiently small we have
	\[
		\Phi(e^{sM}z_0) = \Phi(z_0) + k_1 s^{2I_0+1}+\BigO(s^{2N_0+2}).
	\]
	We then have
	\[
		\Phi(e^{tM/2}z_0) - \Phi(e^{-tM/2}z_0) = \frac{k_1}{2^{2I_0}} t^{2I_0+1} + \BigO(t^{2I_0+2}).
	\]
	Let $\tilde{z}_0 = e^{-tM/2}z_0$ and note both that $|\tilde{z}_0| = 1+\BigO(t)$ and that $G\tilde{z}_0 = Gz_0 + \BigO(t)$. We follow the proof of \eqref{eq.delta.0.deriv} in writing
	\[
		\Phi^{(\delta)}(e^{tM}\tilde{z}_0) - \Phi(\tilde{z}_0) = \frac{k_1}{2^{2I_0}} t^{2I_0+1} - \delta|Gz_0|^2(1+\BigO(t)) + \BigO(\delta^2+t^{2I_0+2})
	\]
	and noting that if $\delta = k_1t^{2I_0+1}+C_1t^{2I_0+2}$ for $C_1$ sufficiently large, then 
	\[
		\Phi^{(\delta)}(e^{tM}\tilde{z}_0) < \Phi(\tilde{z}_0), \quad 0 < t \ll 1.
	\]
	This proves the right-hand inequality in \eqref{eq.subell.delta.0.bounds}.
	
	For the upper bound, we follow the proof of \cite[Prop.~3.2]{Sj2010}.  Define
	\[
		f(t,z) = \Phi(e^{tM}z)-\Phi(z).
	\]
	If the left-hand inequality in \eqref{eq.subell.delta.0.bounds} does not hold, then we must be able to find some sequence $\{(t_k, z_k)\}_{k=1}^\infty$ in $(0, \infty) \times \{|z|=1\}$ converging to $(0, z_\infty)$ for which
	\[
		\lim_{k\to\infty} \frac{f(t_k,z_k)}{t_k^{2I_0+1}} = 0.
	\]
	By our assumption \eqref{eq.ellipticity}, we know that $f(t,z)$ is nondecreasing in $t$, so furthermore
	\begin{equation}\label{eq.subell.f.to.0}
		\lim_{k\to\infty} \sup_{0 \leq t \leq t_k}\frac{f(t,z_k)}{t_k^{2I_0+1}} = 0.
	\end{equation}
	
	Then write
	\[
		\tilde{f}_k(s) = \frac{f(t_ks, z_k)}{t_k^{2I_0+1}}, \quad s \in [0,1],
	\]
	which converges uniformly to zero on $[0,1]$ by \eqref{eq.subell.f.to.0}.  Since
	\[
		\tilde{f}_k(s) = \sum_{j=0}^{2I_0 + 1} \frac{1}{t_k^{2I_0+1-j}j!} \partial_t^j f(0,z_k) s^j + \BigO(t_ks^{2I_0+2}),
	\]
	we conclude that, for all $0 \leq j \leq 2I_0+1$,
	\[
		\partial_t^j f(0, z_\infty) = \lim_{k\to\infty} \partial_t^j f(0, z_k) = 0.
	\]
	By Theorem \ref{thm.subelliptic.escape}, this violates the assumption that $I(z_\infty) \leq I_0 < \infty$.  This contradiction establishes the left-hand side of \eqref{eq.subell.delta.0.bounds} and completes the proof of the theorem.
\end{proof}

\subsection{The case where $h$ vanishes}\label{subsec.orthogonal}

With the weight function $\Phi$ decomposed as in \eqref{eq.decompose.Phi}, we focus on the case $h(z) = 0$. In particular, abandoning the assumption that $M$ is in Jordan normal form, we may assume after a change of variables that
\[
	\Phi(z) = \Psi(z) = \frac{1}{2}|z|^2.
\]
This assumption is convenient because it forces the $\Pi_N$ in \eqref{eq.def.Pi} to be orthogonal projections, and it is relevant because it is satisfied when treating operators like the Fokker-Planck quadratic model in Section \ref{subsubsec.FP}.

Using the tools already introduced, we can see that this assumption allows us to exactly determine the norm of the solution operator $\exp(\tau P)$, its return to equilibrium, and its regularization properties; these are all closely related and are given by the norm of a matrix exponential.  Because, in this special case, we can obtain extremely precise information using only a standard Bargmann transform, we present these and other results with much shorter proofs in \cite{AlVi2014a}.

Necessary conditions for a quadratic operator on $L^2(\Bbb{R}^n)$ to admit a unitary equivalence like in Proposition \ref{prop.which.q} with $\Phi = \Psi$ are discussed in \cite[Thm.~1.4]{Vi2013}.  To avoid complications like in Example \ref{ex.nontrivial.trivial.h}, we assume that $\opnm{Spec} M$ is contained in a proper half-plane.

\begin{proposition}\label{prop.h.vanishes}
Let $q(x,\xi)$ satisfy the conditions in Proposition \ref{prop.which.q} and let $q^w(x,D_x)$ and $p^w(z,D_z) = Mz\cdot iD_z + \frac{1}{2}\opnm{Tr}M$ be related by \eqref{eq.q.p.unitary}.  Let $P = Mz\cdot \partial_z$ act on $H_\Phi$ for $\Phi$ satisfying \eqref{eq.Phi.assumptions}. Assume furthermore that there exists $\theta_0 \in \Bbb{R}$ for which
\[
	\opnm{Spec} M \subset \{\Re e^{i\theta_0}\lambda > 0\}.
\]

Then the following are equivalent:
\begin{enumerate}[(i)]
\item the harmonic part $\Re h(z)$ from $\Phi$ in \eqref{eq.decompose.Phi} is zero;
\item the ground state of $P$ and its adjoint agree, or $\ker P = \ker P^*$;
\item\label{it.orthog.conjugates} the manifolds $\Lambda^\pm$ of $q$ are complex conjugates, $\Lambda^+ = \overline{\Lambda^-}$;
\item\label{it.orthog.q1} conjugation by $\mathcal{U}_\chi$ as in \eqref{eq.q.p.unitary} reduces $q(x,\xi)$ to
\[
	q_1(x,\xi) = (q\circ\chi^{-1})(x,\xi) = \frac{1}{2}M(x-i\xi)\cdot(x+i\xi);
\]
\item the projection $\Pi_0$ is orthogonal; and
\item every projection $\Pi_N$ for $N \in \Bbb{N}$ is orthogonal.
\end{enumerate}
\end{proposition}

\begin{proof} Apart from (\ref{it.orthog.q1}), the equivalences follow from \cite[Thm.~1.4]{Vi2013}, and its proof, after identifying $\ker P$ with $\opnm{span} \{1\}$ using Theorem \ref{thm.eigen.core}.  That (\ref{it.orthog.q1}) implies (\ref{it.orthog.conjugates}) follows from the fact that $\Lambda^\pm(q_1) = \{(x,\pm ix)\}$ and that (\ref{it.orthog.conjugates}) is invariant under composition with real linear canonical transformations.  Finally, that the other conditions imply (\ref{it.orthog.q1}) is immediate from \eqref{eq.creation.Fock} and \eqref{eq.annihilation.Fock} with $G = 1$ and $H = 0$.
\end{proof}

\begin{theorem}\label{thm.orthog.delta.0}
Let the matrix $M$ and the operators $P$ and $\exp(\tau P)$, acting on $H_\Psi$ for $\Psi(z) = \frac{1}{2}|z|^2$, be as in Proposition \ref{prop.solve.exptP}.  Then, with $\delta_0(\tau)$ from \eqref{eq.def.delta.0},
\begin{equation}\label{eq.orthog.delta.0.norm}
	\delta_0(\tau) = -\log \|e^{\tau M}\|.
\end{equation}
In particular, $\exp(\tau P)$ is bounded if and only if $\|e^{\tau M}\| \leq 1$ and is compact if and only if $\|e^{\tau M}\| < 1$.
\end{theorem}

\begin{proof}
From the definition \eqref{eq.def.delta.0} of $\delta_0(\tau)$ and the observation that $\Psi^{(\delta)}(z) = \frac{e^{-2\delta}}{2}|z|^2$, we have that
\[
	\delta_0(\tau) = \sup\{\delta \in \Bbb{R} \::\: \forall z \in \Bbb{C}^n,~~ e^{-2\delta}|e^{-\tau M}z|^2 \geq |z|^2\}.
\]
The invertible change of variables $y = e^{\tau M}z$ and some elementary manipulations reveal that
\[
	\delta_0(\tau) = \sup\left\{\delta \in \Bbb{R} \::\: \forall y \in \Bbb{C}^n \backslash \{0\},~~ \delta \leq -\log\frac{|e^{\tau M}y|}{|y|} \right\},
\]
from which \eqref{eq.orthog.delta.0.norm} follows.  The claim about boundedness and compactness follows from Theorem \ref{thm.boundedness.delta}.
\end{proof}

Before turning to exact formulas for return to equilibrium, we consider an example where there exists a gap between the bounds in Theorem \ref{thm.subell.decay}.

\begin{example}\label{ex.subell.not.sharp}
The right-hand bound in Theorem \ref{thm.subell.decay} is not generally sharp, even though Proposition \ref{prop.FP.return} shows that it happens to be true for the Fokker-Planck quadratic model from Section \ref{subsubsec.FP}.  That is to say, the slowest decay for $\Phi(e^{-tM}z)-\Phi(z)$ when $\Theta \geq 0$ may not always come from the worst Taylor expansion indicated by Theorem \ref{thm.subelliptic.escape}.

Let
\[
	M = \left(\begin{array}{ccc} 0 & -b & 0 \\ b & 0 & -a \\ 0 & a & 1\end{array}\right), \quad a,b \in \Bbb{R},
\]
and consider $P = Mz\cdot \partial_z$ acting on $H_\Psi$ with $\Psi(z) = \frac{1}{2}|z|^2$. Note from Theorem \ref{thm.orthog.delta.0} that, as $\delta_0(-t) \to 0$,
\[
	\|e^{-tM}\| = 1 - \delta_0(-t) + \BigO(\delta_0(-t)^2).
\]
We therefore study asymptotics of $\|e^{-tM}\|$ to compare with the bounds in \eqref{eq.subell.delta.0.bounds}.

In the language of Theorem \ref{thm.subell.decay}, $I_0 = 2$ is attained at $(1,0,0)$ for which 
\[
	\Theta(M^2 (1,0,0)) = a^2b^2.
\]
Then the upper bound for $\|e^{-tM}\|-1$ from Theorem \ref{thm.subell.decay} is
\[
	\begin{aligned}
	-\frac{k_1}{4^{I_0}}t^5 + \BigO(t^6) &= \frac{1}{4^{I_0}(2I_0+1)!} \binom{2I_0}{I_0}\Theta(M^2(1,0,0))t^5+\BigO(t^6)
	\\ &= -\frac{a^2 b^2}{320} t^5 + \BigO(t^6)
	\end{aligned}
\]

On the other hand, an optimization argument similar to the argument in the proof of Proposition \ref{prop.FP.return} leads us to the vector
\[
	v = \left(1, \frac{1}{2}bt, \frac{1}{12}abt^3\right),
\]
for which
\[
	|e^{-tM}v|^2 - |v|^2 = -\frac{1}{360}a^2b^2 t^5 + \BigO(t^6).
\]
Dividing (harmlessly) by $|v| = 1+\BigO(t)$ and taking the square root, which halves the coefficient of $t^5$, gives
\[
	\|e^{-tM}\| = 1-\frac{a^2 b^2}{720}t^5 + \BigO(t^6).
\]
Therefore, while the optimal power of $t$ in $\|e^{-tM}\|-1$ is $2I_0+1 = 5$ from Theorem \ref{thm.subell.decay}, the coefficient of $t^5$ does not necessarily come from a curve passing through a point $z$ where $I(z) = I_0$.
\end{example}

Finally, having shown in Theorem \ref{thm.orthog.delta.0} that $\delta_0(-t)$ and $\|e^{-tM}\|$ are closely related for the standard weight $\Phi = \Psi$,  we show that the same principle applies to return to equilibrium of any order.

\begin{theorem}\label{thm.orthog.return}
Let the matrix $M$ and the operators $P$ and $\exp(\tau P)$, acting on $H_\Psi$ for $\Psi(z) = \frac{1}{2}|z|^2$, be as in Proposition \ref{prop.solve.exptP}. Then, recalling definition \ref{eq.def.Pi} of $\Pi_N$, if $\|e^{\tau M}\| \leq 1$ then for any $N \in \Bbb{N}$ we have
\begin{equation}\label{eq.orthog.return}
	\|\exp(\tau P)(1-\Pi_N)\| = \|e^{\tau M}\|^{N+1}.
\end{equation}
and
\begin{equation}\label{eq.orthog.norm.1}
	\|\exp(\tau P)\| = 1.
\end{equation}
\end{theorem}

\begin{proof}
For $\tau$ fixed, let $U_1, U_2$ be unitary matrices such that
\[
	U_1 e^{\tau M} U_2^* = \Sigma
\]
where $\Sigma$ is a diagonal matrix with entries $\{\sigma_j\}_{j=1}^n$ equal to the singular values of $e^{\tau M}$.  Note that, for $U$ a unitary matrix, the change of variables $\mathcal{V}_U$ from \eqref{eq.def.V} takes $H_\Psi$ to $H_\Psi$.  Therefore
\[
	\mathcal{V}_{U_2}^*\exp(\tau P)\mathcal{V}_{U_1}u(z) = u(\Sigma z)
\]
acting on $H_\Psi$.

This operator is of the form $\exp(\mathcal{Q}_{\log\Sigma})$ as in Proposition \ref{prop.h.o.semigroup}, which also gives that this operator is self-adjoint.  We recall that $\Pi_N$ is orthogonal, so
\[
	\|\exp(\tau P)(1-\Pi_N)\| = \|\exp(\tau P)|_{\mathcal{M}_{N+1}}\|
\]
for $\mathcal{M}_{N+1}$ from \eqref{eq.def.MN}.  Since the changes of variables $\mathcal{V}_{U_1}$ and $\mathcal{V}_{U_2}$ preserve the spaces $\mathcal{M}_{N+1}$, we deduce that
\[
	\|\exp(\tau P)(1-\Pi_N)\| = \|\exp(\mathcal{Q}_{\log \Sigma})|_{\mathcal{M}_{N+1}}\|.
\]

By Theorem \ref{thm.eigen.core}, or simply checking on the orthonormal basis $\{f_\alpha\}$ from \eqref{eq.def.f.alpha}, we see that
\[
	\opnm{Spec} \exp(\mathcal{Q}_{\log \Sigma})|_{\mathcal{M}_{N+1}} = \left\{\prod_{j=1}^n \sigma_j^{\alpha_j} \::\: \alpha \in \Bbb{N}^n,~~ |\alpha| = N+1\right\}.
\]
This set is contained in $(0,1]$, since singular values are nonnegative, $e^{\tau M}$ is invertible, and the largest $\sigma_j$ is $\|e^{\tau M}\|$ which we assumed was at most $1$.  Therefore the largest eigenvalue of $\exp (\mathcal{Q}_{\log \Sigma})|_{\mathcal{M}_{N+1}}$ is $\|e^{\tau M}\|^{N+1}$.  Since $(\exp \mathcal{Q}_{\log \Sigma})|_{\mathcal{M}_{N+1}}$ is a positive definite self-adjoint operator, its largest eigenvalue is its norm, completing the proof \eqref{eq.orthog.return}.  Naturally, \eqref{eq.orthog.norm.1} follows upon omitting the projection $1-\Pi_N$.
\end{proof}

As mentioned in Remark \ref{rem.exact.norm}, we understand both the value of the norm and the return to equilibrium for our solution operators acting on $H_\Psi$.  We can therefore indirectly deduce the norms of embedding operators of the type considered in Proposition \ref{prop.embedding} between spaces $H_\Phi$ where the pluriharmonic part $-\Re h(z)$ vanishes; recall from \eqref{eq.decompose.Phi} that this means that $\Phi(z) = \frac{1}{2}|Gz|^2$ for some invertible matrix $G$.

\begin{corollary}
Let $G_1, G_2 \in GL_n(\Bbb{C})$ be invertible matrices, and let $\Psi(z) = \frac{1}{2}|z|^2$.  Then the embedding
\[
	\iota : H_{\Psi(G_1\cdot)} \ni u(z) \mapsto u(z) \in H_{\Psi(G_2\cdot)}
\]
is bounded if and only if $\|G_1G_2^{-1}\| \leq 1$ in which case
\[
	\|\iota\| = |\det G_1G_2^{-1}|.
\]
\end{corollary}

\begin{proof}
Let $U_1, U_2$ be unitary matrices such that
\[
	U_1 G_1 G_2^{-1} U_2^* = \Sigma
\]
for $\Sigma$ the diagonal matrix with entries the singular values of $G_1G_2^{-1}$.  Then, using the change of variables operators from \eqref{eq.def.V},
\[
	\mathcal{V}_{U_2}^* \mathcal{V}_{G_2}^* \iota \mathcal{V}_{G_1}\mathcal{V}_{U_1} u(z) = |\det G_1 G_2^{-1}|u(\Sigma z)
\]
is an operator on $H_\Psi$.  Since this operator is equal to $|\det G_1 G_2^{-1}|\exp(\mathcal{Q}_{\log \Sigma})$ as in Proposition \ref{prop.h.o.semigroup}, by Theorem \ref{thm.orthog.return} it is bounded if and only if
\[
	\|e^{\log \Sigma}\| = \|\Sigma\| = \|G_1G_2^{-1}\| \leq 1,
\]
in which case its norm is 
\[
	\|\iota\| = |\det G_2^{-1}G_1| \|\exp \mathcal{Q}_{\log \Sigma}\|_{\mathcal{L}(H_\Psi)} = |\det G_2^{-1}G_1|.
\]
\end{proof}

\appendix

\section{Equivalence of weak ellipticity conditions}\label{sec.appendix.subell}

\begin{proof}[Proof of Proposition \ref{prop.subell.relation}]
	Throughout, we regard the quadratic forms $q$ and $p$ as well as the canonical transformation $\mathcal{K}$ and the point $(x,\xi)$ as fixed.
	
	It is more convenient in what follows to allow complex variables and deal with the full matrix $F$ instead of its real and imaginary parts.  To begin, we show that
	\begin{equation}\label{eq.subell.equiv.J.sigma}
		J(x,\xi) = \min\{k \in \Bbb{N} \::\: \Re \sigma(F^k(x,\xi),F^{k+1}(x,\xi)) \neq 0\}.
	\end{equation}
	Note that this is a natural extension of $\Re q(F^k(x,\xi))$ except that $q$ is ordinarily viewed as a function on $\Bbb{R}^{2n}$.  We will see that replacing $(\Im F)^k$ with $F^k$ has no effect (beyond a sign change) so long as $k \leq J(x,\xi)$.

	Where $J(x,\xi) = 0$, the equality \eqref{eq.subell.equiv.J.sigma} follows from \eqref{eq.real.semidef} which implies that $(\Re q)^{-1}(\{0\}) = \ker \Re F$.  We proceed by showing by induction that, for any $k \in \Bbb{N}$,
	\begin{equation}\label{eq.subell.equiv.ind.matrix}
		(x,\xi) \in \ker \Re F(\Im F)^j, \quad j = 0,1, \dots, k
	\end{equation}
	if and only if
	\begin{equation}\label{eq.subell.equiv.ind.sigma}
		\Re \sigma(F^j(x,\xi),F^{j+1}(x,\xi)) = 0, \quad j = 0,1, \dots, k.
	\end{equation}
	Assume that \eqref{eq.subell.equiv.ind.matrix} and \eqref{eq.subell.equiv.ind.sigma} are equivalent for some $k \geq 0$ fixed and that \eqref{eq.subell.equiv.ind.matrix} or \eqref{eq.subell.equiv.ind.sigma} is true for $k+1$; therefore both \eqref{eq.subell.equiv.ind.matrix} and \eqref{eq.subell.equiv.ind.sigma} are true for $k$. Expanding
	\[
		\Re \sigma((\Re F + i\Im F)^{k+1}(x,\xi), (\Re F + i\Im F)^{k+2}(x,\xi)),
	\]
	we see by \eqref{eq.subell.equiv.ind.matrix} that every term where $\Re F$ is applied to $(\Im F)^j(x,\xi)$, for some $0 \leq j \leq k$, vanishes. As a result,
	\[
		\begin{aligned}
		\sigma(F^{k+1}(x,\xi),F^{k+2}(x,\xi)) = & i^{2k+2}\sigma((\Im F)^{k+1}(x,\xi), \Re F(\Im F)^{k+1}(x,\xi)) 
			\\ & + i^{2k+3}\sigma((\Im F)^{k+1}(x,\xi), (\Im F)^{k+2}(x,\xi)).
		\end{aligned}
	\]
	Taking the real part, we see that whenever \eqref{eq.subell.equiv.ind.matrix} holds,
	\begin{equation}\label{eq.subell.value.1}
		\begin{aligned}
		\Re \sigma(F^{k+1}(x,\xi),& F^{k+2}(x,\xi)) 
		\\ &= (-1)^{k+1}\sigma((\Im F)^{k+1}(x,\xi), \Re F(\Im F)^{k+1}(x,\xi)) 
		\\ &= (-1)^{k+1}\Re q((\Im F)^{k+1}(x,\xi)),
		\end{aligned}
	\end{equation}
	a quantity which is zero if and only if $(\Im F)^{k+1}(x,\xi) \in \ker \Re F$.  This proves the equivalence of \eqref{eq.subell.equiv.ind.matrix} and \eqref{eq.subell.equiv.ind.sigma} and therefore proves \eqref{eq.subell.equiv.J.sigma}.

	The formulation in \eqref{eq.subell.equiv.ind.sigma} is convenient since it involves the real part of a function $\sigma(F^k \cdot, F^{k+1}\cdot)$ which changes simply when $q$ is composed with a real or complex linear canonical transformation. Recall that $\mathcal{K}$ is the (complex linear) canonical transformation such that $p = (q\circ \mathcal{K}^{-1})$ where $p(z,\zeta) = (Mz)\cdot (i\zeta)$. Recall also from \eqref{eq.F.transform.K} that $F(p) = \mathcal{K}F(q)\mathcal{K}^{-1}$ and the simple form of the fundamental matrix $F(p)$ in \eqref{eq.F.p}. We let $z$ be determined by the canonical transformation relation 
	\[
		\mathcal{K}(x,\xi) = (z,-2i\Phi'_z(z)),
	\]
	as in \eqref{eq.subell.reln.K}. A direct computation using the bilinear form \eqref{eq.def.Phi.symmetric} and the fact that $\mathcal{K}$ is canonical shows that
	\begin{equation}\label{eq.subell.value.2}
		\begin{aligned}
		\Re \sigma(F(q)^k(x,\xi), &F(q)^{k+1}(x,\xi)) 
		\\ &= \Re \sigma \left(\mathcal{K}^{-1}(\mathcal{K}F(q)\mathcal{K}^{-1})^k\mathcal{K}(x,\xi), \mathcal{K}^{-1}(\mathcal{K}F(q)\mathcal{K}^{-1})^{k+1}\mathcal{K}(x,\xi)\right)
		\\ &= \Re \sigma\left(F(p)^k(z,  -2i\Phi'_z(z)),F(p)^{k+1}(z, -2i\Phi'_z(z))\right)
		\\ &= \Re \left(2^{-2k+1} M^{2k+1}z\cdot \Phi'_z(z)\right)
		\\ &= 2^{-2k+1} \Phi(M^{2k+1}z,z).
		\end{aligned}
	\end{equation}

	At this point, with the same $z$ from \eqref{eq.subell.reln.K}, we have that
	\[
		J(x,\xi) = \min\{k \in \Bbb{N} \::\: \Phi(M^{2k+1}z,z) \neq 0\}.
	\]
	To complete the proof, we establish that
	\begin{equation}\label{eq.subell.equiv.Phi}
		\Phi(M^{2j+1}z,z) = 0,\quad j = 0,1,\dots k
	\end{equation}
	if and only if
	\begin{equation}\label{eq.subell.equiv.Theta}
		\Theta(M^j z) = \Phi(M^{j+1}z,M^j z) = 0, \quad j=0,1,\dots, k.
	\end{equation}
	Again this is obvious for $k=0$, so we proceed by assuming that \eqref{eq.subell.equiv.Phi} and \eqref{eq.subell.equiv.Theta} are equivalent for some $k \geq 0$ fixed and that \eqref{eq.subell.equiv.Phi} or \eqref{eq.subell.equiv.Theta} holds for $k + 1$, which implies that \eqref{eq.subell.equiv.Phi} and \eqref{eq.subell.equiv.Theta} hold for $k$.  We compute using \eqref{eq.Theta.symmetric} that
	\[
		\Phi(M^{2k+3}z,z)  = -\Phi(M^{2k+2}z, Mz) + \Theta(M^{2k+2}z,z),
	\]
	and continuing and using the symmetry of $\Phi(z,\zeta)$, we have
	\[
		2\Phi(M^{2k+3}z,z) = \sum_{j=0}^{2k+2} (-1)^{j}\Theta(M^{2k+2-j}z,M^jz).
	\]
	As in the proof of Theorem \ref{thm.subelliptic.escape}, the Cauchy-Schwarz inequality for the positive semidefinite form $\Theta(z,\zeta)$, along with the induction hypothesis, shows that all $\Theta$ terms vanish except for the middle one, $j = k+1$. Therefore
	\begin{equation}\label{eq.subell.value.3}
		\Phi(M^{2k+3}z,z) = \frac{(-1)^{k+1}}{2}\Theta(M^{k+1}z),
	\end{equation}
	and this completes the proof relating $J(x,\xi)$ to $I(z)$.  The relation \eqref{eq.subell.equiv.value} follows from \eqref{eq.subell.value.1}, \eqref{eq.subell.value.2}, and \eqref{eq.subell.value.3}.
\end{proof}

\section{Small-time asymptotics for the Fokker-Planck model}

Here, we compute the small-time asymptotics for the matrix exponential corresponding to a Fokker-Planck operator in Section \ref{subsubsec.FP}.

\begin{proposition}\label{prop.FP.return}
Let
\[
	M_{a,0} = \left(\begin{array}{cc} 0 & -a \\ a & 1\end{array}\right), \quad a \in \Bbb{R}.
\]
Then, as $t \to 0^+$,
\[
	\|e^{-tM_{a,0}}\| = 1 - \frac{a^2}{12}t^3 + \BigO(t^4).
\]
\end{proposition}

\begin{proof}
We write, with $v = (v_1, v_2)$,
\[
	\begin{aligned}
	|e^{tM_a}v|^2 &= \sum_{j,k=0}^\infty \frac{t^{j+k}}{j!k!} \langle M_a^j v, M_a^k v \rangle
	\\ &= |v|^2 + t(\langle M_a v, v\rangle + \langle v, M_a v\rangle)
	\\ &\quad +t^2\left(\frac{1}{2}\langle M_a^2 v, v\rangle + \langle M_a v, M_a v\rangle + \frac{1}{2}\langle v, M_a^2 v\rangle \right)
	\\ &\quad + t^3\left(\frac{1}{6} \langle M_a^3 v, v\rangle + \langle M_a^2 v, M_a v\rangle + \langle M_a v, M_a^2 v\rangle + \frac{1}{6}\langle v, M_a^3 v\rangle\right) + \BigO(t^4).
	\end{aligned}
\]
We re-arrange the inner products by putting all matrices on the left-hand sides of inner products, so for instance the coefficient of $t^3$ becomes
\begin{multline*}
	\left\langle \left(\frac{1}{6}(M_a^3 + (M_a^3)^*) + \frac{1}{2}(M_a^*M_a^2 + (M_a^2)^*M_a)\right)v, v\right\rangle \\ = \left\langle \left(\begin{array}{cc} 2a^2/3 & a \\ a & (4-2a^2)/3\end{array}\right)v,v\right\rangle.
\end{multline*}

We conclude that
\begin{equation}\label{eq.FP.return.v1.tv2}
	\begin{aligned}
	|e^{tM_a}v|^2 &= |v|^2 + 2t|v_2|^2 + 2t^2(a\Re(v_1\overline{v_2}) + |v_2|^2) 
	\\ &\quad + t^3\left(\frac{2a^2}{3}|v_1|^2 + 2a\Re(v_1\overline{v_2}) - \frac{2}{3}(a^2-2)|v_2|^2\right) + \BigO(t^4).
	\end{aligned}
\end{equation}

In order to optimize, note that the second term $2t|v_2|^2$ must be much smaller than $|v|^2$.  In fact, to have
\[
	|e^{-tM_a}v|^2 = |v|^2 + \BigO(t^3|v|^2),
\]
we need to have $v_2 = \BigO(tv_1)$.  Multiplying by a complex number with modulus one, we may assume that $v_1 = 1$, so under these assumptions
\[
	|e^{-tM_a}v|^2 = |v|^2 - 2t|v_2|^2 + 2t^2a \Re \overline{v_2} - \frac{2a^2}{3}t^3 + \BigO(t^4).
\]
We then observe that $-2t|v_2|^2 + 2t^2a\Re \overline{v_2}$ is maximized when $v_2 = at/2$.

We conclude that the optimal witness for small-time decay is
\[
	|e^{-tM_a}(1, at/2)|^2 = |(1, at/2)|^2 - \frac{a^2}{6}t^3 + \BigO(t^4).
\]
Dividing by the norm, which is harmless since $|(1,at/2)| = 1+\BigO(t)$, and taking a square root, using the Taylor expansion $\sqrt{1-x} = 1-x/2+\BigO(x^2)$, we obtain the conclusion of the proposition.
\end{proof}

\bibliographystyle{acm}
\bibliography{MicrolocalBibliography2}

\end{document}